\documentclass[3p,11pt,letterpaper, reqno]{amsart}
\usepackage{mathrsfs}
\usepackage{amsmath,amssymb}
\usepackage{bbm}
\usepackage{color} 
\usepackage{graphicx}
\usepackage[english]{babel}
\usepackage{algorithm}
\usepackage{algorithmic}
\usepackage{hyperref}
\usepackage{makros}

\begin{document}

\title[MLE{n}KF for spatio-temporal processes]
      {Multilevel ensemble Kalman filtering for spatio-temporal processes}

\author[A. Chernov]{Alexey Chernov$^{*}$}
\thanks{$^{*}$
Institute for Mathematics, Carl von Ossietzky University Oldenburg, Germany (alexey.chernov@uni-oldenburg.de)}
\author[H. Hoel]{H{\AA}kon Hoel$^\dagger$}
\thanks{$^\dagger$Chair of Mathematics for Uncertainty Quantification, RWTH Aachen University, Aachen, Germany (hoel@uq.rwth-aachen.de)}
\author[K. J. H. Law]
{Kody J. H. Law$^\ddagger$}
\thanks{$^\ddagger$
School of Mathematics, 
University of Manchester, United Kingdom (kody.law@manchester.ac.uk)}

\author[F. Nobile]
{Fabio Nobile$^{\mathsection}$}
\thanks{$^{\mathsection}$
Mathematics Institute of Computational Science and Engineering, \'Ecole polytechnique f\'ed\'erale de Lausanne, 
Switzerland (fabio.nobile@epfl.ch)}
\author[R. Tempone]
{Raul Tempone$^\mathparagraph$}
\thanks{
  $^\mathparagraph$ Chair of Mathematics for Uncertainty Quantification, RWTH Aachen University, Aachen, Germany (tempone@uq.rwth-aachen.de)
  \and
Applied Mathematics and Computational Sciences, 
KAUST, Thuwal, Saudi Arabia.}



%

\begin{abstract}
  We design and analyse the performance of a multilevel ensemble
  Kalman filter method (MLEnKF) for filtering settings where the
  underlying state-space model is an infinite-dimensional
  spatio-temporal process. We consider underlying models that needs to
  be simulated by numerical methods, with discretization in both space
  and time. The multilevel Monte Carlo (MLMC) sampling strategy,
  achieving variance reduction through pairwise coupling of ensemble
  particles on neighboring resolutions, is used in the sample-moment
  step of MLEnKF to produce an efficient hierarchical filtering method
  for spatio-temporal models. Under sufficient regularity, MLEnKF is
  proven to be more efficient for weak approximations than EnKF,
  asymptotically in the large-ensemble and fine-numerical-resolution
  limit. Numerical examples support our theoretical findings.
   
   
\end{abstract}

\maketitle




\section{Introduction}
\label{sec:intro}


Filtering refers to the sequential estimation of the state $u$ and/or
parameters of a system through sequential incorporation of online data
$y$.  The most complete estimation of the state $u_n$ at time $n$ is
given by its probability distribution conditional on the observations
up to the given time $\bbP(du_n|y_1,\ldots, y_n)$ \cite{jaz70,
  BC09}. For linear Gaussian systems, the analytical solution may be
given in closed form via update formulae for the mean and covariance
known as the Kalman filter \cite{kalman1960new}.  More generally,
however, closed form solutions typically are not known.  One must
therefore resort to either algorithms which approximate the
probabilistic solution by leveraging ideas from control theory in the
data assimilation community \cite{kal03,jaz70}, or Monte Carlo methods
to approximate the filtering distribution itself \cite{BC09,
  doucet2000sequential, del2004feynman}.  The ensemble Kalman filter
(EnKF) \cite{burgers1998analysis, evensen2003ensemble, law2015data}
combines elements of both approaches.  In the linear Gaussian case it
converges to the Kalman filter solution in the large-ensemble limit
\cite{mandel2011convergence}, and even in the nonlinear case, under
suitable assumptions it converges \cite{le2011large,
  law2014deterministic} to a limit which is optimal among those which
incorporate the data linearly and use a single update
iteration~\cite{law2014deterministic, luenberger1968optimization,
  pajonk2012deterministic}.
In the case of spatially extended models approximated on a numerical
grid, the state space itself may become very high-dimensional and even
the linear solves may become intractable, due to the cost of computing
the covariance matrix.  Therefore, one may be inclined to use the EnKF
filter even for linear Gaussian problems in which the solution is
computationally intractable despite being given in closed form by the
Kalman filter.


The Multilevel Monte Carlo method (MLMC) is a hierarchical and
variance-reduction based approximation method initially developed for
weak approximations of random fields and stochastic differential
equations~\cite{heinrich2001multilevel, Giles08, Giles14}.  Recently,
a number of works have emerged which extend the MLMC framework to the
setting of Monte Carlo algorithms designed for Bayesian inference.
Examples include Markov chain Monte Carlo
\cite{ketelsen2013hierarchical, hoang2013complexity}, sequential Monte
Carlo samplers \cite{beskos2015multilevel, jasra2016forward,
  del2016multilevel}, particle filters \cite{jasra2015multilevel,
  gregory2016multilevel}, and EnKF~\cite{ourmlenkf}. The filtering
papers thus far~\cite{jasra2015multilevel, gregory2016multilevel,
  ourmlenkf} consider only \hh{finite-dimensional} SDE forward models.
In this work, we develop a new multilevel ensemble Kalman filtering
method (MLEnKF) for the setting of infinite-dimensional state-space
models with evolution in continuous-time.  The method consists of a
hierarchy of pairwise coupled EnKF-like ensembles on different
finite-resolution levels of the underlying infinite-dimensional
evolution model that all depend on the same Kalman gain in the update
step.  The method presented in this work may be viewed as an extension
of the finite-dimensional-state-space MLEnKF method~\cite{ourmlenkf},
which only considered a hierarchy of time-discretization resolution
levels.

Under sufficient regularity, the large-ensemble limit of EnKF is equal
in distribution to the so-called mean-field EnKF (MFEnKF),
cf.~\cite{le2011large, law2014deterministic, kwiatkowski2014convergence}. In
nonlinear settings, however, MFEnKF is not equal in distribution to
the Bayes filter, which is the exact filter distribution. More
precisely, the error of EnKF approximating the Bayes filter may be
decomposed into a statistical error, due to the finite ensemble size,
and a Gaussian bias that is introduced by the Kalman-filter-like
update step in EnKF.
%
%
While the update-step bias error in EnKF is difficult both to quantify
and deal with, the statistical error can, in theory, be reduced to
arbitrary magnitude.  However, the high computational cost of
simulations in high-dimensional state space often imposes small
ensemble size as a practical constraint. By making use of hierarchical
variance-reduction techniques, the MLEnKF method developed in this
work is capable of obtaining a much smaller statistical error than
EnKF at the same fixed cost.
%

In addition to design an MLEnKF method for spatio-temporal processes,
we provide an asymptotic performance analysis of the method that is
applicable under sufficient regularity of the filtering problem and
$L^p$-strong convergence of the numerical method approximating the
underlying model dynamics. Sections~\ref{sec:numericalAnalysis}
and~\ref{sec:numex} are devoted to a detailed analysis and practical
implementation of MLEnKF applied to linear and semilinear stochastic
reaction-diffusion equations.  In particular, we describe how the
pairwise coupling of EnKF-like hierarchies should be implemented for
one specific numerical solver (the exponential-Euler method), and
provide numerical evidence for the efficiency gains of MLEnKF over
EnKF.



Since particle filters are known to often perform better than EnKF,
we also include a few remarks on how we believe such methods
would compare to MLEnKF in filtering settings with spatial processes.
Due to the poor scaling of particle ensemble size in high dimensions,
which can even be exponential~\cite{BLB08, BBL08}, particle filters
are typically not used for spatial processes, or even modestly
high-dimensional processes.  There has been some work in the past few
years which overcomes this issue either for particular examples
\cite{beskos2014stable} or by allowing for some bias \cite{bcj11,
  vl10, rebeschini2013can, weare2009particle}. But particle filters
cannot yet be considered practically applicable for general spatial
processes.  If there is a well-defined limit of the model as the
state-space dimension $d$ grows such that the effective dimension of
the target density with respect to the proposal remains finite or even
small, then useful particle filters can be
developed~\cite{kantas2014sequential,llopis2017particle}.  As noted
in~\cite{chatterjee2015sample}, the key criterion which needs to be
satisfied is that the proposal and the target are not mutually
singular in the limit. MLMC has been applied recently to particle
filters, in the context where the approximation arises due to time
discretization of a finite-dimensional SDE \cite{jasra2015multilevel,
  gregory2016multilevel}.  It is an interesting open problem to design
multilevel particle filters for spatial processes: Both the range of
applicability and the asymptotic performance of such a method versus
MLEnKF when applied to spatial processes are topics that remain to be
studied.

The rest of the paper is organized as follows.
Section~\ref{sec:kalman} introduces the filtering problem and
notation. The design of the MLEnKF method is presented in
section~\ref{sec:mlmf}.  Section~\ref{sec:theory} studies the weak
approximation of MFEnKF by MLEnKF, and shows that in this setting,
MLEnKF inherits almost the same favorable asymptotic ``cost-to-accuracy''
performance as standard MLMC applied to weak approximations of
stochastic spatio-temporal processes.  Section\ref{sec:numericalAnalysis}
presents a detailed analysis and description of the implementation of
MLEnKF for a family of stochastic reaction-diffusion models. Section~\ref{sec:numex}
provides numerical studies of filtering problems with
linear and semilinear stochastic reaction-diffusion models that
corroborate our theoretical findings. Conclusions and future
directions are presented in section \ref{sec:conclusion}, and
auxiliary theoretical results and technical proofs are
provided in Appendices~\ref{app:mzIneq},~\ref{app:proofLemmaConv} and~\ref{app:extras}.





\section{Set-up and single level algorithm}
\label{sec:kalman}


\subsection{General set-up}
\label{subsec:genFiltering}

Let $(\Omega,\cF, \bbP)$ be a complete probability space, where $\bbP$
is a probability measure on the measurable space $(\Omega, \cF)$.  Let
$\cH$ be a separable Hilbert space with inner product
$\langle \cdot,\cdot \rangle_{\cH}$ and norm
$\|\cdot \|_{\cH} = \sqrt{\langle \cdot,\cdot \rangle_{\cH}}$.  Let
$V$ denote a subspace of $\cH$ which is closed in the
topology induced by the norm
$\|\cdot \|_{\cHO} = \sqrt{\langle \cdot,\cdot \rangle_{\cHO}}$,
which is assumed to be a stronger norm than $\|\cdot \|_{\cH}$. For
an arbitrary separable Hilbert space $(\cK, \|\cdot \|_{\cK})$, we denote
the associated $L^p$-Bochner space by
\[
L^p(\Omega,\cK) = \{u : \Omega \rightarrow \cK \mid u \text{ is measurable
  and } \E{\|u\|^p_\cK}
<\infty \}, \quad \text{for } p \in [1, \infty),
\]
where $\| u \|_{L^p(\Omega,\cK)} = (\E{\|u\|^p_\cK})^{1/p}$, or the
shorthand $\| u \|_p$ whenever confusion is not possible.  For an
arbitrary pair of Hilbert spaces $\cK_1$ and $\cK_2$, the space of bounded
linear mappings from the former space into the latter is denoted by
\[
L(\cK_1, \cK_2) \coloneq \left\{H:\cK_1 \to \cK_2 \mid H \text{ is
  linear and } \|H\|_{L(\cK_1, \cK_2)} < \infty \right\},
\]
where
\begin{equation*}\label{eq:operatorNorm}
\|H\|_{L(\cK_1,\cK_2)} \coloneq \sup_{x \in \cK_1\setminus\{0\}} \frac{\|Hx\|_{\cK_2}}{\|x\|_{\cK_1}}.
\end{equation*}
In finite dimensions, $(\bbR^m, \langle \cdot, \cdot \rangle)$
represents the $m$-dimensional Euclidean vector space with norm
$| \cdot | \coloneq \sqrt{\langle \cdot, \cdot \rangle}$,
and for matrices $A \in L(\bbR^{m_1},\bbR^{m_2})$, $|A| := \|A\|_{L(\bbR^{m_1},\bbR^{m_2})}$.

\subsubsection{The filtering problem}

Given $u_0 \in \cap_{p\ge 2} L^p(\Omega,\cHO)$ and the mapping
$\Psi: L^p(\Omega,\cHO) \times \Omega \rightarrow L^p(\Omega,\cHO)$,
we consider the discrete-time dynamics 
\begin{equation}\label{eq:psiDefinition}
u_{n+1}(\omega) = \Psi(u_n,\omega), \quad \text{for} \quad n=0,1,\ldots,N-1.  
\end{equation}
and the sequence of observations
\begin{equation}
y_n(\omega) = H u_n(\omega) + \eta_n(\omega),  \quad n = 1,2, \ldots, N.
\label{eq:obdef}
\end{equation}
Here, $H \in L(\cH, \bbR^{m})$, the sequence $\{\eta_n\}$ consists of
independent and identically $N(0,\Gamma)-$distributed random variables
with $\Gamma \in \bbR^{m\times m}$ positive definite.  In the sequel,
the explicit dependence on $\omega$ will be suppressed where confusion
is not possible.  A general filtering objective is to track the signal
$u_n$ given a fixed sequence of observations
$Y_n := (y_1,y_2, \ldots, y_n)$, i.e., to track the
distribution of $u_n|Y_n$ for $n=1,\ldots$.  In this work, however, we
restrict ourselves to considering the more specific objective of approximating
$\E{\varphi(u_n)|Y_n}$ for a given quantity of interest (QoI)
$\varphi: \cH \to \bbR$. 
The index $n$ will be referred to as time, 
whether the actual time between observations is 1 or not 
(in the examples in Section\ref{sec:numericalAnalysis} and beyond it will be called $T$), 
but this will not cause confusion since time is relative.

\subsubsection{The dynamics}

We consider problems in which $\Psi$ is the finite-time evolution of
an SPDE, e.g.~\eqref{eq:she}, and we will assume that $\Psi$ cannot be
evaluated exactly, but that there exists a sequence 
$\{\psiL:L^p(\Omega,\cH)\times \Omega \to L^p(\Omega,\cH)
\}_{\ell=0}^\infty$ of approximations to the solution
$\Psi:=\Psi^{\infty}$ satisfying the following uniform-in-$\ell$ stability
properties
\begin{assumption}  For every $p \geq 2$, it holds that
  $\Psi: L^p(\Omega,V)\times \Omega \to L^p(\Omega,V)$,
  and for all $u,v \in L^p(\Omega, \cH)$,
  the solution operators $\{\psiL\}_{\ell=0}^\infty$ satisfy the
 following conditions: there exists 
 a constant $0<c_\Psi<\infty$ depending on $p$ such that 
\begin{itemize} 
\item[(i)] $\|\psiL(u) -\psiL(v) \|_{L^p(\Omega,\cH)} {\le} c_{\Psi}
  \|u-v\|_{L^p(\Omega,\cH)} $ , and  
\item[(ii)] $\|\psiL(u)\|_{L^p(\Omega,\cH)} \leq c_{\Psi} (1+\|u\|_{L^p(\Omega,\cH)})$.
\end{itemize}
\label{ass:psilip}
 \end{assumption}

For notational simplicity, we restrict ourselves to settings in which
the map $\Psi(\cdot)$ does not depend on $n$, but the results in this
work do of course extend easily to non-autonomous settings when
the assumptions on $\{\Psi_n\}_{n=1}^N$ are uniform with respect to
$n$.

\begin{remark}
  The two approximation spaces $\cHO \subset \cH$ are introduced in
  order to obtain convergence rates for numerical simulation methods
  $\Psi^\ell$ that are discretized in physical or 
  state space. See Assumption~\ref{ass:mlrates}(i)-(ii) and
  inequality~\eqref{eq:spatialRate} for an example of how this may be
  obtained in practice.
\end{remark}

\subsubsection{The Bayes filter}
The pair of discrete-time stochastic processes $(u_n, y_n)$
constitutes a hidden Markov model, and the exact (Bayes-filter)
distribution of $u_n|Y_n$ may in theory be determined
iteratively through the system of prediction-update equations
\begin{align*}
\label{eq:filteringdist}
\bbP(du_n | Y_n) & = 
\frac1{Z(Y_n)}{\cL(u_n;y_n) \bbP(du_n |Y_{n-1})}, \\
\nonumber
\bbP(du_n |Y_{n-1}) & = \int_{u_{n-1}\in \cH} \bbP(du_n|u_{n-1})\bbP(du_{n-1}|Y_{n-1}), \\
\nonumber
\cL(u_n;y_n) & = \exp\Big\{-\frac{1}{2} |\Gamma^{-1/2}(y_n-Hu_n)|^2\Big\}, \\
Z(Y_n) & = \int_{u_{n}\in \cH} \cL(u_n;y_n) \bbP(du_n |Y_{n-1}).
\nonumber
\end{align*}
When the state space is infinite-dimensional and the dynamics cannot
be evaluated exactly, however, this is an extremely challenging
problem. Consequently, we will here restrict ourselves to constructing
weak approximation methods of the mean-field EnKF,
cf.~Section~\ref{sec:mfenkf}.


\subsection{Some details on Hilbert spaces, Hilbert-Schmidt operators, and Cameron-Martin spaces}
\label{ssec:dets}

For two arbitrary separable Hilbert spaces
$\cK_1 $ and $\cK_2$, the tensor product $\cK_1 \otimes \cK_2$ is
also a Hilbert space. For  rank-1 tensors, its inner product is defined by
\begin{equation*}
 \langle u \otimes v, u' \otimes v' \rangle_{\cK_1 \otimes \cK_2} = \langle u, u' \rangle_{\cK_1}\langle v, v'  \rangle_{\cK_2} \qquad \forall u, u' \in \cK_1, \quad \forall v,v' \in \cK_2,
\end{equation*}
which extends by linearity to any tensor of finite rank.  The
Hilbert space $\cK_1\otimes \cK_2$ is the completion of this set with
respect to the induced norm 
\begin{equation}\label{eq:hsdef}
\|u \otimes v\|_{\cK_1 \otimes \cK_2} = \|u\|_{\cK_1}\|v\|_{\cK_2}.
\end{equation}
Let $\{e_k\}$ and $\{\hat e_k\}$ be orthonormal bases for $\cK_1$ and 
$\cK_2$, respectively, and observe that finite sums of rank-1 tensors
of the form $X \coloneq \sum_{i,j}
\alpha_{ij} e_i \otimes \hat e_j \in \cK_1 \otimes \cK_2$ 
can be identified with a bounded linear mapping 
\begin{equation}\label{Tuv-def}
 T_{X}: \cK_2^* \to \cK_1 \quad \text{with} \quad T_{X}(f)
 \coloneq \sum_{i,j} \alpha_{ij}  f(\hat e_j) e_i, \quad \text{for } f \in \cK_2^*.
\end{equation}
For two bounded linear operators $A,B: \cK_2^* \to \cK_1$ we recall
the definition of the Hilbert-Schmidt inner product and norm
\begin{equation*}
 \langle A,B \rangle_{HS} = \sum_{k} \langle A \hat e^*_k, B \hat e^*_k \rangle_{\cK_1},
\qquad 
|A|_{HS} = \langle A,A \rangle_{HS}^{1/2},
\end{equation*}
where $\{\hat e^*_k\}$ is the orthonormal basis of
$\cK_2^*$ satisfying $\hat e^*_k(\hat e_j) = \delta_{jk}$ for all $j,k$ in the
considered index set.
 A bounded linear operator $A:\cK_2^* \to \cK_1$ is called a
Hilbert-Schmidt operator if $|A|_{HS} < \infty$ and $HS(\cK_2^*,
\cK_1)$ is the space of all such operators. In view of \eqref{Tuv-def},
\[
\begin{split}
 |T_{X}|_{HS}^2 &= \sum_{k} \left\langle \sum_{i,j} \alpha_{ij} e^*_k(\hat e_j) e_i, \sum_{i',j'} \alpha_{i'j'} e^*_k(\hat e_{j'}) e_{i'} \right\rangle_{\cK_1} \\
&= \sum_{i,j} |\alpha_{ij}|^2 \\
  &= \|X\|_{\cK_1 \otimes \cK_2}.
\end{split}
\]
By completion, the space $\cK_1 \otimes \cK_2$ is
isometrically isomorphic to $HS(\cK_2^*,\cK_1)$
(and also to $HS(\cK_2,\cK_1)$ by the Riesz representation theorem).
 For an element
$A \in \cK_1\otimes \cK_2$ we identify the norms
\begin{equation*}
 \|A\|_{\cK_1 \otimes \cK_2} = |A|_{HS},
\end{equation*}
and such elements will interchangeably be considered either as members of $\cK_1\otimes \cK_2$
or of $HS(\cK_2^*,\cK_1)$. When viewed as $A \in HS(\cK_2^*,\cK_1)$,
the mapping $A: \cK_2^* \to \cK_1$ is defined by
\[
A f \coloneq \sum_{i,j} A_{ij} f( \hat e_j) e_i, \quad \text{for } f
\in \cK_2^*, 
\]
where $A_{ij} \coloneq \langle e_i, A \hat e_j^* \rangle_{\cK_1}$,
and when viewed as $A \in \cK_1\otimes \cK_2$, we use tensor-basis representation
\[
A = \sum_{i,j} A_{ij} e_i \otimes \hat e_j.
\]

The covariance operator for a pair of random variables $Z, X \in L^2(\Omega, \cHO)$
is denoted by
\begin{equation*}
\mathrm{Cov}[Z,X] \coloneq\E{ (Z-\E{Z}) \otimes (X-\E{X} ) } \in \cHO
\otimes \cHO,
\end{equation*}
and whenever $Z=X$,
we employ the shorthand $\mathrm{Cov}[Z] \coloneq \mathrm{Cov}[Z,Z]$.
For completeness and later reference, let us prove that said covariance belongs to $\cHO
\otimes \cHO$.


\begin{proposition}\label{prop:fincov}
 If $u\in L^2(\Omega,\cHO)$, then  $C:=\mathrm{Cov}[u] \in \cHO \otimes \cHO$.
\end{proposition}
\begin{proof}

  By Jensen's inequality,  
  \[
  \begin{split}
    \|C\|_{\cHO \otimes \cHO} &= \| \E{(u -\E{u})\otimes (u-\E{u})} \|_{\cHO \otimes \cHO}\\
    & \leq \E{\norm{(u-\E{u}) \otimes (u-\E{u})}_{\cHO \otimes \cHO}} \\
    & = \norm{u-\E{u}}_{L^2(\Omega, \cHO)}^2 \\
    &= \|u\|_{L^2(\Omega, \cHO)}^2 - \norm{\E{u}}_{\cHO}^2 < \infty.
    \end{split}
  \]

      
\end{proof}

\subsection{Ensemble Kalman filtering}
\label{subsec:enkf}
EnKF is an ensemble-based extension of Kalman filtering to nonlinear
settings. Let $\{ \hv_{0,i}\}_{i=1}^M$ denote an ensemble of $M$
i.i.d.~particles with $\hv_{0,i} \stackrel{D}{=} u_0$. The initial
distribution $\bbP_{u_0}$ can thus be approximated by the empirical measure
of $\{ \hv_{0,i}\}_{i=1}^M$. By extension, let
$\{ \hv_{n,i}\}_{i=1}^M$ denote the ensemble-based approximation of
the updated distribution $u_n|Y_n$ (at $n=0$ we employ the convention
$Y_0=\emptyset$, so that $u_0=u_0|Y_0$).
Given an updated ensemble $\{ \hv_{n,i}\}_{i=1}^M$, the ensemble-based approximation
of the prediction distribution
$u_{n+1}|Y_n$ is obtained through simulating 
each particle one observation time ahead:
\begin{equation}\label{eq:EnKF}
  v_{n+1,i} = \Psi(\hv_{n,i}), \quad i =1,2,\ldots, M.
\end{equation}
We will refer to $\{\hv_{n+1,i}\}_{i=1}^M$ as the
prediction ensemble at time $n+1$, and we also note that
in many settings, the exact dynamics $\Psi$ in~\eqref{eq:EnKF}
have to be approximated by a numerical solver.

Next, given $\{\hv_{n+1,i}\}_{i=1}^M$ and a new observation $y_{n+1}$,
the ensemble-based approximation of the updated distribution
$u_{n+1}|Y_{n+1}$ is obtained through updating each particle path
 \begin{equation}\label{eq:enkfUp}
 \begin{split}
  \yTilde{n+1,i}	 & = y_{n+1} + \eta_{n+1,i},\\
  \hat{v}_{n+1,i} 
  & = (I - K^{\rm MC}_{n+1}H) v_{n+1,i} + K^{\rm MC}_{n+1} \yTilde{n+1,i}, 
\end{split}
\end{equation}
where $\{ \eta_{n+1,i} \}_{i=1}^M$ is an independent and identically $N(0, \Gamma)-$distributed
sequence, the Kalman gain 
\begin{equation*}
  K^{\rm MC}_{n+1}   = \left( C_{n+1}^{\rm MC} H^* \right) (S^{\rm MC}_{n+1})^{-1}
\label{eq:auxop1}
\end{equation*}
is a function of
\[
S^{\rm MC}_{n+1}            = HC_{n+1}^{\rm MC}  H^* + \Gamma,
\]
the adjoint observation operator $H^* \in L(\bbR^{m}, \cH^*)$,  defined by
\[
(H^*a)(w) = \langle a, Hw \rangle
\quad \text{for all} \quad  a \in \bbR^m \quad \text{and} \quad w \in \cH,
\]
and the prediction covariance  
\[
\begin{split}
  C_{n+1}^{\rm MC}   		& = \cov_M[v_{n+1}] \ ,
  \end{split}
\]
with
\begin{equation}\label{eq:sampleCov}
\cov_M[u_n,v_n] \coloneq \frac{M}{M-1}\parenthesis{E_M[u_n \otimes v_n] - E_M[u_n] \otimes E_M[v_n]}\,
\end{equation}
\begin{equation*}\label{eq:sampleAvg}
E_M[v_{n}] \coloneq \frac{1}{M} \sum_{i=1}^M  v_{n,i} \, 
\end{equation*}
and the shorthand $\cov_M[u_n] \coloneq \cov_M[u_n,u_n]$.

We introduce the following notation for the empirical measure of the
updated ensemble $\{\hv_{n,i}\}_{i=1}^{M}$:
\begin{equation}
\hat \mu^{\rm MC}_n = \frac{1}{M} \sum_{i=1}^{M} \delta_{\hv_{n,i}},
\label{eq:emp}
\end{equation}
and for any QoI 
$\varphi: \cH \rightarrow \bbR$, let 
\[
\hat \mu_n^{\rm MC}[\varphi] := \int \varphi d \hat\mu^{\rm MC}_n = 
 \frac{1}{M}\sum_{i=1}^{M} \varphi(\hv_{n,i}) .
\]

Due to the update formula~\eqref{eq:enkfUp}, all ensemble particles
are correlated to one another after the first update.  Even in the
linear Gaussian case, the ensemble will not remain Gaussian after the
first update.  Nonetheless, it has been shown that the in the
large-ensemble limit, EnKF converges in $L^p(\Omega)$ to the correct
(Bayes-filter) Gaussian in the linear and finite-dimensional case
\cite{mandel2011convergence, le2011large}, with the rate
$\cO(M^{-1/2})$ for Lipschitz-functional QoI with polynomial growth at
infinity.  Furthermore, in the nonlinear cases admitted by
Assumption~\ref{ass:psilip}, EnKF converges in the same sense and with
the same rate to a mean-field limiting distribution described below.

\begin{remark}
The perturbed observations $\tilde y_{n,i}$ were originally introduced
in \cite{burgers1998analysis} to correct the variance-deflation-type
error that appeared in its absence in implementations following the original
formulation of EnKF~\cite{evensen1994sequential}.  It has become known
as the perturbed observation implementation.
\end{remark}

\subsection{Mean-field Ensemble Kalman Filtering}
\label{sec:mfenkf}

In order to describe and study convergence properties of EnKF in the
large-ensemble limit, we now introduce the mean-field EnKF
(MFEnKF)~\cite{law2014deterministic}:
Let $\hmfv_0 \sim \bbP_{u_0}$
and
\begin{equation}
\qquad\qquad\quad\mbox{Predict}\;\left\{\begin{array}{lll}
\mfv_{n+1}& = \Psi (\hmfv_n), \\
\mfm_{n+1}&=\E{\mfv_{n+1}},\\
\mfc_{n+1}&=\E{({\mfv}_{n+1}-\mfm_{n+1})\otimes({\mfv}_{n+1}-\mfm_{n+1})}
 \end{array}\right.
\label{eq:mfpredict}
\end{equation}
\begin{equation}
\mbox{Update}\left\{\begin{array}{llll} \bar S_{n+1}&=H\mfc_{n+1}H^* +\Gamma \\
\mfk_{n+1}&=\mfc_{n+1}H^* \bar S_{n+1}^{-1}\\
{\tilde y}_{n+1}&=y_{n+1}+\eta_{n+1}\\
\hmfv_{n+1}&=(I- \mfk_{n+1}H){\mfv}_{n+1}+\mfk_{n+1}{\tilde y}_{n+1}.\\
\end{array}\right.
\label{eq:mfupdate}
\end{equation}
Here $\eta_n$ are i.i.d.~draws from $N(0,\Gamma).$ In the
finite-dimensional state-space setting, it was shown
in~\cite{le2011large} and~\cite{law2014deterministic} that for
nonlinear state-space models and nonlinear models with additive
Gaussian noise, respectively, EnKF converges to MFEnKF with the
$L^p(\Omega)$ convergence rate $\cO(M^{-1/2})$, as long as the models
satisfy a Lipschitz criterion, similar to (but stronger than)
Assumption~\ref{ass:psilip}.  And in~\cite{ourmlenkf}, we showed for
that MLEnKF converges toward MFEnKF with a higher rate than EnKF does
in said finite-dimensional setting. The
work~\cite{kwiatkowski2014convergence} extended convergence results to
\hh{infinite-dimensional} state space for square-root filters.  In
this work, the aim is to prove convergence of the MLEnKF for
infinite-dimensional state space, with the same favorable asymptotic
cost-to-accuracy \hh{performance} as in~\cite{ourmlenkf}.

The following $L^p$-boundedness properties ensures the existence of
the MFEnKF-process and its mean-field Kalman gain, and they will be needed
when studying the properteis of MLEnKF:

\begin{proposition}\label{prop:reg}
  Assume the initial data of the hidden Markov
  model~\eqref{eq:psiDefinition} and~\eqref{eq:obdef} satisfies
  $u_0 \in \cap_{p\ge 2} L^p(\Omega,\cHO)$.  Then the MFEnKF process
  \eqref{eq:mfpredict}--\eqref{eq:mfupdate} satisfies
  $\bar{v}_n, \hat{\bar{v}}_n \in \cap_{p \ge2} L^p(\Omega,\cHO)$ and
  $\|\mfk_n\|_{L(\bbR^{m},\cHO)} < \infty$ for all $n \in \bbN$.
\end{proposition}

\begin{proof}
Since $\hat{\bar{v}}_0 = u_0$, the property clearly holds for $n=0$.
Given $\hat{\bar{v}}_n\in L^p(\Omega,\cHO)$, Assumption \ref{ass:psilip}
guarantees $\bar{v}_{n+1} \in L^p(\Omega,\cHO)$.
By Proposition \ref{prop:fincov}, $\mfc_{n+1} \in \cHO\otimes \cHO$.
Since $H\mfc_{n+1}H^* \geq 0$ and $\Gamma>0$, it follows that
$H^*\bar S_{n+1}^{-1} \in L(\bbR^{m}, \cH^*)$ as
\[
\begin{split}
\|H^* \bar S_{n+1}^{-1}\|_{ L(\bbR^{m}, \cH^*)}
&\le \|H^*\|_{ L(\bbR^{m}, \cH^*)} |\bar S_{n+1}^{-1}|\\
&\le \|H\|_{L(\cH,\bbR^{m})} |\Gamma^{-1}| \\
&< \infty.
\end{split}
\]
Furthermore, since $\cH^* \subset \cHO^*$ it also holds that
$\|H^* \bar S_{n+1}^{-1}\|_{L(\bbR^{m},\cHO^*)} < \infty$ and 
\[
\begin{split}
\|\mfk_{n+1}\|_{L(\bbR^{m},\cHO)} &\le \|\mfc_{n+1}\|_{L(\cHO^*,\cHO)} \|H^* \bar S_{n+1}^{-1}\|_{L(\bbR^{m},\cHO^*)}\\
&\le \|\mfc_{n+1}\|_{\cHO \otimes \cHO} \|H^* \bar S_{n+1}^{-1}\|_{L(\bbR^{m},\cHO^*)}\\
&< \infty.
\end{split}
\]
The result follows by recalling that $\cHO \subset \cH$ and by the triangle inequality:
\begin{align*}
  \|\hat{\bar{v}}_{n+1}\|_{L^p(\Omega,\cHO)} &\le
  \|\bar{v}_{n+1}\|_{L^p(\Omega,\cHO)}\\
  & \quad + \|\mfk_{n+1}\|_{L(\bbR^{m},\cHO)}\prt{\|H \bar{v}_{n+1}\|_{L^{p}(\Omega,\bbR^{m})} +  \|\tilde y_{n+1}\|_{L^p(\Omega, \bbR^{m})}} \\
&< \infty.
\end{align*}
\end{proof}

We conclude this section with some remarks on tensorized
representations of the Kalman gain and related auxiliary operators
that will be useful when developing MLEnKF algorithms
in Section~\ref{subsec:mlenkfAlg}.
 
\subsubsection{The Kalman gain and auxiliary operators}
Introducing complete orthonormal bases $\{e_i\}_{i=1}^m$ for $\bbR^{m}$,
and $\{\phi_j\}$ for $\cH$, it follows that $H \in L(\cH, \bbR^{m})$ can
be written 
\[
H = \sum_{i=1}^m\sum_{j=1}^\infty H_{ij} e_i \otimes \phi_j^*
\]
with $H_{ij} \coloneq \langle e_i , H \phi_j\rangle$. And since 
$\|H\|_{L(\cH,\bbR^{m})} < \infty$,
it holds that
\[
\sum_{j=1}^\infty H_{ij} \phi_j^* \in \cH^*, \quad \text{for all } i \in \{1,2,\ldots,m\}.
\]
For the covariance matrix, it holds almost surely that
$\covHatMC{n+1} \in  \cHO \otimes \cHO \subset \cH \otimes \cH$,
so it may be represented by
\[
\covHatMC{n+1}= \sum_{i,j=1}^\infty C_{n+1,ij}^{{\rm MC}} \phi_i \otimes \phi_j,
\quad \text{where} \quad 
C_{n+1,ij}^{{\rm MC}} \coloneq \langle \phi_i,\covHatMC{n+1} \phi_j^*\rangle_{\cH}.
\]
For the auxiliary operator, it holds almost surely that
\[
R^{\rm MC}_{n+1} \coloneq C_{n+1}^{\rm MC} H^* \in L(\bbR^{m}, \cHO),
\]
so it can be represented by
\begin{equation*}\label{eq:arrr}
R^{\rm MC}_{n+1} = \sum_{i=1}^m \sum_{j=1}^\infty R^{{\rm MC}}_{n+1,ij} \phi_i \otimes e_j,
\quad \text{where} \quad 
R^{{\rm MC}}_{n+1,ij}= \sum_{k=1}^\infty C_{n+1,ik}^{{\rm MC}} H_{jk}.
\end{equation*}
Lastly, since $(S^{\rm MC})^{-1} \in L(\bbR^{m},\bbR^{m})$
and $K^{{\rm MC}} \in L(\bbR^{m},V)$ almost surely,
it holds that
\begin{equation*}
 S^{\rm MC}_{n+1,ij} = \left( \sum_{k=1}^\infty H_{ik} R^{{\rm MC}}_{n+1,kj}\right) + \Gamma_{ij} \text{ and }   
 K^{{\rm MC}}_{n+1,ij}   = \sum_{k=1}^m R^{{\rm MC}}_{n+1,ik} \left((S^{\rm MC}_{n+1})^{-1}\right)_{kj}.
\label{eq:auxop2}
\end{equation*}

\section{Multilevel EnKF}
\label{sec:mlmf}

\subsection{Notation and assumptions}
\label{subsec:mlenkf}
Recall that $\{\phi_k\}_{k=1}^{\infty}$ represents a complete orthonormal
basis for $\cH$ and consider the hierarchy of subspaces
$\cH_\ell = {\rm span} \{\phi_k\}_{k=1}^{N_\ell}$, {where $\{N_\ell\}$
  is an exponentially increasing sequence of natural numbers further
  described below in Assumption~\ref{ass:mlrates}}.  By construction,
$\cH_0 \subset \cH_1 \subset \dots \subset \cH$.
We define a sequence of orthogonal
projection operators $\{\cP_\ell: \cH \to \cH_\ell\}$ by
\begin{equation*}\label{eq:projDef}
  \cP_\ell v \coloneq \sum_{j=1}^{N_\ell} \langle \phi_j, v\rangle_{\cH} \phi_j \in \cH_\ell.
 \end{equation*}
 It trivially follows that $\cH_\ell$ is isometrically isomorphic to
 $\bbR^{N_\ell}$, so that any element $v^\ell \in \cH_\ell$ will, when
 convenient, be viewed as the unique corresponding element of
 $\bbR^{N_\ell}$ whose $k$-th component is given by
 $\langle \phi_k, v^\ell \rangle_{\cH}$ for
 $k \in \{1,2,\ldots,N_\ell\}$.
For the practical construction of numerical methods, we also 
introduce a second sequence of projection operators
$\{\Pi_\ell: \cH \to \cH_\ell\}$, e.g., interpolant operators,
which are assumed to be close to the corresponding orthogonal
projectors and to 
satisfy the constraint $\Pi_\ell \cH = \cP_\ell \cH = \cH_\ell$.
This framework can accommodate spectral methods, for which 
typically $\Pi_\ell = \cP_\ell$, as well as finite element type
approximations, for which $\Pi_\ell$ more commonly will be 
taken as an interpolant operator. In the latter case, the basis
$\{\phi_j\}$ will be a hierarchical finite element basis,
cf.~\cite{urban2012,brenner2008}.

We now introduce two additional assumptions on the hierarchy of
dynamics and two assumptions on the projection operators that will be
needed in order to prove the convergence of MLEnKF and its superior
efficiency compared to EnKF.  For two non-negative sequences
$\{f_\ell\}$ and $\{g_\ell\}$, the notation $f_\ell \lesssim g_\ell$
means there exist a constant $c>0$ such that $f_\ell \leq c
g_\ell$ holds for all $\ell \in \mathbb{N} \cup \{0\}$, and the
notation $f_\ell \eqsim g_\ell$ means that both $f_\ell \lesssim
g_\ell$ and $g_\ell \lesssim f_\ell$ are true.
\begin{assumption} 
  \label{ass:mlrates}
  Assume the initial data of the hidden Markov
  model~\eqref{eq:psiDefinition} and~\eqref{eq:obdef} satisfies $u_0
  \in \cap_{p\ge 2}L^p(\Omega,\cHO)$.  Consider a hierarchy of
  solution operators $\{\psiL: L^p(\Omega,\cH)\times \Omega \to
  L^p(\Omega,\cH_\ell) \}$ for which Assumption~\ref{ass:psilip} holds
  that are associated to a hierarchy of subspaces $\{\cH_\ell\}$ with
  resolution dimension $N_\ell \eqsim \kappa^\ell$ for some $\kappa
  >1$. Let $h_\ell \eqsim N_\ell^{-1/d}$ and $\Delta t_\ell \eqsim
  h_\ell^{\gamma_t}$, for some $\gamma_t>0$, respectively denote the
  spatial and the temporal resolution parameter on level $\ell$. For a
  given set of exponent rates $\beta, \gamma_x, \gamma_t>0$, the
  following conditions are fulfilled:
\begin{enumerate}

\item[(i)] $\| \psiL(u) - \Psi(u) \|_{L^p(\Omega, \cH)} \lesssim
    (1+\|u\|_{L^p(\Omega, \cHO)}) h_\ell^{\beta/2}$, \\ 
\text{ for all }$p\ge 2$ \text{ and } $u\in \cap_{p\ge 2}L^p(\Omega, \cHO)$,

\item[(ii)] for all $u \in \cHO$,
\[
\|(I- \cP_\ell) u \|_{\cH} \lesssim \|u\|_{\cHO} h_\ell^{\beta/2}
\quad \text{and} \quad 
\|(\Pi_\ell - \cP_\ell)u\|_{\cH} \lesssim \|u\|_{\cHO}  h^{\beta/2}_\ell,
\]
  
\item[(iii)] 
  the computational cost of applying $\Pi_\ell$ to any element of
  $\cH$ is $\cO(N_\ell)$ and that of applying
  $\psiL$ to any element of $\cH$ is 
  \[
  \text{Cost}(\psiL) \eqsim h_\ell^{-(d\gamma_x + \gamma_t)},
  \]
  where $d$ denotes the dimension of the spatial domain of elements in $\cH$,
  and $d\gamma_x + \gamma_t \ge d$.

 \end{enumerate}
\end{assumption}

\subsection{The MLEnKF method}

MLEnKF computes particle paths on a hierarchy of finite-dimensional function spaces with
accuracy levels determined by the solvers
$\{\psiL: L^p(\Omega,\cH)\times \Omega \to L^p(\Omega,\cH_\ell)
\}$.  
Let $v^{\ell}_{n}$ and $\hv^\ell_{n}$ respectively represent
prediction and updated ensemble state at time $n$ of a particle on
resolution level $\ell$, i.e., with dynamics governed by $\Psi^\ell$.
For an ensemble-size sequence $\{M_\ell\}_{\ell=0}^L \subset \bbN\setminus\{1\}$
that is further described in~\eqref{eq:chooseMlr}, the initial setup for MLEnKF
consists of a hierarchy of ensembles $\{\hv^{0}_{0,i}\}_{i=1}^{M_0}$
and
$\{(\hv^{\ell-1}_{0,i},
\hv^{\ell}_{0,i})_{i=1}^{M_\ell}\}_{\ell=1}^L$.  For $\ell=0$,
$\{\hv^{0}_{0,i}\}_{i=1}^{M_0}$ is a sequence of i.i.d. random
variables with $\hv^{0}_{0,i} \sim \bbP_{\Pi_0 u_0}$, and for
$\ell\ge1$,
$\{(\hv^{\ell}_{0,i}, \hv^{\ell-1}_{0,i})\}_{i=1}^{M_\ell}$ is a
sequence of i.i.d.~random variable 2-tuples with
$\hv^{\ell}_{0,i} \sim \bbP_{\Pi_\ell u_0}$ and pairwise coupling through
$\hv^{\ell-1}_{0,i} = \Pi_{\ell-1} \hv^{\ell}_{0,i}$.
MLEnKF approximates the
initial reference distribution $\bbP_{u_0|Y_0}$ (recalling the convention
$Y_0 = \emptyset$, so that $u_0|Y_0 = u_0$) by the multilevel-Monte-Carlo-based and signed empirical measure
\[
\hat \mu^{\rm ML}_0 = \frac{1}{M_0} \sum_{i=1}^{M_0} \delta_{
  \hv^{0}_{0,i}} + 
\sum_{\ell=1}^{L} \frac{1}{M_\ell} 
\sum_{i=1}^{M_\ell}( \delta_{\hv^{\ell}_{0,i}} -\delta_{\hv^{\ell-1}_{0,i}} ).
\]
Similar to EnKF, the mapping
\[
\{ (\hv^{0}_{n,i})_{i=1}^{M_0}, \{(\hv^{\ell-1}_{n,i}, \hv^{\ell}_{n,i})_{i=1}^{M_\ell}\}_{\ell=1}^L\}
\mapsto
\{ (\hv^{0}_{n+1,i})_{i=1}^{M_0}, \{(\hv^{\ell-1}_{n+1,i}, \hv^{\ell}_{n+1,i})_{i=1}^{M_\ell}\}_{\ell=1}^L\}
\]
represents the transition of the MLEnKF hierarchy of ensembles over
one prediction-update step and
\[
\hat \mu^{\rm ML}_0 = \frac{1}{M_0} \sum_{i=1}^{M_0} \delta_{
  \hv^{0}_{n,i}} + 
\sum_{\ell=1}^{L} \frac{1}{M_\ell} 
\sum_{i=1}^{M_\ell}( \delta_{\hv^{\ell}_{n,i}} -\delta_{\hv^{\ell-1}_{n,i}} ),
\]
represents the empirical distribution of the updated MLEnKF at time $n$. 
The MLEnKF prediction step consists of
simulating all particle paths on all resolution one observation-time
forward:
\begin{equation*}
  v^{0}_{n+1,i} = \Psi^{0}(\hv^{0}_{n,i},\omega_{0,i}), 
\end{equation*}
for $\ell =0$ and $i=1,2,\ldots, M_0$, and the pairwise coupling
\begin{equation}\label{eq:DlvHatDef}
\begin{split}
  v^{\ell-1}_{n+1,i}  	& = \Psi^{\ell-1}(\hv^{\ell-1}_{n,i},\omegaLI),\\
  v^{\ell}_{n+1,i} 	& = \Psi^\ell(\hv^{\ell}_{n,i}, \omegaLI),
  \end{split}
\end{equation}
for $\ell =1,\ldots,L$ and $i = 1,2, \ldots, M_\ell$.
Note here that the driving noise in the second
argument of the dynamics $\Psi^{\ell-1}$ and $\Psi^{\ell}$ is pairwise
coupled, and otherwise independent.
For the update step, 
the MLEnKF prediction covariance matrix is given
by the following multilevel sample-covariance estimator
\begin{equation}
\label{eq:mlSampleMoments}
\begin{split}
\cml_{n+1}  &= \sum_{\ell=0}^L \cov_{M_\ell}[ v^{\ell}_{n+1}]
- \cov_{M_\ell}[ v^{\ell-1}_{n+1}],
\end{split}
\end{equation}
and the multilevel Kalman gain is defined by
\begin{equation}
 \kml{n+1}  = \cml_{n+1} H^* 
 (S^{\rm ML}_{n+1})^{-1}, \text{ where } 
  S^{\rm ML}_{n+1}  := (H \cml_{n+1} H^*)^+    + \Gamma,
\label{eq:newkay}
\end{equation}
where
\begin{equation}\label{eq:covzee}
(H \cml_{n+1} H^*)^+ := \sum_{ \substack{ i=1\\ \lambda_i\geq 0} }^m \lambda_i q_i q_i^\transpose,
\end{equation}
with $(\lambda_j,q_j)_{j=1}^m$ denoting the eigenpairs of $H \cml_{n+1} H^* \in \bbR^{m\times m}$.
The new observation $y_{n+1}$ is assimilated into the hierarchy of ensembles
by the following multilevel extension of EnKF at the zeroth level:
\begin{equation*}\label{eq:upsamps0}
\begin{split}
  \yTildeL{0}{n+1,i}		& = y_{n+1} + \eta^0_{n+1,i} \\
  \hv^{0}_{n+1,i}  & =     (I - \Pi_0 \kml{n+1} H ) \vHatL{0}{n+1,i} + 
  \Pi_{0} \kml{n+1} \yTildeL{0}{n+1,i},  
  \end{split}
\end{equation*}
for $i=1,2,\ldots, M_0$ with $\{\eta^0_{n+1,i}\}_{i=1}^{M_0}$ ,
and for each of the higher levels, $\ell=1,\ldots,L$, the pairwise coupling of perturbed observations
\begin{equation}\label{eq:upsamps}
\begin{split}
  \yTildeL{\ell}{n+1,i}		& = y_{n+1} + \etaL_{n+1,i} \\
  \hv^{\ell-1}_{n+1,i}  & =     (I - \Pi_{\ell-1} \kml{n+1} H ) \vHatL{\ell-1}{n+1,i} + 
  \Pi_{\ell-1} \kml{n+1} \yTildeL{\ell}{n+1,i},  \\
  \hv^{\ell}_{n+1,i}  & =     (I - \Pi_\ell \kml{n+1} H ) \vHatL{\ell}{n+1,i} + 
  \Pi_{\ell} \kml{n+1} \yTildeL{\ell}{n+1,i},  
  \end{split}
\end{equation}
for $i=1,\ldots,M_\ell$, with the sequence
$\{\etaL_{n+1,i}\}_{i,\ell}$ being independent and identically
$N(0, \Gamma)-$distributed. It is precisely the multiplication with
the Kalman gain in the update step that correlates all the MLEnKF
particles. In comparison to standard MLMC where all samples
except the pairwise coupled ones are independent, the
this global correlation in MLEnKF substantially complicates
the convergence analysis of the method.
\begin{remark}
  Although unlikely, the multilevel sample prediction covariance
  $\cml_{n+1}$ may have negative eigenvalues and, worst case, this could
  lead to $S^{\rm ML}_{n+1}= H \cml_{n+1} H^* + \Gamma$ becoming a singular
  matrix. The impetus for replacing the matrix $(H \cml_{n+1} H^*)$
  with its positive semidefinite ``counterpart''
  $(H \cml_{n+1} H^*)^+$ in the Kalman gain formula~\eqref{eq:newkay}
  is to ensure that $S^{\rm ML}_{n+1}$ is invertible and to obtain the
  bound $|(S^{\rm ML}_{n+1})^{-1}| \le |\Gamma^{-1}|$.
\end{remark}

The following notation denotes the (signed) empirical measure of the multilevel
ensemble $\{(\hv^{\ell-1}_{n,i},\hv^{\ell}_{n,i})_{i=1}^{M_\ell}\}_{\ell=0}^L$:
\begin{equation}
\hat \mu^{\rm ML}_n = \frac{1}{M_0} \sum_{i=1}^{M_0} \delta_{
  \hv^{0}_{n,i}} + 
\sum_{\ell=1}^{L} \frac{1}{M_\ell} 
\sum_{i=1}^{M_\ell}( \delta_{\hv^{\ell}_{n,i}} -\delta_{\hv^{\ell-1}_{n,i}} ),
\label{eq:mlemp}
\end{equation}
and for any QoI $\varphi: \cH \rightarrow \bbR$, let 
\[
\hat \mu_n^{\rm ML}[\varphi] := \int \varphi d\hat \mu^{\rm ML}_n = 
\sum_{\ell=0}^{L}  \frac{1}{M_\ell}\sum_{i=1}^{M_\ell} 
\prt{\varphi(\hv^{\ell}_{n,i}) - \varphi(\hv^{\ell-1}_{n,i})}.
\]

We conclude this section with an estimate that relates to the
computational cost of one MLEnKF update step.
\begin{proposition}\label{prop:covcost}
  Given an MLEnKF hierarchy of prediction ensembles
  \[
    \{ (v^{0}_{n+1,i})_{i=1}^{M_0}, \{(v^{\ell-1}_{n+1,i}, v^{\ell}_{n+1,i})_{i=1}^{M_\ell}\}_{\ell=1}^L\},
    \]
  the cost of constructing the multilevel Kalman Gain $\kml{n+1}$ is proportional to
  $\sum_{\ell=0}^L m N_\ell M_\ell$. And if Assumption~\ref{ass:mlrates}(iii) holds,
  then the cost of updating the $\ell$-th level ensemble
  \[
    (v^{\ell-1}_{n+1,i}, v^{\ell}_{n+1,i})_{i=1}^{M_\ell} \mapsto
    (v^{\ell-1}_{n+1,i},\hv^{\ell}_{n+1,i})_{i=1}^{M_\ell}
  \]
  by~\eqref{eq:upsamps} is proportional to $m N_\ell M_\ell$.
\end{proposition}	
	
\begin{proof}
Notice that it is not required to compute the full multilevel prediction covariance $C^{\rm ML}_{n+1}$ 
in order to build the MLEnKF Kalman gain, but rather only 
\begin{equation}\label{eq:rMLExpression}
R^{\rm ML}_{n+1} := \cml_{n+1} H^* = \sum_{\ell=0}^L \Big( \cov_{M_\ell}[ v^{\ell}_{n+1}, 
H v^{\ell}_{n+1}]
- \cov_{M_\ell}[ v^{\ell-1}_{n+1}, H v^{\ell-1}_{n+1}] \Big).
\end{equation}
(The advantage of storing $R^{\rm ML}_{n+1} \in \bbR^{N_L \times m}$ rather than
$C^{\rm ML}_{n+1} \in \bbR^{N_L \times N_L}$ is the dimensional reduction obtained for large $L$,
since then $N_L \gg m$.)

For the Kalman gain, the cost of computing
$\cov_{M_\ell}[ v^{\ell}_{n+1}, H v^{\ell}_{n+1}] \in \bbR^{N_\ell
  \times m}$, is proportional to $m N_\ell M_\ell$.
There are also the insignificant one-time costs of
constructing and inverting $S^{\rm ML}_{n+1}$, and
the matrix multiplication $R^{\rm ML}_{n+1}(S^{\rm ML}_{n+1})^{-1}$.
In total, these costs are proportional to $N_Lm^2$.

The cost of updating the $\ell$-th level ensemble by~\eqref{eq:upsamps}
contains the one-time cost of the matrix multiplications $\Pi_\ell \kml{n+1}$
which by Assumption~\ref{ass:mlrates}(iii) is proportional to $mN_\ell$.
For each particle, the cost of computing $H \vHatL{\ell}{n+1,i}$ is
proportional to $N_\ell$, since $\vHatL{\ell}{n+1,i} \in \cH_\ell$, and
the cost of computing $(\Pi_\ell \kml{n+1})(H \vHatL{\ell}{n+1,i})$
and $(\Pi_\ell \kml{n+1})\yTildeL{\ell}{n+1,i}$ are both proportional
to $mN_\ell$.

\end{proof}

\subsection{MLEnKF algorithms}
\label{subsec:mlenkfAlg}

A subtlety with computing~\eqref{eq:rMLExpression} efficiently is that
the summands will be elements of different sized tensor spaces
since 
\[
\cov_{M_\ell}[ v^{\ell-1}_{n}, H v^{\ell-1}_{n} ]\in
\bbR^{N_{\ell-1} \times m}
\quad \text{while} \quad 
\cov_{M_\ell}[ v^{\ell}_{n}, H v^{\ell}_{n}] \in \bbR^{N_{\ell} \times m}
\]
for $\ell =1,2,\ldots,L$. The algorithm presented below efficiently
computes~\eqref{eq:rMLExpression} through performing all arithmetic
operations in the tensor space of lowest possible dimension available
at the current stage of computations. When
Proposition~\ref{prop:covcost} applies, the computational cost of the
algorithm is~\mbox{$\cO(m\sum_{\ell=0}^L M_\ell N_\ell)$}.  For ease of
exposition, we will in the sequel employ the convention
$v^{-1}_{n,i} = \hv^{-1}_{n,i} =0$ for all $n,i$.

\begin{algorithm}[H]
  \caption{Computing the auxiliary variable $R^{\rm ML}_{n}$}
  \begin{algorithmic}\label{alg:rAlgorithm}
  \STATE{{\bf Input:} Observation operator \mbox{$H \in L(\cV,\bbR^{m})$} and prediction ensemble $\{(v^{\ell-1}_{n,i},v^{\ell}_{n,i})_{i=1}^{M_\ell}\}_{\ell=0}^L$.}
  \STATE{{\bf Output:} $R^{\rm ML}_{n}$.}
  \STATE{Initialize $R^{\rm ML}_{n} = 0 \in \bbR^{N_L \times m}$.}
  \FOR{$\ell =0$ \TO $L-1$}
  \STATE{Update the submatrix $R^{\rm ML}_n(1:N_\ell, :) \in \bbR^{N_\ell \times m}$ consisting of
    the $N_\ell$ first rows and all columns of $R^{\rm ML}_n$ as follows:
  \[
    R^{\rm ML}_n(1:N_\ell, :) = R^{\rm ML}_n(1:N_\ell, :)
    + \cov_{M_\ell}[ v^{\ell}_{n}, H v^{\ell}_{n} ]- \cov_{M_{\ell+1}}[ v^{\ell}_{n}, H v^{\ell}_{n}].
  \]}
  \ENDFOR
  \STATE{Lastly, add finest level sample covariance:
    \[
    R^{\rm ML}_n = R^{\rm ML}_n + \cov_{M_L}[ v^{L}_{n}, H v^{L}_{n}].
    \]
    }
  \RETURN $R^{\rm ML}_n$.
  \end{algorithmic}
\end{algorithm}

In Algorithm~\ref{alg:mlenkf}, we summarize
the main steps for one predict-update iteration of the MLEnKF method.

\begin{algorithm}[H]
  \caption{MLEnKF predict-update iteration}
  \begin{algorithmic}\label{alg:mlenkf}
    \STATE{{\bf Input:} Hierarchy of projection operators $\{\Pi_\ell: \cH \to \cH_\ell\}$, observation $H \in L(\cV,\bbR^{m})$,
      observation noise covariance matrix $\Gamma$, and multilevel update ensemble 
      $\{(\hv^{\ell-1}_{n,i},\hv^{\ell}_{n,i})_{i=1}^{M_\ell}\}_{\ell=0}^L$.}
  \STATE{{\bf Output:} multilevel update ensemble 
      $\{(\hv^{\ell-1}_{n+1,i},\hv^{\ell}_{n+1,i})_{i=1}^{M_\ell}\}_{\ell=0}^L$.}
  \STATE{{\bf Predict:}}
  \FOR{$\ell =0$ \TO $L$}
  \FOR{$i =0$ \TO $M_\ell$}
  \STATE{Compute particle pair paths $(v^{\ell-1}_{n+1,i}, v^{\ell}_{n+1,i})$
    according to~\eqref{eq:DlvHatDef}.}
  \ENDFOR
  \ENDFOR
  \STATE{{\bf Update:}}
  \STATE{Compute $R^{\rm ML}_{n+1}$ by Algorithm~\ref{alg:rAlgorithm}, and $K^{\rm ML}_{n+1}$ by~\eqref{eq:newkay} (using that $R^{\rm ML}_{n+1} = \cml_{n+1} H^*$).}
  \FOR{$\ell =0$ \TO $L$}
  \FOR{$i =0$ \TO $M_\ell$}
  \STATE{Generate the perturbed observation $\tilde y^\ell_{n+1,i}$ and update the particle
    pair $(\hv^{\ell-1}_{n+1,i},\hv^{\ell}_{n+1,i})$ by~\eqref{eq:upsamps}.}
  \ENDFOR
  \ENDFOR
  \RETURN $\{(\hv^{\ell-1}_{n+1,i},\hv^{\ell}_{n+1,i})_{i=1}^{M_\ell}\}_{\ell=0}^L$.
  \end{algorithmic}
\end{algorithm}



\section{Theoretical Results}
\label{sec:theory}

In this section we derive theoretical results on the approximation
error and computational cost of weakly approximating the MFEnKF
filtering distribution by MLEnKF. We begin by stating the main
theorem of this paper. It gives an upper
bound for the computational cost of achieving a sought accuracy in
$L^p(\Omega)$-norm when using the MLEnKF method to approximate
the expectation of a QoI.  The theorem may be considered an
extension to spatially extended models of the earlier work
\cite{ourmlenkf}.  

\begin{theorem}[MLEnKF accuracy vs.~cost]\label{thm:main}
Consider a Lipschitz continuous QoI 
$\varphi:\cH \to \bbR$ and suppose
Assumption~\ref{ass:mlrates} holds. For a given
$\varepsilon >0$, let $L$ and $\{M_\ell\}_{\ell=0}^L$ be defined under the
constraints $L = \lceil 2\kl{d}\log_\kappa(\varepsilon^{-1})/\beta\rceil$ and
\begin{equation}\label{eq:chooseMlr}
\kl{M_\ell \eqsim  
 \begin{cases} 
   h_\ell^{(\beta+d\gamma_x+\gamma_t)/2}  h^{-\beta}_L, & \text{if} \quad \beta >
   d\gamma_x +\gamma_t, \\
   h_\ell^{(\beta+d\gamma_x +\gamma_t)/2}  L^2 h^{-\beta}_L, & \text{if} \quad  \beta = d\gamma_x +\gamma_t, \\
   h_\ell^{(\beta + d\gamma_x +\gamma_t)/2} h^{-(\beta + d\gamma_x +\gamma_t)/2}_L, &\text{if} \quad \beta< d\gamma_x +\gamma_t. 
 \end{cases}}
\end{equation}
Then, {for any $p\ge 2$ and $n \in \bbN$},
\begin{equation}
\|\hat \mu^{\rm ML}_n [\varphi] - \hat \mfmu_n [\varphi] \|_{L^p(\Omega)} \lesssim |\log(\varepsilon)|^n\varepsilon,
\label{eq:lperror}
\end{equation}
where we recall that $\hat \mu^{\rm ML}_n$ denotes the MLEnKF empirical measure~\eqref{eq:mlemp},
and $\hat \mfmu_n$ denotes the mean-field EnKF distribution at time $n$ (meaning
$\hmfv_{n} \sim \hat \mfmu_n$).

The computational cost of the MLEnKF estimator
\[
\cost{\hat \mu^{\rm ML}_n [\varphi]}:= \sum_{\ell=0}^L M_\ell \cost{\Psi^\ell}
\]
satisfies 
\kl{
\begin{equation*}\label{eq:mlenkfCosts2}
\cost{\hat \mu^{\rm ML}_n [\varphi]} \eqsim 
\begin{cases}
\varepsilon^{-2}, & \text{if} \quad \beta > d\gamma_x +\gamma_t,\\              
\varepsilon^{-2} \abs{\log(\varepsilon)}^3, & \text{if} \quad \beta = d\gamma_x +\gamma_t,\\
\varepsilon^{- 2(d\gamma_x +\gamma_t)/\beta}, & \text{if} \quad \beta < d\gamma_x+\gamma_t.
\end{cases}
\end{equation*}
}
\end{theorem}
The proof of this theorem is presented at the end of this section, and
it depends upon the intermediary results presented prior to the proof.
\begin{remark}
  The constraint $d\gamma_x +\gamma_t \ge d$ in
  Assumption~\ref{ass:mlrates}(iii) was imposed to ensure that the
  computational cost of the forward simulation,
  Cost$(\Psi^\ell) \eqsim h_\ell^{-(d\gamma_x +\gamma_t) }$, is either
  linear or superlinear in $N_\ell$.  In view of Proposition
  \ref{prop:covcost}, the share of the total cost of a single predict
  and update step assigned to level $\ell$ is proportional to
  $M_\ell \text{Cost}(\Psi^\ell)$. This cost estimate is used as input in
  the standard-MLMC-constrained-optimization approach to determining
  $M_\ell$, cf.~\eqref{eq:chooseMlr}. However, it is important to
  observe that in settings with high dimensional observations,
  $m \ge N_0$, the input in said optimization problem needs to
  be modified accordingly, as then the cost on the lower levels
  will be dominated by $m$ rather than Cost$(\Psi^\ell)$.
\end{remark}

The first result we present is a collection of direct consequences of
Assumption~\ref{ass:mlrates}:

\begin{proposition}\label{prop:asimp}
  If Assumption~\ref{ass:mlrates} holds, then for all
  $u,v \in \cap_{p \ge 2} L^p(\Omega, \cHO)$, and globally Lipschitz
  QoI $\varphi:\cH \to \mathbb{R}$,
\begin{itemize}
\item[(i)] 
$\| \psiL(v) - \Psi^{\ell-1}(v) \|_{L^p(\Omega, \cH)} 
\lesssim  (1+\|v\|_{L^p(\Omega, \cHO)}) h_\ell^{\beta/2}$,
for all $p\ge 2$, 
\item[(ii)] $\abs{ \E{\varphi(\psiL(u)) - \varphi(\Psi(v)) }} \lesssim \|u-v\|_{L^p(\Omega, \cH)} +  (1+\|u\|_{L^p(\Omega, \cHO)}) h_\ell^{\beta/2}$,
{ for all $p \ge 2$},
\item[(iii)] and for all $n\ge1$, the MFEnKF prediction covariance~\eqref{eq:mfpredict} satisfies
  \[
    \|(I- \Pi_\ell) \bar C_n \|_{\cH\otimes \cH} \lesssim
    \|\Psi(\hat{\bar v}_{n-1}) \|_{L^2(\Omega, \cHO)} h_\ell^{\beta/2}.
  \]
\end{itemize}
\end{proposition}

\begin{proof}
Property (i) follows from Assumption \ref{ass:mlrates}(i) and the triangle inequality.  
{Property (ii) follows from the Lipschitz continuity of $\varphi$
followed by the triangle inequality, Assumption \ref{ass:psilip}(i), and Assumption \ref{ass:mlrates}(i).}
For property (iii), Proposition \ref{prop:reg}, 
Jensen's inequality, definition~\eqref{eq:hsdef},
and H\"older's inequality implies that
\[
\begin{split}
\|(I-\Pi_\ell) \bar{C}_n\|_{\cH\otimes \cH} &= 
\|\E{(I-\Pi_\ell) (\bar{v}_n - \E{\bar{v}_n}) \otimes (\bar{v}_n -
  \E{\bar{v}_n})} \|_{\cH\otimes \cH}  \\
& \leq \|(I-\Pi_\ell) (\bar{v}_n - \E{\bar{v}_n})\|_{L^2(\Omega,\cH)} \| \bar{v}_n - \E{\bar{v}_n}\|_{L^2(\Omega,\cH)} \\
& \leq \|(I-\Pi_\ell) \bar{v}_n \|_{L^2(\Omega,\cH)} \| \bar{v}_n \|_{L^2(\Omega,\cH)}.
\end{split}
\] 
Since $(I-\cP_\ell) \bar{v}_n =(I-\cP_\ell)\Psi(\hat{\bar{v}}_{n-1})$, 
Assumption~\ref{ass:mlrates}(ii) implies that 
\[
\begin{split}
\|(I-\Pi_\ell)\Psi(\hat{\bar{v}}_{n-1}) \|_{L^2(\Omega,\cH)}
&\leq \|(I-\cP_\ell)\Psi(\hat{\bar{v}}_{n-1}) \|_{L^2(\Omega,\cH)}
+ \| (\Pi_\ell-\cP_\ell) \Psi(\hat{\bar{v}}_{n-1}) \|_{L^2(\Omega,\cH)}\\
& \leq 2 \|\Psi(\hat{\bar{v}}_{n-1}) \|_{L^2(\Omega, \cHO)}
h^{\beta/2}_\ell. 
\end{split}
\]

\end{proof}

Similar to the analysis in~\cite{ourmlenkf} and \cite{le2011large,
  law2014deterministic, mandel2011convergence},
we next introduce an auxiliary mean-field multilevel
ensemble
$\{(\mfv^{\ell-1}_{n,i},
\mfv^{\ell}_{n,i})_{i=1}^{M_\ell}\}_{\ell=0}^L$, where every particle
pair $(\mfv^{\ell-1}_{n,i},\mfv^{\ell}_{n,i})$ evolves by the
respective forward mappings $\Psi^{\ell-1}$ and $\psiL$ using the same
driving noise realization as the corresponding MLEnKF particle pair
$(v^{\ell-1}_{n,i},v^{\ell}_{n,i})$. Note however that in the update
of the mean-field multilevel ensemble, the limiting form MFEnKF
covariance $\mfc_n$ and Kalman gain $\mfk_n$ from
equations~\eqref{eq:mfpredict} and~\eqref{eq:mfupdate} are used rather
than corresponding ones based on sample moments of the multilevel
ensemble itself. That is, the initial condition for each 
is coupled particle pair is identical to that of MLEnKF:
\begin{equation}\label{eq:initialMFEnKF}
  (\hmfv^{\ell-1}_{0,i},\hmfv^{\ell}_{0,i}) = (\hv^{\ell-1}_{0,i},\hv^{\ell}_{0,i})
\end{equation}
and one prediction-update iteration is given by 
\begin{equation}
\mspace{-250mu}\mbox{Prediction}\;\left\{\begin{array}{ll}
 \mfv^{\ell-1}_{n+1,i}  	& = \Psi^{\ell-1}(\hmfv^{\ell-1}_{n,i},\omegaLI),\\
 \mfv^{\ell}_{n+1,i}      & = \Psi^\ell(\hmfv^{\ell}_{n,i}, \omegaLI),                                       
 \end{array}\right. 
\label{eq:mfmlpredict}
\end{equation}
\begin{equation}
\mspace{27mu} \mbox{Update}\left\{\begin{array}{lll} 
{\tilde y}_{n+1,i}^\ell& =y_{n+1}+\eta_{n+1,i}^\ell,\\
\hmfv_{n+1,i}^{\ell-1} &=(I- \Pi_{\ell-1}\mfk_{n+1}H){\mfv}_{n+1,i}^{\ell-1}+\Pi_{\ell-1}\mfk_{n+1}{\tilde y}_{n+1,i}^\ell,\\
\hmfv_{n+1,i}^{\ell} &=(I- \Pi_{\ell} \mfk_{n+1}H){\mfv}_{n+1,i}^\ell+\Pi_\ell\mfk_{n+1}{\tilde y}_{n+1,i}^\ell,
\end{array}\right.
\label{eq:mfmlupdate}
\end{equation}
for $\ell =0,1,\ldots, L$ and $i = 1,2,\ldots, M_\ell$ (similar to before, we
employ the convention $\mfv^{-1} = \hmfv^{-1} \coloneq 0$).
By similar reasoning as in Proposition~\ref{prop:reg},
it can be shown that also 
$\mfv^{\ell}_n$, $\hmfv^{\ell}_n \in  \cap_{p \ge 2} L^p(\Omega, \cHO)$
for any $\ell,n \in \bbN \cup \{0\}$.
One may think of the auxiliary mean-field multilevel ensemble
as ``shadowing'' the MLEnKF ensemble.

Before bounding the difference between the multilevel and mean-field
Kalman gains by the two next lemmas, let us recall that they respectively
are given by
\begin{equation}
K^{\rm ML}_n = C^{\rm ML}_n H^* 
( (H C^{\rm ML}_n H^*)^+ + \Gamma)^{-1} \quad \text{and} \quad
\mfk_{n} = \bar C_n H^* ( H \bar C_n H^* + \Gamma)^{-1}.
\nonumber
\end{equation}

\begin{lemma}
\label{lem:mlcae}
For the matrix $(H C^{\rm ML}_n H^*)^+:\bbR^{m\times m}$
defined by~\eqref{eq:covzee} with the spectral decomposition
eigenpairs $(\lambda_j, q_j)_{j=1}^m$ it holds that
\begin{eqnarray}
|(H C^{\rm ML}_n H^*)^+ -  H C^{\rm ML}_n H^* | \leq  \norm{H}^2_{L(\cH,\bbR^{m})}  \hhNorm{C^{\rm ML}_n - \mfc_n}.
\label{eq:mlcov_bound}
\end{eqnarray}
\end{lemma}

\begin{proof}
Since $(H C^{\rm ML}_n H^*)^+ -  H \bar C_n H^*$ 
is self-adjoint and positive semi-definite, 
\begin{equation*}\label{eq:gainErr}
\begin{split}
|(H C^{\rm ML}_n H^*)^+ -  H C^{\rm ML}_n H^*|
& =  \max_{\|q\|_{\bbR^{m}}=1} q^*\Big( (H C^{\rm ML}_n H^*)^+ -  H C^{\rm ML}_n H^* \Big) q\\
& =  \max(-\min_{j; \lambda_j < 0} \lambda_j, 0).
\end{split}
\end{equation*}
It remains to verify the lemma when $\{j \mid \lambda_j < 0\} \neq \emptyset$.
Let the normalized eigenvector associated to 
the eigenvalue $ \min_{\{j; \lambda_j< 0\}} \lambda_j$ 
be denoted $q_{\max}$. Then, since $(HR^{\rm ML}_n)^+ q_{\max}= 0$
and the mean-field covariance $\bar C_n$ 
is self-adjoint and positive semi-definite, 
\begin{equation*}\label{eq:gainErr}
  \begin{split}
    |(H C^{\rm ML}_n H^*)^+ -  H C^{\rm ML}_n H^*| &=   -q_{\max}^* H C^{\rm ML}_n H^* q_{\max}\\
& \leq  q_{\max}^* H\bar C_n H^* q_{\max} - q_{\max} H C_n^{\rm ML} H^*q_{\max}\\
& \leq |H(\bar C_n - C_n^{\rm ML})H^*|\\
& \leq \norm{H}^2_{L(\cH,\bbR^{m})} \norm{\bar C_n - C_n^{\rm ML}}_{L(\cH,\cH)}\\
& \leq \norm{H}^2_{L(\cH,\bbR^{m})} \hhNorm{\bar C_n - C_n^{\rm ML}}.
\end{split}
\end{equation*}
\end{proof}

The next step is to bound the Kalman gain error in terms of the
covariance error.
\begin{lemma}
\label{lem:gce}
There exists a positive constant $\tilde c_n<\infty$, depending on 
$\norm{H}_{L(\cH,\bbR^{m})}$, $|\Gamma^{-1}|$, and $\norm{\mfk_n }_{L(\bbR^{m},\cH)}$, such that 
\begin{equation*}
\norm{K^{\rm ML}_n - \mfk_n}_{L(\bbR^{m}, \cH)}
\leq \tilde c_n \hhNorm{{C}^{\rm ML}_{n}-\mfc_{n}}. 
\label{eq:gaincov}
\end{equation*}
\end{lemma}

\begin{proof} The proof of this lemma as is similar to that of~\cite[Lemma 3.4]{ourmlenkf}.
For completeness, we have included a proof in Appendix~\ref{app:extras}.
\end{proof}

The next lemma bounds the distance between the prediction covariance
matrices of MLEnKF and MFEnKF. For that purpose, let us first recall
that the dynamics for the mean-field multilevel ensemble
$\{(\mfv^{\ell-1}_{n,i},
\mfv^{\ell}_{n,i})_{i=1}^{M_\ell}\}_{\ell=0}^L$ is described in
equations~\eqref{eq:mfmlpredict} and~\eqref{eq:mfmlupdate}, and
introduce the auxiliary matrix 
\begin{equation}\label{eq:covML-def}
\mfc^{\rm ML}_n \coloneq \sum_{\ell=0}^L \cov_{M_\ell}[\mfv^\ell_n] - \cov_{M_\ell}[\mfv^{\ell-1}_n].
\end{equation}

\begin{lemma}
\label{lem:covspliteps}
For any $\varepsilon >0$, let $L$ and $\{M_\ell\}_{\ell=0}^L$ be defined as in Theorem~\ref{thm:main}.
If Assumption~\ref{ass:mlrates} holds, then for any $p\ge 2$ and $n \in \bbN$,
\begin{equation}
\|C^{\rm ML}_n - \mfc_n\|_{L^p(\Omega, \cH \otimes \cH)} \lesssim \varepsilon 
+ \|C^{\rm ML}_n - \mfc^{\rm ML}_n\|_{L^p(\Omega, \cH \otimes \cH)}.
\label{eq:covspliteps}
\end{equation}
\end{lemma}

\begin{proof}
Introducing the auxiliary covariance matrix 
\[
\mfc^L_n \coloneq \cov[\mfv^L_n]
\]
and using the triangle inequality,
\begin{equation*}
\|C^{\rm ML}_n - \mfc_n\|_{p} \leq \|\mfc^L_n - \mfc_n\|_{p} + 
\|\mfc^{\rm ML}_n - \mfc^L_n\|_{p} + \|C^{\rm ML}_n - \mfc^{\rm ML}_n \|_{p}.
\label{eq:covsplit}
\end{equation*}
The result follows by Lemmas~\ref{lem:disccov} and~\ref{lem:iidcover}.
\end{proof}

\begin{lemma}
\label{lem:disccov}
For any $\varepsilon>0$, let $L$ be defined as in Theorem~\ref{thm:main}.
If Assumption~\ref{ass:mlrates} holds, then for any $n\in \bbN$ and $p\ge 2$,
\begin{equation}\label{eq:disccov1}
\max\left(  \| \mfv_{n}^\ell -  \mfv_{n}\|_{L^p(\Omega, \cH)}, \| \hmfv_{n}^\ell -  \hmfv_{n}\|_{L^p(\Omega, \cH)}  \right) 
\lesssim h^{\beta/2}_\ell,
\end{equation}
\begin{equation}\label{eq:disccov2}
\max\left(  \| \mfv_{n}^\ell -  \mfv_{n}^{\ell-1}\|_{L^p(\Omega, \cH)}, \| \hmfv_{n}^\ell -  \hmfv_{n}^{\ell-1}\|_{L^p(\Omega, \cH)}  \right) 
\lesssim h_\ell^{\beta/2}, \quad \forall \ell \in \bbN,
\end{equation}
\begin{equation}\label{eq:disccov3}
\|\mfc^L_n- \mfc_n\|_{\cH \otimes \cH}
\lesssim \varepsilon.
\end{equation}
\end{lemma}

\begin{proof}
Recall that initial data of the limit mean-field methods is given by
$\hmfv_0\sim \hat \mu_0$ and that $\hmfv^\ell_0 =\Pi_\ell\hmfv_0$, so
that by Assumption~\ref{ass:mlrates}(ii),
\[
\|\hmfv_0 -  \hmfv^\ell_0 \|_{L^p(\Omega,\cH)} \lesssim \|\hmfv_0 \|_{L^p(\Omega,\cHO)}h_\ell^{\beta/2}.
\]
By Assumptions~\ref{ass:psilip}(i) and~\ref{ass:mlrates}(i),
\[
\|\mfv_n -  \mfv^\ell_n \|_{L^p(\Omega,\cH)} \lesssim  \|\hmfv_{n-1} -
\hmfv^\ell_{n-1} \|_{L^p(\Omega,\cH)} 
+ (1+\|\hmfv_{n-1}\|_{L^p(\Omega,\cHO)} )h_\ell^{\beta/2},
\]
and by Proposition \ref{prop:asimp}(iii),
\[
\begin{split}
 \|\hmfv_n -  \hmfv^\ell_n \|_{{L^p(\Omega,\cH)}}  
&\leq  \norm{I-\mfk_{n}H}_{L(\cH,\cH)} \|\mfv_{n}^\ell -  \mfv_{n}\|_{{L^p(\Omega,\cH)} }
+ \|(I-\Pi_\ell)\mfk_{n}(H\mfv_{n}^\ell + \tilde y_n) \|_{L^p(\Omega,\cH)}  \\
& \leq c \left( \|\mfv_{n}^\ell -  \mfv_{n}\|_{L^p(\Omega,\cH)}  + \|(I-\Pi_\ell)\bar C_{n}\|_{\cH \times \cH} \right) \\
& \lesssim \|\mfv_{n}^\ell -  \mfv_{n}\|_{L^p(\Omega,\cH)}  +
\|\Psi(\hmfv_{n-1})\|_{L^2(\Omega, \cHO)} h_\ell^{\beta/2}.
\end{split}
\]
Inequality~\eqref{eq:disccov1} consequently holds by induction,
and thus also~\eqref{eq:disccov2} by the triangle inequality.
To prove inequality~\eqref{eq:disccov3},
\[
\begin{split}
 &\|\mfc^L_n- \mfc_n\|_{\cH \otimes \cH}\\
 &= \hhNorm{\E{ (\mfv^L_n- \E{\mfv^L_n}) \otimes (\mfv^L_n- \E{\mfv^L_n}) 
      - (\mfv_n- \E{\mfv_n}) \otimes (\mfv_n- \E{\mfv_n})} }\\
  &= \hhNorm{\E{ (\mfv^L_n- \E{\mfv^L_n}) \otimes (\mfv^L_n- \E{\mfv_n}) 
 - (\mfv_n- \E{\mfv_n^L}) \otimes (\mfv_n- \E{\mfv_n})} }\\
&\leq  \norm{ (\mfv^L_n- \E{\mfv^L_n}) \otimes (\mfv^L_n- \E{\mfv_n}) 
 - (\mfv_n- \E{\mfv_n^L}) \otimes (\mfv_n- \E{\mfv_n}) }_{L^1(\Omega, \cH \otimes \cH)} \\
& \leq   (\| \mfv^L_n- \E{\mfv^L_n} \|_2 +  \| (\mfv_n- \E{\mfv_n}) \|_2 ) \| \mfv^L_n - \mfv_n\|_2\\
& \lesssim \varepsilon. 
\end{split}
\]

\end{proof}

We complete the proof of Lemma~\ref{lem:covspliteps} by deriving the following
bound for $\|\mfc^{\rm ML}_n - \mfc^L_n\|_p$~:

\begin{lemma}[Multilevel \iid sample covariance error]  
  \label{lem:iidcover}
For any $\varepsilon >0$, let $L$ and $\{M_\ell\}_{\ell=0}^L$ be defined as in Theorem~\ref{thm:main}.
If Assumption~\ref{ass:mlrates} holds, then for any $p\ge 2$ and $n \in \bbN$,
\begin{equation*}
\|\mfc^{\rm ML}_n - \mfc^L_n\|_{L^p(\Omega, \cH \otimes \cH)}   \lesssim \varepsilon,
\label{eq:iidcover}
\end{equation*}
where we recall that $\mfc^L_n \coloneq \cov[\mfv^L_n]$.
\end{lemma}

\begin{proof}
  Since the sample covariances in~\eqref{eq:covML-def} are unbiased, 
  \[
  \E{\mfc^{\rm ML}_n } =  \cov[\mfv^L_n] = \sum_{\ell=0}^L \cov[\mfv^\ell_n]-\cov[\mfv^{\ell-1}_n],
  \]
and therefore
\begin{equation*}
\begin{split}
 \|\mfc^{\rm ML}_n  &- \mfc^L_n \|_{p} =
 \|\mfc^{\rm ML}_n  - \E{\mfc^{\rm ML}_n}\|_{p}.
\end{split}
\end{equation*}
Next, introduce the linear centering operator
$\Upsilon:L^1(\Omega,\cH \otimes \cH) \to L^1(\Omega,\cH \otimes \cH)$,
defined by $\rvCenter{Y} = Y - \E{Y}$.  
Then, by equation~\eqref{eq:covML-def},
\begin{equation*}
 \begin{split}
 \|\mfc^{\rm ML}_n  - &\E{\mfc^{\rm ML}_n}\|_{p}
= \bigg\|\sum_{\ell = 0}^L \rvCenter{\cov_{M_\ell}[\mfv^\ell_n] - \cov_{M_\ell}[\mfv^{\ell-1}_n]}\bigg\|_{p} \\
&\leq \sum_{\ell = 0}^L \big\|\rvCenter{\cov_{M_\ell}[\mfv^\ell_n] - \cov_{M_\ell}[\mfv^{\ell-1}_n] } \big\|_{p}\\
&\leq \sum_{\ell = 0}^L \left( \big\|\rvCenter{\cov_{M_\ell}[\mfv^\ell_n, \Delta_\ell \mfv_n]} \big\|_{p}
+ \big\|\rvCenter{\cov_{ M_\ell}[\Delta_\ell \mfv_n, \mfv^{\ell-1}_n]} \big\|_{p}
\right),
 \end{split}
\end{equation*}
where $\Delta_\ell \mfv_n := \mfv_n^\ell - \mfv_n^{\ell-1}$ with the convention $\mfv_n^{-1}=0$,
and
\begin{equation*}
  \begin{split}
\rvCenter{\cov_{M_\ell}[\mfv^\ell_n, \Delta_\ell \mfv_n]}
&= \cov_{M_\ell}[\mfv^\ell_n, \Delta_\ell \mfv_n] -  \cov[\mfv^\ell_n, \Delta_\ell \mfv_n],\\
\rvCenter{\cov_{ M_\ell}[\Delta_\ell \mfv_n, \mfv^{\ell-1}_n]}
&= \cov_{M_\ell}[\Delta_\ell \mfv_n, \mfv^{\ell-1}_n]  -  \cov[\Delta_\ell \mfv_n, \mfv^{\ell-1}_n].\\
\end{split}
\end{equation*}
By Lemmas~\ref{lem:disccov} and~\ref{lem:covM-Lr-error} (the latter lemma is located in Appendix~\ref{app:mzIneq}),
\begin{equation*}
 \begin{split}
\|\mfc^{\rm ML}_n  - \E{\mfc^{\rm ML}_n}\|_{p}
&\leq 2\sum_{\ell = 0}^L \frac{c}{\sqrt{M_\ell}}
( \|\mfv_n^\ell\|_{2p} + \|\mfv_n^{\ell-1}\|_{2p}) 
\| \Delta_\ell \mfv_n 
\|_{2p} \\
&\lesssim 
\sum_{\ell = 0}^L \frac{1}{\sqrt{M_\ell}} \|\Delta_\ell \mfv_n \|_{2p}
 \lesssim  \sum_{\ell = 0}^L M_\ell^{-1/2}h_\ell^{\beta/2}
\lesssim \varepsilon.
 \end{split}
\end{equation*}
\end{proof}

We now turn to bounding the last
term of the right-hand side of inequality~\eqref{eq:covspliteps}.

\begin{lemma}
\label{lem:mlcov}
For any $\varepsilon >0$, let $L$ and $\{M_\ell\}_{\ell=0}^L$ be defined as in Theorem~\ref{thm:main}.
If Assumption~\ref{ass:mlrates} holds, then for any $p\ge 2$ and $n \in \bbN$,
\begin{equation}
\begin{split}
 \|C^{\rm ML}_n - \bar{C}^{\rm ML}_n\|_{L^p(\Omega, \cH \otimes \cH)} \leq &
8 \sum_{l=0}^L \|v_{n}^{\ell} - \mfv_{n}^\ell\|_{L^{2p}(\Omega, \cH)} (\|v_{n}^{\ell}\|_{L^{2p}(\Omega, \cH)}  + \|\mfv_{n}^{\ell}\|_{L^{2p}(\Omega,\cH)}).
\label{eq:mlcov}
\end{split}
\end{equation}
\end{lemma}

\begin{proof}
From the definitions of the sample covariance~\eqref{eq:sampleCov} and multilevel sample covariance~\eqref{eq:mlSampleMoments},
one obtains the bounds
\[
\begin{split}
\|C^{\rm ML}_n - \bar{C}^{\rm ML}_n\|_{p} & \leq 
  \sum_{\ell=0}^L \Big( \|\cov_{M_\ell}[v_n^\ell] - \cov_{M_\ell}[\mfv_n^\ell]\|_{p}  \\
&+ \|\cov_{M_\ell}[v_n^{\ell-1}] - \cov_{M_\ell}[\mfv_n^{\ell-1}]\|_{p}\Big), 
\end{split}
\]
and
\[
\begin{split}
\norm{\cov_{M_\ell}[v_n^\ell] - \cov_{M_\ell}[\mfv_n^\ell]}_{p}
&\leq \frac{M_\ell}{M_\ell-1}\norm{E_{M_\ell}[v_n^\ell \otimes v_n^\ell] - E_{M_\ell}[\mfv_n^\ell \otimes \mfv_n^\ell]}_{p} \\
 & + \frac{M_\ell}{M_\ell-1} \norm{E_{M_\ell}[v_n^\ell] \otimes E_{M_\ell}[ v_n^\ell] - E_{M_\ell}[\mfv_n^\ell ] \otimes E_{M_\ell}[\mfv_n^\ell]}_{p}\\
& \eqcolon I_{1} + I_{2}.
\end{split}
\]
The bilinearity of the sample covariance yields that
\begin{equation}\label{eq:I1}
I_1 \leq 2\norm{E_{M_\ell}[(v_n^\ell - \mfv_n^\ell) \otimes v_n^\ell] }_{p}
 + 2\norm{E_{M_\ell}[\mfv_n^\ell \otimes (v_n^\ell - \mfv_n^\ell) ]}_{p}
\end{equation}
and 
\begin{equation*}\label{eq:I2}
I_2 \leq 2\norm{E_{M_\ell}[(v_n^\ell - \mfv_n^\ell)]  \otimes E_{M_\ell}[v_n^\ell] }_{p} 
 + 2\norm{E_{M_\ell}[\mfv_n^\ell\hh{]} \otimes E_{M_\ell}[(v_n^\ell - \mfv_n^\ell) ]}_{p}.
\end{equation*}
For bounding $I_1$ we use Jensen's and H\"older's inequalities: 
\[
\begin{split}
\norm{E_{M_\ell}[(v_n^\ell - \mfv_n^\ell) \otimes v_n^\ell] }_{p}^p
& = \E{ \hhNorm{E_{M_\ell}[(v_n^\ell - \mfv_n^\ell)  \otimes v_n^\ell]}^p}\\
& \leq \E{E_{M_\ell}\Big[\hNorm{v_n^\ell - \mfv_n^\ell}^p  \hNorm{v_n^\ell}^p\Big]} \\
& = \E{ \hNorm{v_n^\ell - \mfv_n^\ell}^p  \hNorm{v_n^\ell}^p} \\
& \leq \norm{v_n^\ell - \mfv_n^\ell}_{2p}^p \norm{v_n^\ell}_{2p}^p.
\end{split}
\]
The second summand of inequality~\eqref{eq:I1} is bounded similarly, and we obtain
\[
I_1 \leq 2\norm{v_n^\ell - \mfv_n^\ell}_{2p} 
\parenthesis{\norm{v_n^\ell}_{2p} + \norm{\mfv_n^\ell}_{2p}}.
\]
The $I_2$ term can also be bounded with similar steps 
as in the preceding argument so that also
\[
I_2 \leq 2\norm{v_n^\ell - \mfv_n^\ell}_{2p}  
\parenthesis{\norm{v_n^\ell}_{2p }+ \norm{\mfv_n^\ell}_{2p}}.
\]
The proof is finished by summing the contributions of $I_1$ and $I_2$ over all levels.
\end{proof}

The propagation of error in update steps of MLEnKF
is governed by the magnitude $\|\bar C_n - C_n^{\rm ML}\|_p$,
i.e., the distance between the MFEnKF prediction covariance and the MLEnKF
prediction covariance. The next lemma makes use of Lemma~\ref{lem:mlcov}
to bound the distance between the mean-field multilevel ensemble
$\{(\hmfv^{\ell-1}_{n,i},\hmfv^{\ell}_{n,i})_{i=1}^{M_\ell}\}_{\ell=0}^L$
and the MLEnKF ensemble
$\{(\hv^{\ell-1}_{n,i}, \hv^{\ell}_{n,i})_{i=1}^{M_\ell}\}_{\ell=0}^L$.
\begin{lemma}[Distance between ensembles.] 
\label{lem:ensdist}
For any $\varepsilon >0$, let $L$ and $\{M_\ell\}_{\ell=0}^L$ be defined as in Theorem~\ref{thm:main}.
If Assumption~\ref{ass:mlrates} holds, then for any $p\ge 2$ and $n \in \bbN$,
\begin{equation}
\sum_{\ell=0}^L  \|\hv_{n}^{\ell}- \hmfv_{n}^{\ell}\|_{L^p(\Omega, \cH)} \lesssim |\log(\varepsilon)|^n\varepsilon.
\label{eq:ensdist}
\end{equation}
\end{lemma}

\begin{proof} The proof is similar to that of~\cite[Lemma
  3.10]{ourmlenkf}. For completeness, a proof is given in Appendix
  \ref{app:extras}.
\end{proof}

With the bound between MLEnKF and its multilevel MFEnKF shadow, that
conveniently for analysis consists of independent particles, we are
finally ready to prove the main result.

\begin{proof}[Proof of Theorem~\ref{thm:main}] 

  
By the triangle inequality, 
\begin{align}
\nonumber
\|\hat \mu^{\rm ML}_n [\varphi] - \hat \mfmu_n [\varphi]\|_{p} & \leq 
\|\hat \mu^{\rm ML}_n [\varphi] - \hat \mfmu^{\rm ML}_n [\varphi]\|_{p}  
 + \|\hat\mfmu^{\rm ML}_n [\varphi] - \hat \mfmu^L_n[\varphi]\|_{p} \\
&+ \|\hat \mfmu^L_n[\varphi] - \hat \mfmu_n [\varphi]\|_{p},
\label{eq:lpertri}
\end{align}
where $\hat \mfmu^{\rm ML}_n$ denotes the empirical measure associated
to the mean-field multilevel ensemble
$\{(\hmfv^{\ell-1}_{n,i},
\hmfv^{\ell}_{n,i}))_{i=1}^{M_\ell}\}_{\ell=0}^L$, and
$\hat \mfmu^L_n$ denotes the probability measure associated to
$\hmfv^L$. The two first summands on the right-hand side above
relate to the statistical error, whereas the last relates
to the bias.

By the Lipschitz continuity of the QoI $\varphi$,
the triangle inequality, Lemma~\ref{lem:ensdist},
and using the conventions $\varphi(\hv_n^{-1})=0$ and $\varphi(\hmfv_n^{-1})=0$,
the first term satisfies the following bound 
\begin{equation*} \label{eq:finensembles}
\begin{split}
\norm{\hat \mu^{\rm ML}_n [\varphi] - \hat \mfmu^{\rm ML}_n [\varphi]}_{p} 
&= \left\| \sum_{\ell=0}^L E_{M_\ell}\big[\varphi(\hv_n^{\ell}) - \varphi(\hv_n^{\ell-1}) - 
(\varphi(\hmfv_n^\ell) - \varphi(\hmfv_n^{\ell-1}))\big] \right\|_{p}\\
& \leq \sum_{\ell=0}^L \parenthesis{\norm{ \varphi(\hv_n^{\ell}) - \varphi(\hmfv_n^\ell) }_{p}
+ \norm{ \varphi(\hv_n^{\ell-1}) - \varphi(\hmfv_n^{\ell-1}) }_{p}} \\
& \leq c_\varphi \sum_{\ell=0}^L \parenthesis{\norm{ \hv_n^{\ell} - \hmfv_n^\ell }_{p}
+ \norm{ \hv_n^{\ell-1} - \hmfv_n^{\ell-1} }_{p}} \\
& \lesssim |\log(\varepsilon)|^n\varepsilon.
\end{split}
\end{equation*}

For the second summand~{of~\eqref{eq:lpertri}}, we employ the telescoping property 
\[
  \hat \mfmu^L_n[\varphi] = \sum_{\ell=0}^L \E{  \varphi(\hat \mfmu^\ell_n) - \varphi(\hat \mfmu^{\ell-1}_n)},
\]
and Lemmas~\ref{lem:disccov} and~\ref{lem:mz} to obtain
\begin{equation*}\label{eq:finvar}
 \begin{split}
\norm{\hat{\bar{\mu}}^{\rm ML}_n [\varphi] - \hat \mfmu^L_n[\varphi]}_{p}  & \leq
\sum_{\ell=0}^L \norm{E_{M_\ell}\Big[\varphi(\hmfv_n^\ell) - \varphi(\hmfv_n^{\ell-1}) - 
\E{\varphi(\hmfv_n^\ell) - \varphi(\hmfv_n^{\ell-1})} \Big]}_{p}  \\
 & \leq c \sum_{\ell=0}^L M_\ell^{-1/2}\norm{ \varphi(\hmfv_n^\ell) - \varphi(\hmfv_n^{\ell-1}) }_{p} \\
  & \leq \tilde c \sum_{\ell=0}^L M_\ell^{-1/2}\| \hmfv_n^\ell - \hmfv_n^{\ell-1}\|_{p} \\
& \lesssim \sum_{\ell=0}^L M_\ell^{-1/2}h_\ell^{\beta/2} \lesssim \varepsilon.
 \end{split}
 \end{equation*}
Finally, the bias term in~\eqref{eq:lpertri} satisfies 
\begin{equation}\label{eq:finbias}
\begin{split}
\|\hat \mfmu^L_n[\varphi] - \hat \mfmu_n [\varphi]\|_{p} 
 =  |\hat \mfmu^L_n[\varphi] - \hat \mfmu_n [\varphi]| 
 = \abs{\E{\varphi(\hmfv^L_n) - \varphi(\hmfv_n)} }
\lesssim \varepsilon,
\end{split}
\end{equation}
where the last step follows from the Lipschitz continuity of the
QoI and Lemma~\ref{lem:disccov}. 
\end{proof}

\begin{remark}\label{rem:logTermIsFishy}
  Theorem~\ref{thm:main} shows the cost-to-accuracy performance of
  MLEnKF with a disconcerting logarithmic penalty factor
  in~\eqref{eq:lperror} that grows geometrically in $n$.  The same
  penalty appears in the work~\cite{ourmlenkf}, yet the numerical
  experiments there indicate a rate of convergence that is uniform in
  $n$. The discrepancy between theory and practice may be an artifact
  of conservative bounds used in the proof of said theorem.  By
  imposing further regularity constraints on the dynamics and the QoI,
  we were able to obtain an error bound without said logarithmic
  penalty factor for an alternative finite-dimensional-state-space
  MLEnKF method with local-level Kalman
  gains~\cite{hoel2020multilevel}.  As an alternative to imposing
  further regularity constraints, we also suspect that ergodicity of
  the MFEnKF process may be used to avoid the geometrically growing
  the logarithmic penalty factor. Recently, there has been much work
  on the stability of EnKF~\cite{del2016stability, del2016stability1, tong2016nonlinear}.
\end{remark}

We conclude this section with a result on the
cost-to-accuracy performance of EnKF. It shows that
MLEnKF generally outperforms EnKF.

\begin{theorem}[EnKF accuracy vs.
  cost]\label{thm:mainEnKF}
  {Consider a Lipschitz continuous QoI $\varphi:\cH \to \bbR$, and
    suppose Assumption~\ref{ass:mlrates} holds. For a given
    $\varepsilon >0$, let $L$ and $M$ be defined under the respective
    constraints
    $L = \lceil 2\kl{d} \log_\kappa(\varepsilon^{-1})/\beta\rceil$ and
    $ M \eqsim \varepsilon^{-2}$.  Then, for any $n \in \bbN$ and
    $p \ge 2$,}
\begin{equation*}
\|\hat{\mu}^{\rm MC}_n [\varphi] - \hat{\bar{\mu}}_n [\varphi] \|_{L^p(\Omega,\cH)} \lesssim \varepsilon,
\label{eq:lperrorEnKF}
\end{equation*}
where $\hat{\mu}^{\rm MC}_n$ denotes the EnKF empirical measure,
cf.~equation~\eqref{eq:emp}, with particle evolution given by the EnKF
predict and update formulae at resolution level $L$ (i.e., using the
numerical approximation $\Psi^L$ in the prediction and the projection
operator $\Pi_L$ in the update).

The computational cost of the EnKF estimator
\[
\cost{\hat{\mu}^{\rm MC}_n [\varphi]} := M \cost{\Psi^L}
\]
satisfies
\begin{equation*}\label{eq:mlenkfCostsEnKF}
\cost{\hat{\mu}^{\rm MC}_n [\varphi]} \eqsim \varepsilon^{-2(1+(d\gamma_x\kl{+\gamma_t})/\beta)}.
\end{equation*}
\end{theorem}

\begin{proof}[Sketch of proof]
By the triangle inequality,
\[
\begin{split}
\|\hat{\bar{\mu}}_n[\varphi] -\hat{\mu}^{\rm MC}_n [\varphi]
\|_{L^p(\Omega)} & \leq 
\norm{\hat{\bar{\mu}}_n[\varphi] -  \hat{\bar{\mu}}_n^L [\varphi]}_{L^p(\Omega)}
+\norm{\hat{\bar{\mu}}^L[\varphi] - \hat{\bar{\mu}}_n^{\rm MC} [\varphi] }_{L^p(\Omega)} \\
& +  \norm{\hat{\bar{\mu}}_n^{\rm MC} [\varphi] -\hat{\mu}^{\rm MC}_n [\varphi]}_{L^p(\Omega)}
 \eqcolon I + II + III,
\end{split}
\]
where $\hat{\bar{\mu}}_n^{\rm MC}$ denotes the empirical measure associated
to the EnKF ensemble $\{\hmfv^L_{n,i} \}_{i=1}^M$ and
$\hat{\bar{\mu}}^{L}_n$ denotes the empirical measure associated to
$\hmfv^L_n$. 
It follows by inequality~\eqref{eq:finbias} that $I \lesssim \varepsilon$.

{For the second term, 
the Lipschitz continuity of the QoI $\varphi$ implies there exists 
a positive scalar $c_\varphi$ such that
$|\varphi(x)| \leq c_\varphi (1+\hNorm{x})$}.
Since $\hat{\bar{v}}_n^L \in L^p(\Omega, \cHO)$ for any $n \in \bbN$
and $p\ge 2$, it follows by Lemma~\ref{lem:mz} (on the Hilbert space $\bbR$) that
\begin{equation*}
\begin{split}
II \leq \norm{E_{M}[ \varphi(\hmfv^L_n)] - \E{\varphi(\hmfv^L_n)} }_{L^p(\Omega)}  \leq M^{-1/2} c_\varphi\norm{\hmfv^L_n }_{L^p(\Omega,\cH)} 
\lesssim \varepsilon.
\end{split}
\end{equation*}

For the last term, let us first assume that for any $p \ge 2$ and 
 $n \in \bbN$, 
\begin{equation}\label{eq:termIIIAssumption}
\norm{\hv_{n}^L -  \hmfv_{n}^L }_{L^p(\Omega,\cH)} \lesssim \varepsilon,
\end{equation}
for the {single particle dynamics $\hv_{n,1}^L$ and $\hmfv_{n,1}^L$ respectively
 associated to the EnKF ensemble $\{\hv_{n,i}^L \}_{i=1}^M$ and 
the mean-field EnKF ensemble $\{\hmfv_{n,i}^L\}_{i=1}^M$.
Then the Lipschitz continuity of $\varphi$, the fact
that $\hv_{n,1}^L, \hmfv_{n,1}^L \in L^p(\Omega, \cHO)$}
for any $n\in \bbN$ and  $p\ge 2$ holds (when assuming~\eqref{eq:termIIIAssumption}),
and the triangle inequality yield that
\[
\begin{split}
III = \norm{ E_{M}[ \varphi(\hv_n^L) - \varphi(\hmfv_n^L)] }_{L^p(\Omega)}
\leq c_\varphi \norm{\hv_n^L - \hmfv_n^L }_{L^p(\Omega,\cH)} 
\lesssim \varepsilon.
\end{split}
\]
\noindent All that remains is to verify~\eqref{eq:termIIIAssumption}, but we
omit this as it can be done by similar steps as for the proof of inequality~\eqref{eq:ensdist}.

\end{proof}

\section{MLEnKF-adapted numerical methods for a class of stochastic partial differential equations}
\label{sec:numericalAnalysis}

In this section we develop an MLEnKF-adapted version of the
exponential Euler method, for the purpose of solving a family of
stochastic reaction-diffusion equations.  For a relatively large 
class of problems, we derive an $L^p(\Omega, \cH)$-convergence rate
$\beta$ for pairwise coupled numerical solutions,
cf.~Assumption~\ref{ass:mlrates}, which will needed when
implementing MLEnKF.

\subsection{The stochastic reaction-diffusion eqquation}
\label{subsec:she}
We consider the following stochastic partial differential equation (SPDE)
\begin{equation}
\begin{split}
du &= (\Delta u + f(u)) dt + B dW, \qquad (t,x) \in (0,T] \times (0,1),\\
u(0,\cdot) &= u_0,\\
u(t,0) &= u(t,1) = 0, \mspace{164mu}  t\in (0,T],
\end{split}
\label{eq:she}
\end{equation}
where $T>0$, and the reaction $f$, the cylindrical Wiener process $W$ and 
the linear smoothing operator $B$ will be described below.  
Our base-space is $\cK = L^2(0,1)$, we denote by
$A: D(A) =H^2(0,1) \cap H^1_0(0,1)\to \cK$ 
the Laplace operator $\Delta$ with zero-valued Dirichlet boundary conditions 
and $H^k(0,1)$ denotes the
Sobolev space of order $k \in \bbN$.
A spectral decomposition of $-A$ yields 
the sequence of eigenpairs $\{(\lambda_j, \phi_j)\}_{j \in \bbN}$ where
$-A \phi_j = \lambda_j \phi_j$ with
$\phi_j \coloneq \sqrt{2} \sin(j\pi x)$ and $\lambda_j
=\pi^2 j^2$.
$\cK = \overline{\text{span}\{\phi_j\}}$, it follows that 
\[
Av = \sum_{j \in \bbN} -\lambda_j \langle \phi_j, v \rangle_{\cK} \phi_j, \quad \forall v \in D(A),
\]
and eigenpairs of the spectral decomposition 
give rise to the following family of Hilbert spaces
parametrized over $r \in \bbR$:
\begin{equation*}\label{eq:fracSpace}
\begin{split}
\cK_r \coloneq D((-A)^{r}) = \Big\{v:[0,1] \to \bbR \mid & v \text{ is }
\mathcal{B}([0,1])/\mathcal{B}(\bbR)\text{-measurable} 
\\& \text{and } \sum_{j \in \bbN}
\lambda^{2r}_j \abs{\langle \phi_j, v \rangle_{\cK}}^2 < \infty \Big\},
\end{split}
\end{equation*}
with norm $\|\cdot\|_{\cK_r} \coloneq \|(-A)^{r}(\cdot)\|_{\cK}$.
Associated with the probability space $(\Omega, \cF, \bbP)$ and 
normal filtration $\{\cF_t\}_{t \in [0,T]}$, 
the $I_{\cK}$-cylindrical Wiener process is defined by 
\[
W(t, \cdot) = \sum_{j \in \bbN}  W_j(t)\phi_j, 
\]
where $\{W_j :  [0,T] \times \Omega  \to \bbR \}_{j \in \bbN}$ is a
sequence of independent $\cF_t /\cB(\bbR)$-adapted 
standard Wiener processes.
The smoothing operator is defined by
\begin{equation}\label{eq:bDef}
B \coloneq \sum_{j \in \bbN} \lambda_j^{-b} \phi_j \otimes \phi_j,
\end{equation}
with the smoothing paramter $b \ge 0$. It may be shown that
  $B \cK_r = \cK_{r+b}$, and this implies that $B$ becomes progressively
  more smoothing the higher the value of $b$. 

In the remaining part of this section, we assume the following
conditions on the nested Hilbert spaces $\cHO \subset \cH$
and regularity conditions on the initial data $u_0$ and the
reaction term $f$ hold:
\begin{assumption}\label{ass:spde}
The Hilbert spaces $V \subset \cH$
are of the form $\cH = \cK_{r_1}$ and $V = \cK_{r_2}$
for a pair of parameters $r_1,r_2 \in \bbR$ satisfying
\[
\max(0,b-1/4) \le r_1 < r_2 < b +1/4,
\]
the initial data $u_0$ is $\cF_0 /\cB(\cHO)$-measurable
and  $ u_0\in \cap_{p\ge2} L^p(\Omega, V)$,
and the reaction satisfies
\[
f \in \mathrm{Lip}(\cK_{r_1}) \coloneq
\left\{ g \in C(\cK_{r_1}, \cK_{r_1}) \, \Big| \,
\sup_{u,v \in \cK_{r_1}, u \neq v} \frac{ \|g(u) - g(v)\|_{\cK_{r_1}}}{\|u-v\|_{\cK_{r_1}}} < \infty \right\}.
\]
\end{assumption}
   
Under Assumption~\ref{ass:spde} there exists an up to modifications
unique $(\Omega, \cF, \bbP, \{\cF_t\}_{t \in [0,T]})$-mild solution
of~\eqref{eq:she}, which in this setting corresponds to a mapping
$u:[0,T]\times [0,1] \times \Omega \to \bbR$
that is an $\cF_t/\cB(\cK_{r_2})$-predictable stochastic process satisfying
\begin{equation}\label{eq:mildSolution}
u(t) = e^{At}u(0) + \int_0^t e^{A(t-s)} f(u(s)) ds + \int_0^t
  e^{A(t-s)} B dW_s
\end{equation}
$\bbP$-almost surely for all $t\in [0,T]$.
Moreover, for any $p \ge 2$ and $r \in [r_1,r_2]$, it holds that
\begin{equation}\label{eq:sheWellP}
\|u(T,\cdot)\|_{L^p(\Omega, \cK_r )}  \leq C(1+\|u_0\|_{L^p(\Omega, \cK_r)}), 
\end{equation}
where $C>0$ depends on $p$, $r$, and $T$, cf.~\cite{jentzenNotes2015}.

\begin{remark}
  The Dirichlet zero-valued boundary conditions imposed
  in~\eqref{eq:she} only make pointwise sense provided $u(t,\cdot) \in
  \cK_{1/2+\delta}$ for some $\delta>0$ and all $t \in (0,T]$,
    $\bbP$-almost surely.  In lower-reglarity settings, e.g., when
    $u(t,\cdot) \notin C(0,1)$, said boundary condition should be interpreted
    in mild rather than pointwise sense.
\end{remark}

\subsection{The filtering problem}
\label{subsec:filtering}
We consider a discrete-time filtering problem of the
form~\eqref{eq:psiDefinition} and~\eqref{eq:obdef} with the above SPDE
as underlying model with $\Psi(u_n)$ denoting the mild solution
of~\eqref{eq:she} at $T>0$ given the initial data
$u_n \in \cap_{p \ge 2} L^p(\Omega, \cHO )$. The underlying dynamics
at observation times $0,T,2T, \ldots$ is thus described by the dynamics
\[
u_{n+1} = \Psi(u_{n}),  
\]
and the finite-dimensional partial observation of $u_n$ at time $nT$ is given by 
\begin{eqnarray}
y_n = H u_n + \eta_n, \quad \eta_n \sim N(0,\Gamma) ~~ {\rm i.i.d.}
\perp u_n, 
\label{eq:she_obs}
\end{eqnarray} 
where $H u = [H_1(u), \dots, H_m(u)]^\transpose \in \bbR^m$ with $H_i \in \cH^*$ for $i=1,2,\ldots, m$.
Note further that Assumption~\ref{ass:spde} implies that $\cHO \subset \cK_0=\cK$,
so we may represent the mild solution in the basis $\{\phi_j\}$ at any observation
time $nT$:
\[
u_{n} = \sum_{j \in \bbN} u_{n}^{(j)} \phi_j, \quad \text{where} \quad  u_{n}^{(j)} \coloneq
\langle u_{n+1}, \phi_j \rangle_{\cK}.
\]

\subsection{Spatial truncation}
\label{subsec:spatiallyDiscrete}
Before introducing a fully discrete MLEnKF approximation method for the
filtering problem, let us warm up by having a quick look at 
an exact-in-time-truncated-in-space approximation method.
It consists of the hierarchy of subspaces
$\cH_\ell = \cP_\ell \cH = \text{span}(\{\phi_{k}\}_{k=1}^{N_\ell})$,
$\Pi_\ell = \cP_\ell$ and 
\[
\Psi^\ell(u_n) \coloneq \cP_\ell u_{n+1} = \sum_{j=1}^{N_\ell} u_{n+1}^{(j)}
\phi_j, \quad \text{for any} \quad u_n \in \cH.
\]

To verify that this approximation method can be used in 
the MLEnKF framework, it remains to verify
Assumptions~\ref{ass:psilip} and~\ref{ass:mlrates}(i)-(ii),
and to determine the rate parameter $\beta>0$.
The equation~\eqref{eq:mildSolution}, the regularity $f \in \rm{Lip}(\cK_{r_1})$,
the inequality
\begin{equation}\label{eq:lipExpA}
\sup_{v \in \cK_{r}\setminus \{0\} } \frac{ \| e^{At} v
  \|_{\cK_r}}{\|v\|_{\cK_r}} \leq 1, \quad \text{for all} \quad  r \in \bbR \quad \text{and}  \quad t \ge 0,
\end{equation}
and Jensen's inequality imply that for any $p\ge2$, there exists a $C>0$ such that
\[
\begin{split}
&\|\Psi^\ell(u_0) - \Psi^\ell(v_0)\|_{L^p(\Omega,\cH)} \leq  \|\Psi(u_0)-
\Psi(v_0)\|_{L^p(\Omega,\cH)} \\
&\leq  \|u_0 -v_0\|_{L^p(\Omega,\cH)} + \left\| \int_{0}^T e^{A(t-s)} (f(u(s))-f(v(s)))  ds \right\|_{L^p(\Omega,\cH)} \\
&\leq \|u_0 -v_0\|_{L^p(\Omega,\cH)} +   \int_{0}^T
    \|f(u(s))-f(v(s))\|_{L^p(\Omega,\cH)} ds   \\
&= \|u_0 -v_0\|_{L^p(\Omega,\cH)} + 
 C \int_{0}^T \| u(s)-v(s)\|_{L^p(\Omega, \cH)}  ds.
\end{split}
\]

Hence by Gronwall's inequality, 
\[
\|\Psi^\ell(u_0) - \Psi^\ell(v_0)\|_{L^p(\Omega,\cH)}\leq C
\|u_0-v_0\|_{L^p(\Omega, \cH)},
\]
which verifies Assumption~\ref{ass:psilip}(i).
Assumption~\ref{ass:psilip}(ii) follows
from~\eqref{eq:sheWellP}. 
To verify that Assumption~\ref{ass:mlrates}(i) holds with rate 
$\beta= 4(r_2-r_1)$, observe that for any $p\ge 2$,
\begin{equation}\label{eq:spatialRate}
\begin{split}
\|\Psi^\ell(u_0) - \Psi(u_0) \|_{L^p(\Omega, \cH)} 
&= \| \sum_{j>N_\ell} \lambda_j^{r_1} u_1^{(j)} \phi_j \|_{L^p(\Omega, \cK)} \\
&\leq  \lambda_{N_\ell}^{-(r_2-r_1)}  \| \sum_{j> N_\ell} \lambda_{j}^{r_2} u_1^{(j)} \phi_j \|_{L^p(\Omega, \cK)} \\
& \lesssim N_\ell^{-2(r_2-r_1)} \| u_1 \|_{L^p(\Omega, \cHO)} \\
& \lesssim (1 + \| u_0 \|_{L^p(\Omega, \cHO)}) h_\ell^{2(r_2-r_1)},
\end{split}
\end{equation}
where the last inequality follows from~\eqref{eq:sheWellP} and
$h_\ell \eqsim N_\ell^{-1}$.  Assumption~\ref{ass:mlrates}(ii) follows
by a similar shift-space argument.  (Relating to
Assumption~\ref{ass:mlrates}(iii), we leave the question of the
computational cost of this method open for the time being, but see
Section~\ref{subsec:linearFilteringEx} for treatment in one example.)

\subsection{A fully-discrete approximation method}
\label{subsec:fullyDiscrete}
The fully-discrete approximation method
consists of Galerkin approximation in 
space and numerical integration in
time by the exponential Euler scheme, cf.~\cite{jentzen2009, jentzen2011}.
Given a timestep $\Delta t_\ell = T/J_\ell$, let $\{U_{\ell, k}
\}_{k=0}^{J_\ell} \subset \cH_\ell$ with $U_{\ell,0} = \cP_\ell u_0$ 
denote the numerical approximation SPDE~\eqref{eq:she} on level $\ell$.
It is given by the scheme
\[
\begin{split}
U_{\ell,k+1} &= e^{A_\ell \Delta t_\ell } U_{\ell,k} + A_\ell^{-1}
\parenthesis{e^{A_\ell \Delta t_\ell } -I} f_\ell (U_{\ell, k})  
\\
&+ \underbrace{\cP_\ell \int_{k\Delta t_\ell}^{(k+1)\Delta t_\ell} e^{A
    ((k+1)\Delta t_\ell -s)}  B  dW(s) }_{\eqcolon R_{\ell,k}},
\end{split}
\]
where $A_\ell \coloneq \cP_\ell A$ and $f_\ell \coloneq \cP_\ell f$.
The $j$-th mode of the scheme $U^{(j)}_{\ell,k} \coloneq \langle U_{\ell, k},
\phi_j\rangle_{\cK}$ for $j =1,2,\ldots, N_\ell$,  
is given by 
\begin{equation}\label{eq:modeUFine}
U_{\ell, k+1}^{(j)} = e^{-\lambda_j \Delta t_\ell } U_{\ell,k}^{(j)} + \frac{1-
  e^{-\lambda_j \Delta t_\ell }}{\lambda_j}  \parenthesis{f_\ell(U_{\ell, k} )}^{(j)}
+ R_{\ell,k}^{(j)},
\end{equation}
for $k=0,1,\ldots, J_\ell-1$ with i.i.d.
\begin{equation}\label{eq:rMDist}
R_{\ell,k}^{(j)} \sim N\parenthesis{0, \frac{1-e^{-2\lambda_j \Delta
      t_\ell}}{2\lambda_j^{1+2b}}},
\end{equation}
for $j \in \{1,2,\ldots, N_\ell\}$, $k
\in \{0,1,\ldots, J_\ell-1\}$ and $\ell = 0,1, \ldots$
In view of the mode-wise numerical solution, 
the $\ell$-th level solution operator for the fully-discrete 
approximation method is defined by
\[
\widetilde \Psi^\ell(u_0) \coloneq \sum_{j=1}^{N_\ell} U_{\ell, J_\ell}^{(j)} \phi_j.
\]

\subsubsection{Coupling of levels} For a hierarchy of temporal resolutions
$\{\Delta t_\ell = T/J_\ell\}$ with $J_\ell = 2^\ell J_0$, pairwise
correlated solutions $(\widetilde \Psi^{\ell-1}(u_0), \widetilde
\Psi^\ell(u_0))$ are obtained through first generating the fine-level
driving noise $\{R_{\ell,k}\}_{k}$ and the solution $\widetilde
\Psi^\ell(u_0)$ by~\eqref{eq:modeUFine} and thereafter computing the
coarse level solution conditioned on $\{R_{\ell,k}\}_{k}$.
Since $J_\ell = 2J_{\ell-1}$, it follows that
\[
\begin{split}
&\cP_{\ell-1} \int_{k\Delta t_{\ell-1}}^{(k+1)\Delta t_{\ell-1}}  e^{A ((k+1)\Delta
  t_{\ell-1} -s)}  BdW(s) \\
& = e^{A_{\ell-1}\Delta t_\ell} \cP_{\ell-1} \int_{2k\Delta t_{\ell}}^{(2k+1)\Delta t_{\ell}}  e^{A
  ((2k+1)\Delta t_{\ell} -s)}  B dW(s) \\
& \quad + \cP_{\ell-1} \int_{(2k+1)\Delta t_{\ell}}^{2(k+1)\Delta t_{\ell}}  e^{A ((k+1)\Delta t_{\ell-1} -s)} B dW(s).
\end{split}
\]
Consequently
\[
\begin{split}
R_{\ell-1,k} \Big| (R_{\ell,2k}, R_{\ell,2k+1}) & = 
 e^{A_{\ell-1} \Delta t_\ell} \cP_{\ell-1}R_{\ell,2k} + \cP_{\ell-1} R_{\ell,2k+1},
\end{split}
\]
and the conditional coarse-level solution $\widetilde \Psi^{\ell-1}(u_0)
|\{R_{\ell,k}\}_{k}$ is of the form
\begin{equation*}\label{eq:schemeUCoarse}
\begin{split}
U_{\ell-1, k+1} &=  e^{A_{\ell-1} \Delta t_{\ell-1} } U_{\ell-1,k} +
A_{\ell-1}^{-1} \parenthesis{e^{A_{\ell-1} \Delta t_{\ell-1}} - I} f_{\ell-1}( U_{\ell-1, k}) \\
& \quad + e^{A_{\ell-1} \Delta t_\ell} \cP_{\ell-1} R_{\ell,2k} + \cP_{\ell-1}R_{\ell,2k+1}
\end{split}
\end{equation*}
for $k= 0,1,\ldots,J_{\ell-1}-1$, and with the initial condition $U_{\ell-1,0} = \cP_{\ell-1} u_0$.
In other words, the scheme for the $j$-th mode of the coarse level
solution is given by
\begin{equation}\label{eq:modeUCoarse}
\begin{split}
U_{\ell-1, k+1}^{(j)} &=  e^{-\lambda_j \Delta t_{\ell-1} } U_{\ell-1,k}^{(j)} +
\frac{1-e^{-\lambda_j \Delta t_{\ell-1}}}{\lambda_j} \parenthesis{f_{\ell-1}( U_{\ell-1, k})}^{(j)} \\
& + e^{-\lambda_j \Delta t_\ell} R_{\ell,2k}^{(j)} + R_{\ell,2k+1}^{(j)},
\end{split}
\end{equation}
for $j =1,2,\ldots,N_{\ell-1}$ and $k= 0,1,\ldots,J_{\ell-1}-1$, and 
the coarse level solution takes the form
\[
\widetilde \Psi^{\ell-1}(u_0) = \sum_{j=1}^{N_{\ell-1}} U_{\ell-1, J_{\ell-1}}^{(j)} \phi_j.
\]
This coupling approach may be viewed as an extension of the multilevel
coupling of Ornstein--Uhlenbeck processes for stochastic differential
equations~\cite{muller2015improving}.

\subsubsection{Assumptions and convergence rates}
To show that the fully-discrete numerical method may be used in
MLEnKF, it remains to verify that Assumptions~\ref{ass:psilip}
and~\ref{ass:mlrates}(i)-(ii) hold. 

{\bf Assumption~\ref{ass:psilip}(i):} Let $U_{\ell, k}$ and
$\bar U_{\ell, k}$ denote solutions at time $t=k\Delta t_\ell$ 
of the scheme~\eqref{eq:modeUFine} with respective initial data 
$U_{\ell, 0} = \cP_\ell u_0$ and $\bar U_{\ell, 0} = \cP_\ell v_0$
fpr some
$u_0,v_0 \in L^p(\Omega, \cH)$. Then, by~\eqref{eq:lipExpA}, 
and the properties: (a)
for all $\ell \ge 0$ and $v \in \cH_\ell$
\begin{equation*}\label{eq:expLambdaBound}
\| A_\ell^{-1}( e^{A_\ell \Delta t_\ell} - I) v \|_{\cH} =
\norm{ \int_0^{\Delta t_\ell} e^{A_\ell s} v ds }_{\cH} \leq
\int_{0}^{\Delta t_\ell} \| e^{As} v\|_{\cH} ds \leq \|v\|_{\cH} \Delta t_\ell,
\end{equation*}
and (b) $f \in \rm{Lip}(\cH,\cH)$; there exists a $C>0$
such that 
\[
\begin{split}
\|U_{\ell, J_\ell} - \bar U_{\ell, J_\ell} \|_{L^p(\Omega, \cH)} &\leq
(1+C\Delta t_\ell)\| U_{\ell, J_{\ell}-1} - \bar U_{\ell,J_\ell -1}\|_{L^p(\Omega, \cH)} \\
& \leq (1+C\Delta t_\ell)^{T/\Delta t_\ell} \|\cP_\ell(u_0-
v_0)\|_{L^p(\Omega, \cH)}\\
& \leq e^{CT} \|u_0-v_0\|_{L^p(\Omega, \cH)}.
\end{split}
\]
Consequently, for every $p \ge 2$, there exists a $c_{\Psi}>0$ 
such that 
\[
\norm{\widetilde \Psi^\ell(u_0) - \widetilde \Psi^\ell(v_0) }_p \le c_{\Psi}\norm{u_0 - v_0}_p
\]
holds for all $\ell \ge 0$ and $u_0,v_0 \in L^p(\Omega,\cH)$.

{\bf Assumption~\ref{ass:psilip}(ii):}
Under the regularity
constraints imposed by Assumption~\ref{ass:spde}, it holds for all $\ell \in
\bbN$ and $U_{\ell,k} = \sum_{j=1}^{N_\ell} U^{(j)}_{\ell,k} \phi_k$ that
\begin{equation*}\label{eq:boundedNum}
\max_{k \in \{0,1,\ldots, J_\ell\}} \|U_{\ell,k}\|_{L^p(\Omega, \cK_{r} )}  \leq C(1+\|u_0\|_{L^p(\Omega, \cK_{r})}), \quad \forall p
\ge 2 \text{ and } r \in [r_1, r_2),
\end{equation*}
where $C>0$ depends on $r$ and $p$, but not on $\ell$,
cf.~\cite[Lemma 8.2.21]{jentzenNotes2015}.

{\bf Assumption~\ref{ass:mlrates}(i):} 
We begin by 
introducing the auxiliary $\ell$-th level exact-in-time Galerkin approximation 
\[
u^\ell(t) \coloneq  e^{A_\ell t}u_0 + \int_0^t e^{A_\ell (t-s)}
f_\ell(u^\ell(s)) ds + \int_0^t e^{A_\ell (t-s)} dW_s, \quad t \in [0,T],
\] 
and the notation $\widehat \Psi^\ell(u_0) \coloneq u^\ell(T)$.
Assumption~\ref{ass:spde} and~\cite[Corollary 8.1.12]{jentzenNotes2015} imply that
for any $p \ge 2$ and $r \in [r_1, r_2]$
\[
\sup_{\ell\ge 0}  \norm{\widehat \Psi^\ell(u_0)}_{L^p(\Omega, \cK_r)} \lesssim (1+
\norm{u_0}_{L^p(\Omega, \cK_r)}).
\]
The triangle inequality and~\eqref{eq:spatialRate} yield that
\begin{equation*}\label{eq:assIFirstBound}
  \norm{\Psi(u_0) - \widetilde \Psi^\ell(u_0) }_{L^p(\Omega, \cH)}
  \leq \norm{\Psi(u_0) - \widehat \Psi^\ell (u_0) }_{L^p(\Omega,\cH)}
    + \norm{ \widehat \Psi^\ell(u_0) -  \widetilde \Psi^\ell(u_0)}_{L^p(\Omega,\cH)},
\end{equation*}
and by~\cite[Corollary 8.1.11-12 and Theorem
  8.2.25]{jentzenNotes2015}\footnote{In the notation of the lecture
  notes~\cite{jentzenNotes2015}, the parameters $\gamma$, $\beta$ and
  $\eta$, which describe different properties than in this paper, take
  the values $\gamma=r_1$, $\beta=b-1/4$ and $\eta=2(\gamma-\beta)$.},
it respectively holds that for any $p\ge 2$
\[
\norm{\Psi(u_0) -\widehat \Psi^\ell (u_0) }_{L^p(\Omega,\cH)} \lesssim 
(1+ \|u_0\|_{L^p(\Omega, \cHO)}) N_\ell^{2(r_1-r_2)}, 
\]
and 
\begin{equation}\label{eq:convRF}
\norm{\widehat \Psi^\ell(u_0) -  \widetilde \Psi^\ell(u_0) }_{L^p(\Omega,\cH)} \lesssim
 (1+ \|u_0\|_{L^p(\Omega, \cH)})J_\ell^{r_1-r_2}.
\end{equation}
This verifies Assumption~\ref{ass:mlrates}(i) as it leads to the following bound:
for any $p\ge 2$,
\[
\|\Psi(u_0) - \widetilde \Psi^\ell(u_0) \|_{L^p(\Omega, \cH)} 
\lesssim  (1+ \|u_0\|_{L^p(\Omega, \cHO)}) (N_\ell^{2(r_1-r_2)} +
J_\ell^{r_1-r_2}).
\]
Assumption~\ref{ass:mlrates}(ii) only depends on the projection
operator, and thus follows from~\eqref{eq:spatialRate}.

\subsection{Linear forcing}

For the remaining part of this section consider the linear case
$f(u)=u$ of the filtering problem in Section~\ref{subsec:filtering}.
We derive explicit values for the rate exponents $\beta$, $\gamma_x$
and $\gamma_t$ when applying MLEnKF with either the exact-in-time-truncated-in-space
approximation method in Section~\ref{subsec:spatiallyDiscrete} or the
fully-discrete approximation method in Section~\ref{subsec:fullyDiscrete}. 

The exact solution of the $j$-th mode for this linear case is
\begin{eqnarray*}
u_{n+1}^{(j)} = e^{(1-\lambda_j)T} u_n^{(j)} + \xi_n^{(j)}, \quad 
\xi_n^{(j)} \sim N\left[0,\frac{\lambda_j^{-2b}}{2
    (\lambda_j-1)}(1-e^{2(1- \lambda_j )T})\right] \perp u_n^{(j)}.
\label{eq:she_exactmode}
\end{eqnarray*}
Although we see that the underlying dynamics can be solved exactly, the filtering
problem is still non-trivial since correlations between the modes
$\{u_{n+1}^{(j)}\}_j$ will arise from the assimilation of observations~\eqref{eq:she_obs},
unless the observation operator is of the special form
$H(\cdot) = [H_1(\cdot), \ldots, H_m(\cdot)]^\transpose$ with all operator components of the form
$H_i=\phi_j^*$ for some $j\in \bbN$.

Since the Galerkin and spatial
approximation methods coincide in the linear setting,
meaning $\Psi^\ell  = \widehat \Psi^\ell$,
it holds by~\eqref{eq:spatialRate} that for any $p \ge 2$,
\begin{equation}\label{eq:aNumber}
\|\Psi(u_0) -\widehat \Psi^\ell (u_0) \|_{L^p(\Omega,\cH)} \lesssim 
(1+ \|u_0\|_{L^p(\Omega, \cHO)}) N_\ell^{2(r_1-r_2)}.
\end{equation}
Let us next show that the 
time discretization convergence rate~\eqref{eq:convRF}
is improved from $r_1-r_2$ in the above nonlinear setting
to $1$ in the linear setting. We begin by studying the properties of the sequence
$\{\cP_\ell \widetilde \Psi^{m}(u_0)\}_{m=\ell}^\infty$ for a fixed 
$\ell \in \bbN$.
The $j$-th mode projected difference of coupled
solutions for $m>\ell$ is given by
\[
\langle \cP_\ell(\widetilde \Psi^{m}(u_0) - \widetilde \Psi^{m-1}(u_0)), \phi_j \rangle_{\cK} =
\begin{cases} U^{(j)}_{m,J_{m}} - U^{(j)}_{m-1,J_{m-1}}, & \text{if } j\le N_\ell,\\
  0& \text{otherwise},
  \end{cases}
\]
and the difference can be bounded as follows:
\begin{lemma}\label{lem:lemmaConv}
  Consider the SPDE~\eqref{eq:she} with $f(u)=u$, and other
  assumptions as stated in section~\ref{subsec:she}.
  Then for any $\bar u_0 \in L^2(\Omega, \cH)$
  and $m \in \bbN$, the sequence 
  \[
  I_{m,j} \coloneq \left\langle \cP_{m-1} \prt{\widetilde \Psi^{m}(\bar u_0) - \widetilde \Psi^{m-1}(\bar u_0)}, \phi_j \right\rangle_{\cK}, \quad j =1,2,\ldots
  \]
   can be split into three parts
   \[
   I_{m,j}= I_{m,j,1} +I_{m,j,2} + I_{m,j,3},
   \]
   where $I_{m,j,1}, I_{m,j,2},$ and $I_{m,j,3}$ for every $j=1,2,\ldots$
   is a triplet of mutually independent random variables and
  $I_{m,j,1} = I_{m,j,2} = I_{m,j,3}=0$ for all $j> N_{m-1}$.
  Furthermore, there exists a constant $c>0$ that depends on
  $T>0$ and $\lambda_1>1$ such that for any 
  $m \in \bbN$ and all $j \le N_{m-1}$,
  \[
  |I_{m,j,1}|  \le c |\bar u_0^{(j)}| \Delta t_m,
  \]
  and $I_{m,j,2}$ and $I_{m,j,3}$ are mean zero-valued Gaussians with variance bounded
  by 
  \begin{equation}\label{eq:varBound}
  \max\prt{ \E{I_{m,j,2}^2},  \E{I_{m,j,3}^2}} \le c \frac{\Delta t_m^2}{\lambda_{j}^{1+2b}}.
  \end{equation}
   \end{lemma}
\begin{proof}
See Appendix~\ref{app:proofLemmaConv}.
\end{proof}

By Lemma~\ref{lem:lemmaConv} and Assumption~\ref{ass:spde},
there exists a $C>0$ depending on $p$, $T$, $\lambda_1$ and $b +1/4-r_1$
such that for any $m > \ell$ and $u_0 \in \cap_{p\ge 2} L(\Omega, \cHO)$,
\begin{equation}\label{eq:impBound}
\begin{split}
  & \norm{\cP_\ell\prt{ \widetilde \Psi^m(u_0) - \widetilde \Psi^{m-1}(u_0) }}_{L^p(\Omega,\cH)}^2 \leq
  \norm{\cP_{m-1}\prt{ \widetilde \Psi^m(u_0) - \widetilde \Psi^{m-1}(u_0) }}_{L^p(\Omega,\cH)}^2 \\
  &\leq \norm{\sum_{j=1}^{\infty} (I_{m,j,1} +I_{m,j,2}+I_{m,j,3}) \phi_j }_{L^p(\Omega,\cH)}^2\\
  &\le 3 \norm{ \sum_{j=1}^{\infty} I_{m,j,1} \phi_j }_{L^{p}(\Omega,\cH)}^2
  + 3 \norm{ \sum_{j=1}^{\infty} I_{m,j,2} \phi_j }_{L^{p}(\Omega,\cH)}^2
  +3 \norm{ \sum_{j=1}^{\infty} I_{m,j,3} \phi_j }_{L^{p}(\Omega,\cH)}^2\\
  &\le 3c^2\Delta t_m^2 \norm{ u_0 }_{L^{p}(\Omega,\cH)}^2 +
   3\norm{ \sum_{j=1}^{\infty} I_{m,j,2}^2 \langle \phi_j,\phi_j \rangle_\cH  }_{L^{p/2}(\Omega)} 
   + 3\norm{ \sum_{j=1}^{\infty} I_{m,j,3}^2 \langle \phi_j,\phi_j \rangle_\cH  }_{L^{p/2}(\Omega)} \\
  & \le 3c^2\Delta t_m^2 \norm{ u_0 }_{L^{p}(\Omega,\cH)}^2 +
   3 \sum_{j=1}^{\infty} \prt{\norm{ I_{m,j,2}^2}_{L^{p/2}(\Omega)} + \norm{ I_{m,j,2}^2}_{L^{p/2}(\Omega)}}
   \lambda_j^{2r_1}\\
   &\le  3c^2 \Delta t_m^2\prt{ \norm{ u_0 }_{L^{p}(\Omega,\cH)}^2 + 2\sum_{j=1}^{\infty} \lambda_j^{2(r_1-b)-1}}\\
 & \le  C\prt{1+\norm{ u_0 }_{L^{p}(\Omega,\cH)}}^2 \Delta t_{m}^2.
 \end{split}
\end{equation}
Here, the sixth inequality follows from $I_{m,j,2}$ and
$I_{m,j,3}$ being mean-zero-valued Gaussians with variance bounded
by~\eqref{eq:varBound}, which implies that for any $p\ge 2$,
there exists a constant $C>0$ depending on $p$ such that
\[
\max_{r \in \{2,3\}} \norm{ I_{m,j,r}^2}_{L^{p/2}(\Omega)}= \max_{r \in \{2,3\}} \norm{ I_{m,j,r}}_{L^{p}(\Omega)}^2  \le C \frac{\Delta t_m^2}{\lambda_j^{1+2b}}
\]
holds for all $j \in \bbN$.
And the last inequality follows from the assumption $r_1<b+1/4$, which implies that
$2(r_1-b)-1<-1/2$ and hence 
\[
  \sum_{j=1}^\infty \lambda_j^{2(r_1-b)-1} <  \sum_{j=1}^\infty (j^2)^{2(r_1-b)-1} <\infty.
\]

From inequality~\eqref{eq:impBound} we deduce that
$\{\cP_\ell \widetilde \Psi^{m}(u_0)\}_{m=\ell}^\infty$ is $L^p(\Omega, \cH)$-Cauchy
and that there exists a constant
$C>0$ depending on $p$, $T$, $\lambda_1$ and $b +1/4-r_1$
such that
\begin{equation}\label{eq:timeDiscRate}
\begin{split}
\|\cP_\ell \Psi(u_0) - \widetilde \Psi^\ell(u_0) \|_{L^p(\Omega, \cH)}
& \leq \sum_{m=\ell+1}^\infty \|\cP_\ell(\widetilde \Psi^{m}(u_0) - \widetilde \Psi^{m-1}(u_0) ) \|_{L^p(\Omega, \cH)}\\
& \leq C (1+\|u_0\|_{L^p(\Omega, \cH )})  \sum_{m=\ell+1}^\infty  \Delta t_{m} \\
& \leq C (1+\|u_0\|_{L^p(\Omega, \cH )})  \Delta t_{\ell}\sum_{k=1}^\infty  2^{-k} \\
& = C(1+\|u_0\|_{L^p(\Omega, \cH )})  J_\ell^{-1}.
\end{split}
\end{equation}
In view of the preceding inequality and~\eqref{eq:aNumber}
we obtain the following $L^p$-strong convergence rate for the fully
discrete scheme:
\begin{theorem}
  Consider the SPDE~\eqref{eq:she} with $f(u)=u$ and other assumptions as
  stated in section~\ref{subsec:she}. Then for all $p\ge2$ and $\ell \in \bbN\cup\{0\}$,
  there exists a $C>0$ such that
\begin{equation}\label{eq:assIError}
\|\Psi(u_0) - \widetilde \Psi^\ell(u_0) \|_{L^p(\Omega, \cH)}  \leq 
C (1+ \|u_0\|_{L^p(\Omega, \cHO)})(N_\ell^{2(r_1-r_2)} + J_\ell^{-1}), \quad \forall \ell \ge 0,
\end{equation}
where $C$ depends on $r_1,r_2$ and $p$, but not on $\ell$.  
\end{theorem}
\begin{remark}
  To the best of our knowledge, the $L^p$-strong time-discretization
  convergence rate~\eqref{eq:timeDiscRate} is an
  improvement of the literature in two ways. First, for $p=2$,
  it is slightly higher than $\cO(\log(\Delta t^{-1}) \Delta
  t)$, which is the best rate in the literature, cf.~\cite{jentzen2009}.
  And second, this is the first proof of order 1 $L^p$-strong time-discretization
  convergence rate for any $p\ge 2$.
\end{remark}

\smallskip

\subsubsection{Error equilibration}
\label{subsec:errorEquilibration}
The temporal and spatial discretization errors of~\eqref{eq:assIError}
are equlibrated through determining the base $\kappa >1$ that induces a sequence
$\{N_\ell = N_0 \kappa^\ell\}$ such that 
$N_\ell^{2(r_1-r_2)} \eqsim J_\ell^{-1} \eqsim 2^{-\ell}$. The solution is
$\kappa = 2^{(r_2-r_1)/2}$, which yields the following $L^p$-strong
convergence rate in~\eqref{eq:assIError}:
\[
\|\Psi(u_0) - \widetilde \Psi^\ell(u_0) \|_{L^p(\Omega, \cH)}  \lesssim
(1+ \|u_0\|_{L^p(\Omega, \cHO)}) h_\ell^{2(r_2-r_1)}. 
\]
In view of Assumption~\ref{ass:mlrates},
MLEnKF with the fully-discrete approximation method
yields the  convergence rate
$\beta = 4(r_2-r_1)$ 
and the computational cost rates $\gamma_x = 1$ 
and $\gamma_t = 2(r_2-r_1)$ in the considered linear setting.

\begin{remark} [MLEnKF time]
Note that one could consider applying the SDE version of
\cite{ourmlenkf} to a fixed finite approximation of the SPDE.
However, in this case we would be incurring a fixed baseline cost
associated to that discretization.  In comparison to using a single
level method, there would be a gain in efficiency, as a result of
using the multilevel identity with respect to the time discretization.
But this would be still substantially less efficient than accounting
also for the spatial approximation in the multilevel method, as we do
in the method considered here.
\end{remark}

\section{Numerical examples}
\label{sec:numex}

In this section we present numerical performance studies of EnKF and
MLEnKF applied to two different filtering problems with underlying
dynamics given by the SPDE~\eqref{eq:she}. In the first example, the
reaction term of the SPDE is linear, and in the second example we
consider a nonlinear, and thus more challenging, reaction term.

\subsection{Discretization parameters and the relationship between
  computational cost and accuracy}
\label{subsec:discretizationParams}
If we neglect the logarithmic factor in~\eqref{eq:lperror},
as is motivated by Remark~\ref{rem:logTermIsFishy},
then Theorems~\ref{thm:main} and~\ref{thm:mainEnKF}
respectively imply the following relations between mean squared error (MSE)
and computational cost
\begin{equation*}
  \cost{\hat \mu^{\rm ML}_n [\varphi] }^{\min(1,\beta/(d \gamma_x + \gamma_t))}
  \|\hat \mu^{\rm ML}_n [\varphi] - \hat \mfmu_n [\varphi] \|_2^2 \lesssim
\begin{cases}
  1 & \text{if } \beta \neq d\gamma_x +\gamma_t, \\
  L^3 & \text{if } \beta = d\gamma_x +\gamma_t,
\end{cases}
\end{equation*}
and
\begin{equation*}
\cost{\hat \mu^{\rm MC}_n [\varphi]}^{\beta/(\beta +d \gamma_x +\gamma_t)}
  \|\hat \mu^{\rm MC}_n [\varphi] - \hat \mfmu_n [\varphi] \|_2^2
      \lesssim 1.
\end{equation*}
In other words, 
\begin{equation}\label{eq:mlMseCost}
  \|\hat \mu^{\rm ML}_n [\varphi] - \hat \mfmu_n [\varphi] \|_2^2
   \lesssim \begin{cases}
    \cost{\hat \mu^{\rm ML}_n [\varphi]}^{-1} & \text{if } \beta > d\gamma_x + \gamma_t,\\
    L^3 \cost{\hat \mu^{\rm ML}_n [\varphi]}^{-1} & \text{if } \beta = d\gamma_x + \gamma_t,\\
    \cost{\hat \mu^{\rm ML}_n [\varphi]}^{-\beta/(d\gamma_x+\gamma_t)} & \text{if } \beta < d\gamma_x + \gamma_t,
  \end{cases}
\end{equation}
and 
\begin{equation}\label{eq:mcMseCost}
  \|\hat \mu^{\rm MC}_n [\varphi] - \hat \mfmu_n [\varphi] \|_2^2
    \lesssim \cost{\hat \mu^{\rm MC}_n[\varphi]}^{-\beta/(\beta+ d\gamma_x + \gamma_t)}.
\end{equation}
For all test
problems, we use the observation-time interval $T=1/2$, $N = 40$ observation times, $N_\ell=
2^{\ell+2}$, and, when relevant $J_\ell = 2^{\ell+2}$ (i.e., for the fully-discrete
numerical method). The approximation error, which we refer to as the
mean squared error (MSE),
is defined as the sum of the squared QoI error over the
observation times and averaged over 100 realizations of the respective
filtering methods.  That is,
\[
\text{MSE(MLEnKF)} \coloneq \frac{1}{100}\sum_{i=1}^{100} \sum_{n=0}^{N}
  \abs{\hat\mu^{\rm ML}_{n,i} [\varphi] - \hat \mfmu_n [\varphi]}^2 \approx \sum_{n=0}^{N}
  \|\hat\mu^{\rm ML}_{n} [\varphi] - \hat \mfmu_n [\varphi]\|_2^2,
\]
where $\{\hat \mu^{\rm ML}_{\cdot,i} [\varphi]\}_{i=1}^{100}$ is a sequence of i.i.d.~QoI
evaluations induced from i.i.d.~realizations of the MLEnKF. And similarly,
\[
\text{MSE(EnKF)} \coloneq \frac{1}{100}\sum_{i=1}^{100} \sum_{n=0}^{N}
\abs{\hat \mu^{\rm MC}_{n,i} [\varphi] - \hat \mfmu_n [\varphi]}^2
\approx \sum_{n=0}^{N} \|\hat \mu^{\rm MC}_{n} [\varphi] - \hat \mfmu_n [\varphi]\|_2^2.
\]
In the examples below we numerically verify that the considered
numerical methods respectively fulfill~\eqref{eq:mlMseCost}
and~\eqref{eq:mcMseCost}, when the above computational cost expressions
are replaced/approximated by the wall-clock runtime of the
computer implementations of the respective methods. More precisely,
we numerically verify that the following approximate asymptotic inequalities hold:
\begin{equation}\label{eq:mlMseRuntime}
\text{MSE(MLEnKF)}  \lessapprox \begin{cases}
    \text{Runtime(MLEnKF)}^{-1}  & \text{if } \beta > d\gamma_x + \gamma_t,\\
    L^3 \text{Runtime(MLEnKF)}^{-1}  & \text{if } \beta = d\gamma_x + \gamma_t,\\
    \text{Runtime(MLEnKF)}^{-\beta/(d\gamma_x+\gamma_t)} & \text{if } \beta < d\gamma_x + \gamma_t,
  \end{cases}
\end{equation}
and
\begin{equation}\label{eq:mcMseRuntime}
  \|\hat \mu^{\rm MC}_n [\varphi] - \hat \mfmu_n [\varphi] \|_2^2
    \lessapprox \text{Runtime(EnKF)}^{-\beta/(\beta+ d\gamma_x + \gamma_t)}.
\end{equation}

\subsection{Linear filtering problems}
\label{subsec:linearFilteringEx}

We consider the filtering problem in Section~\ref{subsec:filtering} with the
linear forcing $f(u)=u$ in the underlying dynamics~\eqref{eq:she},
smoothing parameter $b=1/2$, approximation space parameters
$r_1=1/4+\upsilon$ and $r_2 = 3/4-\upsilon$ with $\upsilon = 10^{-4}$,
observation functional
\[
H =\delta_{0.5} = \sqrt{2} \sum_{j=1}^\infty \sin(j\pi/2) \phi_j^*,
\]
observation noise parameter $\Gamma=0.5$,
QoI
\begin{equation}\label{eq:QoIEx1And2}
\varphi = 1^* = \sum_{j \in \bbN} \frac{\sqrt{2}(1-\cos(j\pi))}{j\pi} \phi_{j}^*, 
\end{equation}
and initial data
\[
u_0(x) = 1-2\left|x-\frac{1}{2}\right| =  \sum_{j\in \bbN} (-1)^{j-1} \frac{4\sqrt{2}}{((2j-1)\pi)^2}  \phi_{2j-1}(x).
\]
We note that $H, \varphi \in \cH^*$ and $u_0 \in \cHO$.
Figure~\ref{fig:linearExampleSim} illustrates one exact-in-time
simulation of the SPDE.
\begin{figure}[htbp]
  \centering
  \includegraphics[width=0.96\textwidth]{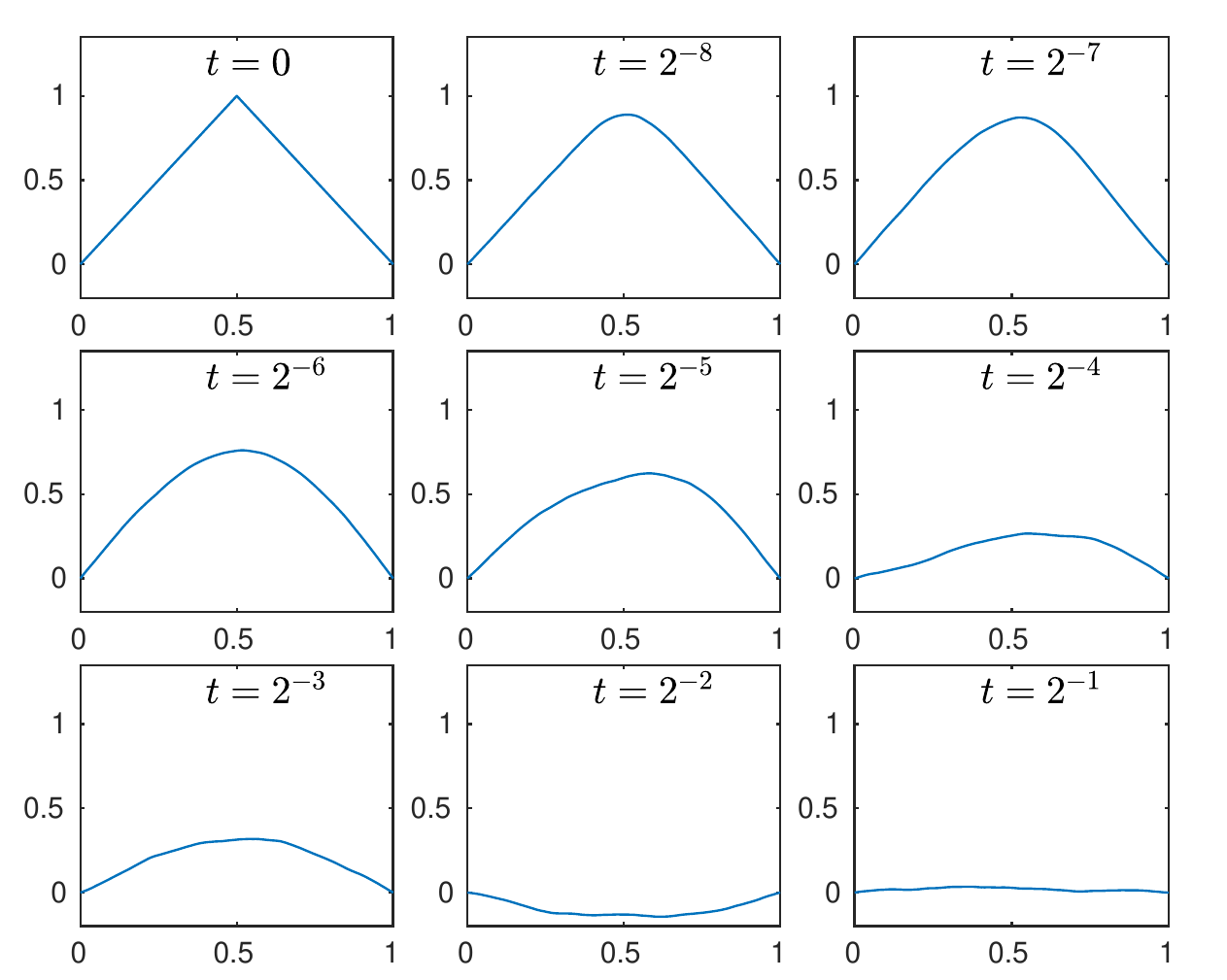}
  \caption{Exact-in-time simulation of the SPDE in
    Section~\ref{subsec:linearFilteringEx}
    over one observation-time interval with spatial resolution
    $N_{10}=2^{12}$.}
  \label{fig:linearExampleSim}
\end{figure}

By the approximation $\beta = 2(r_2-r_1) \approx 2$, application of
the error equilibration in Section~\ref{subsec:errorEquilibration}
yields ($d\gamma_x = 1$,$\gamma_t = 1$) for the fully-discrete method
and ($d\gamma_x = 1$, $\gamma_t = 0$) for the spatially-discrete
method.  Figure~\ref{fig:ex2MseVsCostSpat} and the left subfigure of
Figure~\ref{fig:ex2MseVsCostFull} display the runtime-to-MSE
performance for the spatially-discrete and fully-discrete methods,
respectively.
\begin{figure}[]
  \centering
  \includegraphics[width=0.58\textwidth]{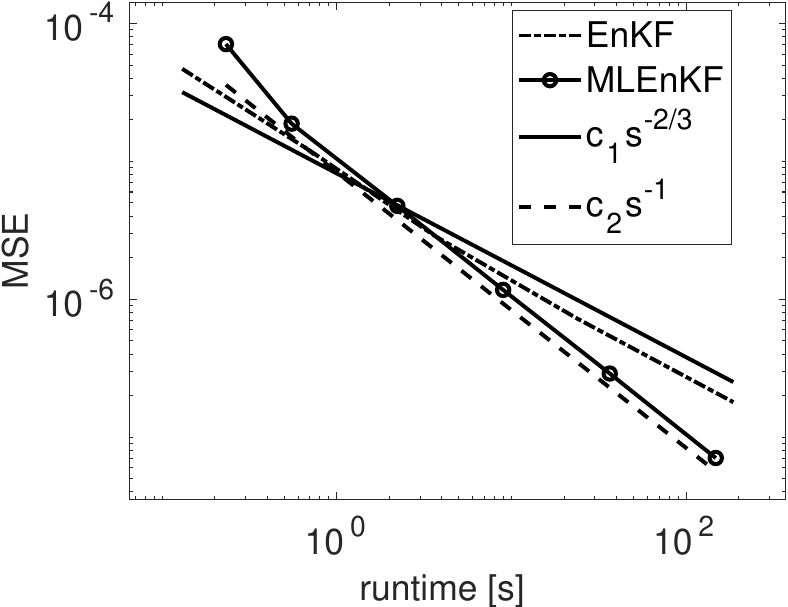}
  \caption{Runtime-to-MSE comparison for the filtering problem in
    Section~\ref{subsec:linearFilteringEx} using the spatially-discrete
    method.}
  \label{fig:ex2MseVsCostSpat}
\end{figure}
The right subfigure of Figure~\ref{fig:ex2MseVsCostFull} displays the
graph of
\[
(\text{Runtime(MLEnKF)}, \, \text{MSE(MLEnKF)}\times \text{Runtime(MLEnKF)}/L^3),
\]
(where $\text{MSE(MLEnKF)}$ in the second argument denotes the ``MSE'' obtained for a given ``Runtime''). 
The numerical observations are consistent with the approximate theoretical
predictions~\eqref{eq:mlMseRuntime} and~\eqref{eq:mcMseRuntime}.
\begin{figure}[htbp]
  \centering
  \includegraphics[width=0.51\textwidth]{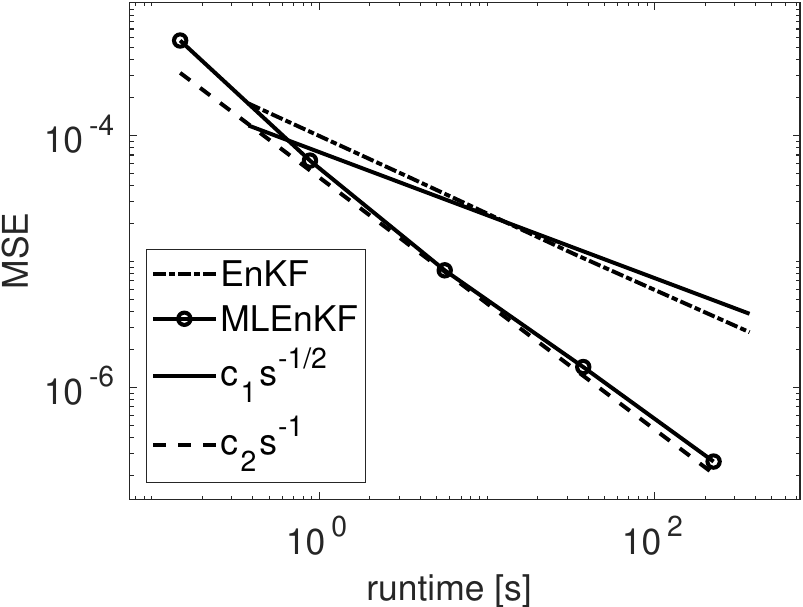}
  \includegraphics[width=0.48\textwidth]{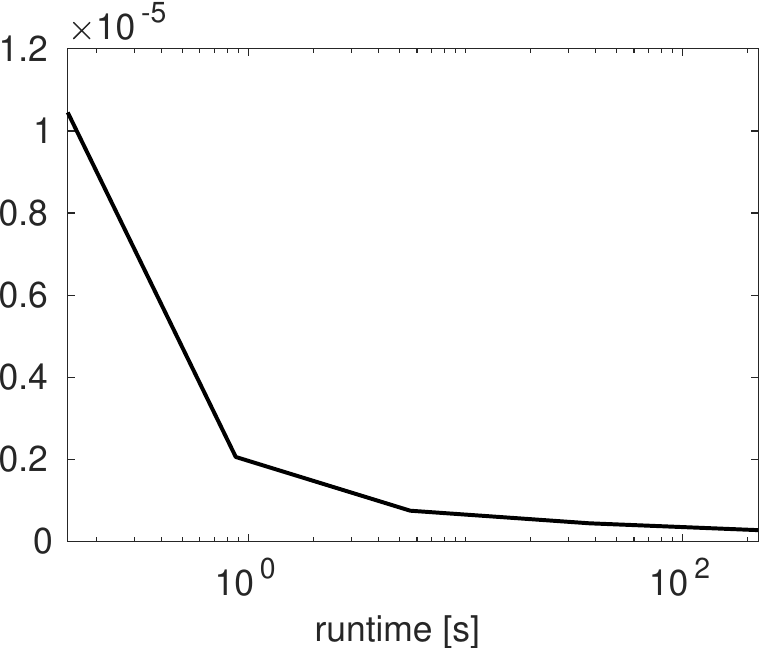}
  \caption{ Left: Runtime-to-MSE comparison for the filtering problem in
    Section~\ref{subsec:linearFilteringEx} using the fully-discrete
    numerical method.  Right: Graph of $(\text{Runtime(MLEnKF)}, \,
    \text{MSE(MLEnKF)}\times \text{Runtime(MLEnKF)}/L^3)$
    for the fully-discrete numerical method. 
        Note the y-axis here, and in future such plots, has linear scaling.}
  \label{fig:ex2MseVsCostFull}
\end{figure}

The reference-solution sequence $\{\hat \mfmu_n [\varphi] \}_{n=1}^N$
that is needed to estimate the MSE in the above figures, is
approximated by Kalman filtering the subspace $\cH_{12} \subset \cH$,
which is an $N_{12} = 2^{14}$-dimensional subspace.  This yields an
accurate approximation, since when the underlying
dynamics~\eqref{eq:she} is linear with Gaussian additive noise, the
full-space Kalman filter distribution equals the reference MFEnKF
distribution $\hat \mfmu$. Furthermore, EnKF and MLEnKF solutions are
computed with ensemble particles at no higher spatial resolution than
$\cH_{9}$ in the cost-to-accuracy studies.

\subsection{A nonlinear filtering problem}
\label{subsec:nonlinearFilteringEx}
We seek the mild solution to the following nonlinear SPDE with periodic boundary conditions 
\begin{equation}
\begin{split}
du &= \prt{(\Delta-I) u + \sin(\pi u)} dt + B  dW, \qquad (t,x) \in (0,T] \times (0,1),\\
u(0,x) &=  4(x-1/2)^2\\
u(t,0) &= u(t,1), \mspace{240mu}  t\in (0,T],
\end{split}
\label{eq:spdeNonlinear}
\end{equation}
where $W$ and $B$ are described below.  Here, the operator $-A =
(I-\Delta)$ is defined as a mapping $A: H^2(0,1) \cap
H^1_{\text{per}}(0,1)\to \cK = L^2(0,1)$, where $H^1_{\text{per}}(0,1)
:= \{f \in H(0,1) \mid (f-f(0)) \in H^1_0(0,1) \}$.  The periodic
boundary condition is different from the zero-valued boundary
condition in~\eqref{eq:she}, and, in order to
spectrally decompose $-A$, we now express the base-space
$\cK =L^2(0,1)$ by the closure of the span of the Fourier basis
\begin{equation}\label{eq:currentBasis}
  \phi_{k}(x) = \begin{cases} 1, & k=1,\\
    \sqrt{2} \cos(2k\pi x), & k=2,4,6\ldots\,,\\
    \sqrt{2} \sin(2(k-1)\pi x),  & k=3,5,7,\ldots
  \end{cases}
\end{equation}
The operator $-A$ is spectrally decomposed by
\[
-A\phi_k = \lambda_k \phi_k
\]
with
\[
  \lambda_k = \begin{cases}
    1, & k=1,\\
    1+(2k\pi)^2, & k=2,4,6, \ldots\, ,\\
    1 + (2(k-1)\pi)^2 & k=3,5,7, \ldots
    \end{cases}
\]
As in Section~\ref{subsec:she}, we introduce the family of
Hilbert spaces parametrized in $r\in \bbR$
\begin{equation*}
  \begin{split}
\cK_r \coloneq D((-A)^{r}) = \Big\{v:[0,1] \to \bbR \mid & v \text{ is }
\mathcal{B}([0,1])/\mathcal{B}(\bbR)\text{-measurable} 
\\& \text{and } \sum_{j \in \bbN}
\lambda^{2r}_j \abs{\langle \phi_j, v \rangle_{\cK}}^2 < \infty \Big\},
\end{split}
\end{equation*}
with norm $\|\cdot\|_{\cK_r} \coloneq \|(-A)^{r}(\cdot)\|_{\cK}$. As
smoothing operator $B$, we consider~\eqref{eq:bDef} with parameter
$b=1/4$, where $W$ denotes an
$I_{\cK}$-cylindrical Wiener process (both $B$ and $W$ are of course
expanded in the currently considered basis~\eqref{eq:currentBasis}).
We consider the approximation spaces
\[
 \cH = \cK_0  \quad \text{and} \quad \cHO = \cK_{(1-\nu)/2},
\]
where $\nu=10^{-4}$, the QoI~\eqref{eq:QoIEx1And2}
and the observation operator
\[
H = 1_{x>0.5}^* = \frac{1}{2}\phi_0^* + \sum_{k=1}^\infty \frac{\sqrt{2}}{\pi k} \phi_{2k+1}^*.
\]
The spectral representation of the initial data
\[
u(0,\cdot) = \frac{1}{3} + \sum_{k=1} \frac{2\sqrt{2}}{(\pi k)^{2}} \phi_{2k}(\cdot)
\]
implies that $u(0,\cdot) \in \cK_{(1-\nu)/2}$.
Figure~\ref{fig:nonlinearExampleSim} illustrates one simulation of
the SPDE by the numerical scheme described below. 
\begin{figure}[htbp]
  \centering
  \includegraphics[width=0.96\textwidth]{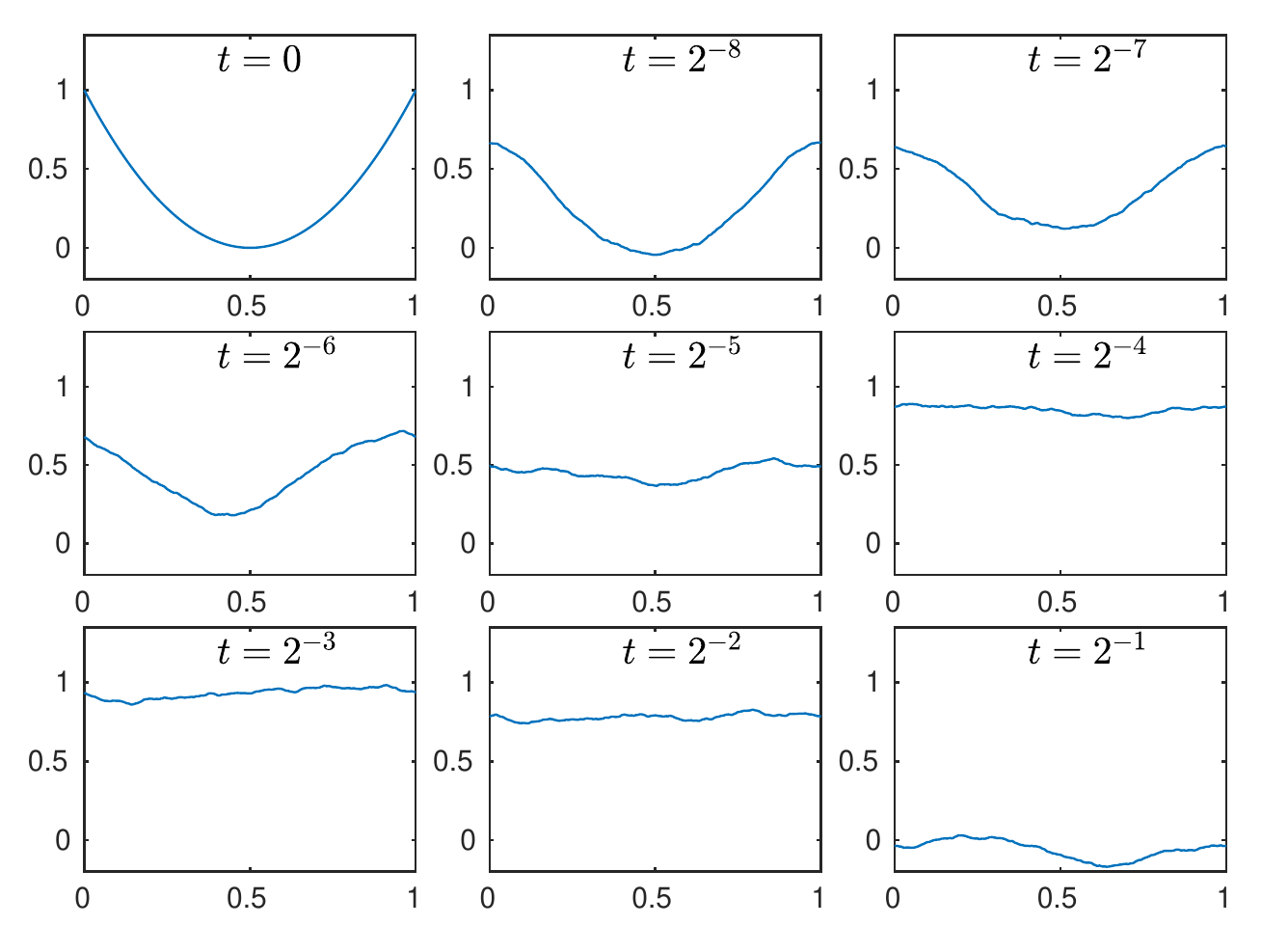}
  \caption{ Simulation of the SPDE~\eqref{eq:spdeNonlinear} over one
    observation-time interval by the numerical scheme
    in Section~\ref{subsec:numSchemeNonlinear} on resolutions
    $N_{10}=J_{10}=2^{12}$.}
  \label{fig:nonlinearExampleSim}
\end{figure}

By the Lipschitz-continuity of the reaction term, it follows that
Assumption~\ref{ass:spde} is fulfilled. Moreover, the well-posedness theory for the
zero-valued boundary condition for the SPDE~\eqref{eq:she} extends to
the current setting, and so does the theory for the fully-discrete
exponential Euler method in Section~\ref{subsec:fullyDiscrete},
cf.~\cite{jentzenNotes2015}.

\subsubsection{Numerical scheme}
\label{subsec:numSchemeNonlinear}
We apply the coupled fully
discrete approximation method described in
Section~\ref{subsec:fullyDiscrete}, but, due to the nonlinearity of the reaction
term $f(u)=\sin(\pi u)$, the spectral representation in the coupled
scheme~\eqref{eq:modeUFine} and~\eqref{eq:modeUCoarse} needs to be
approximated. Namely, the approximation of $[(f_\ell(U_{\ell,k}))^{(1)}, \ldots,
  (f_\ell(U_{\ell,k}))^{(N_\ell)}]$
is obtained by application of the fast Fourier transform (FFT) as follows:
\begin{enumerate}

\item[1.] Given the spectral representation $[U_{\ell, k}^{(1)}, \ldots, U_{\ell, k}^{(N_\ell)}]$ 
  compute the physical-space-on-uniform-mesh representation by the inverse FFT
  \[
    [U_{\ell, k}(0), U_{\ell, k}(1/N_\ell), \ldots, U_{\ell, k}(1 - N_\ell^{-1})]  = \text{IFFT}[U_{\ell, k}^{(1)}, \ldots, U_{\ell, k}^{(N_\ell)}].
  \]

\item[2.] Evaluate the nonlinear reaction term in physical space and approximate the
  spectral representation by FFT
  \begin{multline*}
    [(f_\ell(U_{\ell,k}))^{(1)}, \ldots, (f_\ell(U_{\ell,k}))^{(N_\ell)}]\\
    \approx \text{FFT}[f(U_{\ell, k}(0)), f(U_{\ell, k}(1/N_\ell)), \ldots, f(U_{\ell, k}(1 - N_\ell^{-1}))].
  \end{multline*}
\end{enumerate}
The spectral approximation of the coarse-level reaction term is
obtained analogously. Due to the FFT approximation error in step
2.~above, we cannot directly obtain the rate parameter $\beta$ from
the analysis in Section~\ref{subsec:fullyDiscrete}.  To infer $\beta$,
we instead perform numerical studies of the
$L^p(\Omega,\cH)$-convergence rate of the coupled-level difference of
the FFT-based fully-discrete method $\widetilde \Psi^\ell(u) -
\widetilde \Psi^{\ell-1}(u)$ towards $0$, where the expectation is estimated
with the Monte Carlo method with $M=10^5$ samples:
\begin{equation}\label{eq:testNonlinearRatesMc}
  \prt{\frac{1}{M} \sum_{i=1}^{M} \|\widetilde \Psi^\ell(u_0; \omega_i)
    - \widetilde \Psi^{\ell-1}(u_0; \omega_i)\|_{\cH}^p}^{1/p} \approx
  \|\widetilde \Psi^\ell(u) - \widetilde \Psi^{\ell-1}(u)\|_{L^p(\Omega,\cH)}
\end{equation}
Recalling that $h_\ell^{-1} \eqsim N_\ell \eqsim J_\ell= 2^{2+\ell}$
for the numerical solver $\widetilde \Psi^\ell$, we infer from the
results of the numerical study~\eqref{eq:testNonlinearRatesMc}, which
is provided in Figure~\ref{fig:testNonlinearRates}, that
\begin{equation}\label{eq:testNonlinearRates}
\|\widetilde \Psi^\ell(u) - \Psi(u)\|_p \lessapprox h_\ell^{\beta/2}
\end{equation}
with $\beta =2$. Further numerical studies, which we do not include here,
indicate that the right hand side of~\eqref{eq:testNonlinearRates} may
be decomposed into $\cO(N_\ell^{-1} + J_\ell^{-1})$. On the basis of
these observations, the configuration of discretization
parameters for this problem, $N_\ell \eqsim J_\ell$, is in alignment
with the efficiency-optimized error equilibration strategy in
Section~\ref{subsec:errorEquilibration}.
\begin{figure}[htbp]
  \centering
  \includegraphics[width=0.68\textwidth]{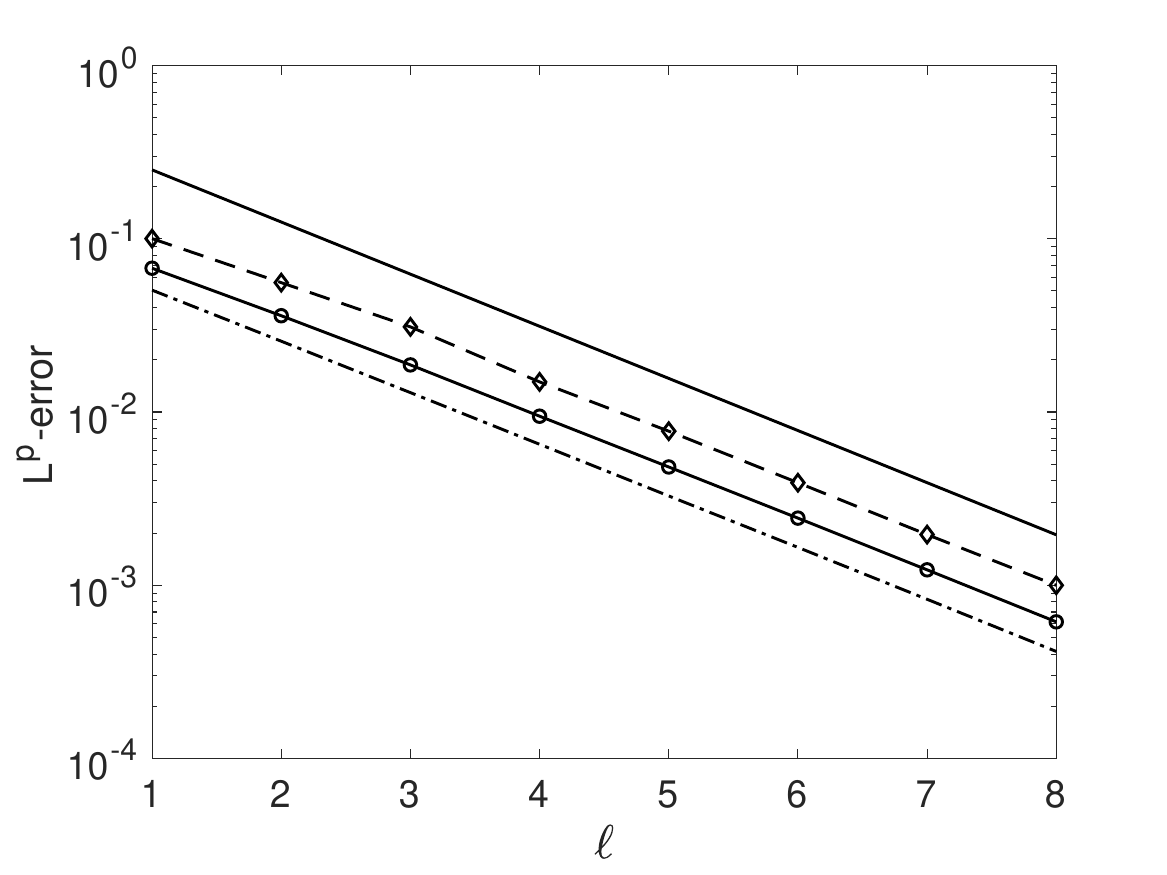}
  \caption{
    Numerical estimates of the error
    \mbox{$\|\widetilde \Psi^\ell(u_0) - \widetilde \Psi^{\ell-1}(u_0)\|_p$}
    by the Monte Carlo method~\eqref{eq:testNonlinearRatesMc}
    for $p=2$ (dash-dot), $p=4$ (solid-circle), and $p=8$ (dash-diamond).
    The solid line represents the reference function $f(\ell) = 2^{-(\ell+1)}$.
  }
  \label{fig:testNonlinearRates}
\end{figure}
The left subfigure in Figure~\ref{fig:nonlinearMseVsCostFull} displays
the results of the runtime-to-MSE studies of EnKF and
MLEnKF.
\begin{figure}[htbp]
  \centering
  \includegraphics[width=0.51\textwidth]{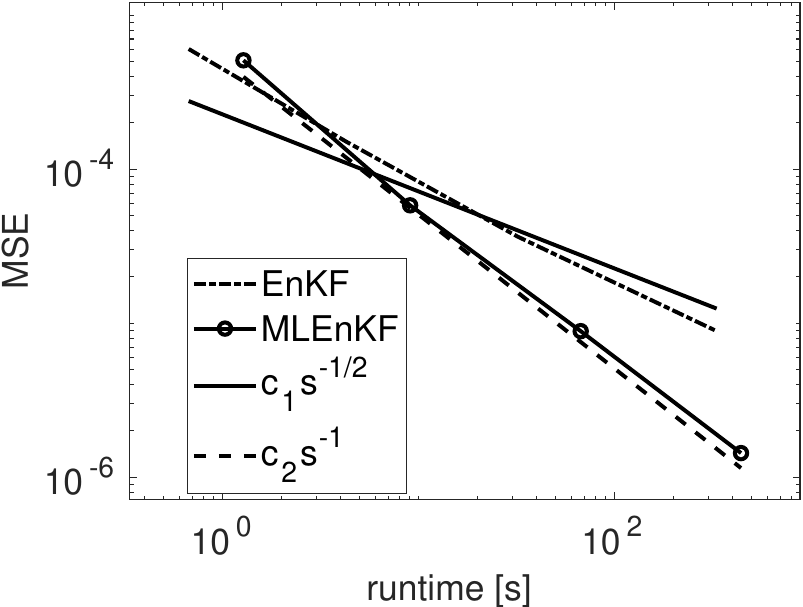}
  \includegraphics[width=0.48\textwidth]{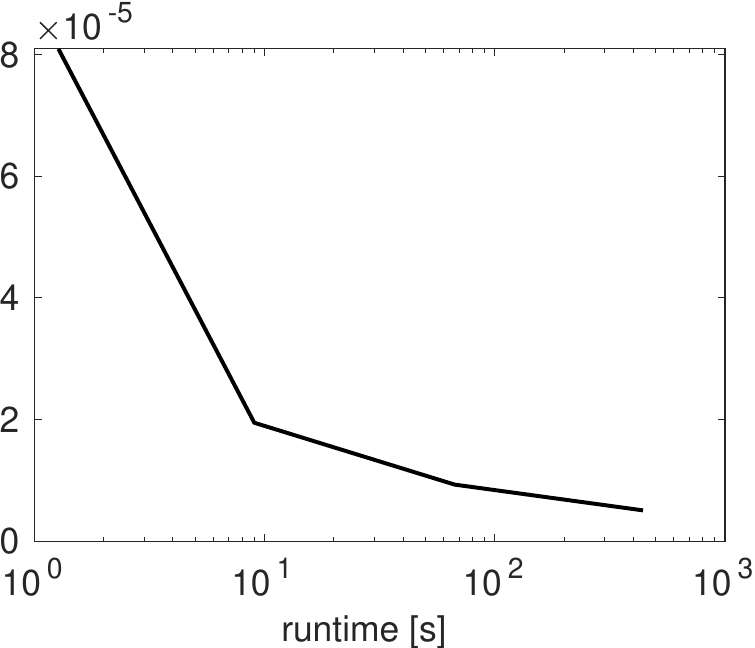}
  \caption{ Left: Runtime-to-MSE for the nonlinear filtering
    problem in Section~\ref{subsec:numSchemeNonlinear} using the fully-discrete method.
    Right: Graph of $(\text{Runtime(MLEnKF)}, \,
    \text{MSE(MLEnKF)}\times \text{Runtime(MLEnKF)}/L^3)$ for the
    fully-discrete method.  }
  \label{fig:nonlinearMseVsCostFull}
\end{figure}
As pseudo-reference solution, we use the approximation
\[
 \frac{1}{200} \sum_{i=1}^{200} \mu^{\rm{ML}}_{n,i}[\varphi]  \approx \hat{\bar{\mu}}_n[\varphi], \quad n=1,2,\ldots, 40,
\]
with the MLEnKF estimator $\mu^{\rm{ML}}_{n,i}[\varphi]$ here being
computed on a finer resolution than all those considered in the
runtime-to-MSE study. The right subfigure in
Figure~\ref{fig:nonlinearMseVsCostFull} displays the graph of
\[
(\text{Runtime(MLEnKF)}, \, \text{MSE(MLEnKF)}\times \text{Runtime(MLEnKF)}/L^3).
\]
Once again, the numerical observations are consistent with the
theoretical asymptotical behavior predicted by~\eqref{eq:mlMseRuntime}
and~\eqref{eq:mcMseRuntime}.

\begin{remark}[MLEnKF versus Multilevel particle filters]
To the best of our knowledge, there does not exist a general multilevel
particle filter for SPDE to this date. 
When the effective dimension on level $\ell$ is $N_\ell$, 
the general requirement for particle filters  
is that the ensemble size on that level is bounded
from below by $ce^{N_\ell}$ particles, for some constant $c>0$.  
Effective dimension refers to the dimension of the space 
over which importance sampling needs to be performed 
\cite{bengtsson2008curse, chatterjee2015sample, agapiou2017importance}.
For example, in the case of full observations, 
the effective dimension can be equal to the state-space dimension.
For MLEnKF, on the other hand, the level $\ell$ ensemble size is always bounded
from above by $\cO(L^2 N_\ell^{(\beta-d\gamma_x-\gamma_t)/(2d)} )$, 
even with full observations. The set of
MLEnKF-tractable problems is therefore substantially larger than
the set of problems tractable by particle filters. 
\end{remark}

\section{Conclusion}
\label{sec:conclusion}

We have presented the design and analysis of a multilevel EnKF method
for infinite-dimensional spatio-temporal processes depending on a
hierarchical decomposition of both the spatial and the temporal
parameters.  We have proved theoretically and provided numerical
evidence that under suitable assumptions, a similar asymptotic
cost-to-accuracy
is obtained for MLEnKF as that one obtains for standard
multilevel Monte Carlo methods.  This result has potential for broad
impact across application areas in which there has been a recent
explosion of interest in EnKF, for example weather prediction and
subsurface exploration.

\appendix

\section{Marcinkiewicz--Zygmund inequalities for separable Hilbert spaces}
\label{app:mzIneq}
In order to prove Lemma~\ref{lem:iidcover}, we will need the following two 
lemmas for extending the Marcinkiewicz--Zygmund inequality from finite-dimensional
state-spaces to separable Hilbert spaces.
\begin{lemma}\cite[Theorem 5.2]{kwiatkowski2014convergence}
\label{lem:mz}
Let $2 \leq p < \infty$ and $X_i \in L^p(\Omega,\cH)$ be \iid samples of $X \in L^p(\Omega,\cH)$. Then 
\begin{equation}\label{eq:LawLargeNum-r}
\|E_M[X] - \E{X}\|_{L^p(\Omega,\cH)} \leq \frac{c_p}{\sqrt{M}} \|X - \E{X}\|_{L^p(\Omega,\cH)}
\end{equation}
where $c_p$ only depends on $p$.
\end{lemma}
\begin{proof}
Let $r_1, r_2, \ldots$ denote a sequence of real-valued \iid random variables with $P(r_i= \pm 1) = 1/2$.
A Banach space $\cK$ is said to be of R-type $q$ if there exists a $c>0$ such that for  every $\bar n \in \bbN$
and for all (deterministic) $x_1, x_2, \ldots, x_{\bar n} \in \cK$,
\[
\E{ \Big\|\sum_{i=1}^{\bar n} r_i x_i \Big\|_{\cK}} \leq c \parenthesis{\sum_{i=1}^{\bar{n}} \|x_i\|_{\cK}^q}^{1/q}.
\]
It is clear that all Hilbert spaces (and for our interest $\cH$, in
particular) are of R-type 2, since their norms are induced by an inner
product. Following the proofs of~\cite[Proposition 2.1 and Corollary
2.1]{woyczynski1980}, let $\{X_i'\}$ denote an additional sequence of
i.i.d.~samples of $X\in L^p(\Omega,\cH)$ for which the collection of
r.v.~$\{X_i\} \cup \{X_i'\}$ also is i.i.d. Introducing the
symmetrization $\widetilde X_i \coloneq (X_i - X_i')$, and noting that
\[
\E{X_i - \E{X}} = \E{ \E{\widetilde X_i \mid X_i} },
\]
we derive by the conditional Jensen's inequality that
\begin{multline*}
    \E{ \hNorm{\sum_{i=1}^{\bar n} X_i - \E{X} }^p} \leq
    \E{ \hNorm{\E{\sum_{i=1}^{\bar n} \widetilde X_i \Bigg| \{X_i\}_{i=1}^{\bar n}}}^p} \\
    \le \E{\hNorm{\sum_{i=1}^{\bar n} \widetilde X_i}^p} 
     = \E{ \hNorm{\sum_{i=1}^{\bar n} r_i \widetilde X_i}^p}
      \leq c \E{ \parenthesis{\sum_{i=1}^{\bar n} \hNorm{\widetilde X_i  }^2}^{p/2}} \\
   \leq c 2^p \, \E{ \parenthesis{\sum_{i=1}^{\bar n} \hNorm{X_i -\E{X}}^2}^{p/2}}.
\end{multline*}
And by another application of H\"older's inequality,
\[
\begin{split}
\E{ \hNorm{\sum_{i=1}^{M} \frac{X_i - \E{X}}{M} }^p} 
&\leq \hat{c} M^{-p} \E{\parenthesis{\sum_{i=1}^{M} \hNorm{X_i-\E{X}}^2}^{p/2}} \\
&\leq \hat{c} M^{-p/2} \E{\hNorm{X-\E{X}}^p}. 
\end{split}
\]
\end{proof}

\begin{lemma}\label{lem:covM-Lr-error}
Let $X,Y \in L^p(\Omega,\cH)$, for some $p \geq 2$. 
Then, for any $1 \leq r,s \leq \infty$ satisfying $1/r + 1/s = 1$, it holds that
\begin{equation*}\label{eq:covM-Lr-error}
\begin{split}
\|\cov_M[X,Y] - \cov[X,Y]\|_{L^p(\Omega,\cH\otimes \cH)}
\leq \frac{c}{\sqrt{M}}
\|X\|_{L^{pr}(\Omega,\cH)} 
\|Y\|_{L^{ps}(\Omega,\cH)}
\end{split}
\end{equation*}
where the upper bound for the constant 
$c = \dfrac{M}{M-1}\bigg(2c_p + \dfrac{c_{pr}c_{ps}+1}{\sqrt{M}}\bigg)$ 
only depends on $r,s$ and $p$.
\end{lemma}

\begin{proof}

Since $\cov[X,Y] = \cov[X-\E{X}, Y-\E{Y}]$ and $\cov_M[X,Y] = \cov_M[X-\E{X},Y-\E{Y}]$, cf.~\eqref{eq:sampleCov}, 
we may without loss of generality assume that $\E{X} = \E{Y} = 0$.
Using the triangle inequality,
\begin{equation*}
\begin{split}
& \frac{M-1}{M}\|\cov_M[X,Y] - \cov[X,Y]\|_p \\
& \leq 
\| E_M[X \otimes Y] - \E{X \otimes Y}\|_p  + \| E_M[X] \otimes E_M[Y] \|_p
 + \frac{1}{M}\|\E{X \otimes Y}\|_{\cH\otimes \cH}.
\end{split}
\end{equation*}
Estimate~\eqref{eq:LawLargeNum-r} and H\"older's inequality yield
\[
\begin{split}
\| E_M[X \otimes Y] - \E{X \otimes Y}\|_p
&\leq \frac{c_p}{\sqrt{M}} \|X\otimes Y - \E{X\otimes Y}\|_p \\
&\leq \frac{2c_p}{\sqrt{M}} \|X\otimes Y \|_p
 \leq \frac{2c_p}{\sqrt{M}} \|X\|_{pr} \|Y\|_{ps}.
\end{split}
\]
Similarly, since $\E{X} = \E{Y} = 0$ by assumption, 
we obtain by \eqref{eq:LawLargeNum-r} and H\"older's inequality
\begin{equation*}
\begin{split}
\| E_M[X] \otimes E_M[Y] \|_{p}
\leq  
\|E_M[X]\|_{pr} 
\|E_M[Y]\|_{ps}
\leq \frac{c_{pr}c_{ps}}{M} 
\|X\|_{pr} 
\|Y\|_{ps}.
\end{split}
\end{equation*}
And, finally, for the last term
\begin{equation*}
\frac{1}{M}\|\E{X \otimes Y}\|_{\cH\otimes \cH}
\leq \frac{1}{M}\|X \otimes Y\|_{L^1(\Omega,\cH\otimes \cH)}
\leq \frac{1}{M} \|X\|_{L^{pr}(\Omega,\cH)} \|Y\|_{L^{ps}(\Omega,\cH)}.
\end{equation*}
\end{proof}

\section{Proof of Lemma~\ref{lem:lemmaConv}} \label{app:proofLemmaConv}

\begin{proof}
Introducing the function $g:(1,\infty) \times (0,\infty) \to \bbR$
defined by
\[
g(\lambda,s) = e^{-\lambda s } + \frac{1-e^{-\lambda s}}{\lambda },
\]
consecutive iterations of the scheme~\eqref{eq:modeUFine}
for $j \le N_m$ yield 
\begin{equation*}\label{eq:twoIterations}
\begin{split}
U^{(j)}_{m,J_m} &=  g(\lambda_j,\Delta t_m) U^{(j)}_{m,J_m-1} + R_{m,J_m-1} \\
&= \prt{g(\lambda_j,\Delta t_m)}^2 U_{m,J_m-2}^{(j)} + g(\lambda_j,\Delta t_m) R_{m,J_m-2}^{(j)}+ R_{m,J_m-1}^{(j)}\\
& = \ldots \\
  & = \prt{g(\lambda_j,\Delta t_m)}^{J_m} U_{m,0}^{(j)} +
  \sum_{k=0}^{J_m-1} \prt{g(\lambda_j,\Delta t_m)}^{J_m-(k+1)} R_{m,k}^{(j)},
\end{split}
\end{equation*}
where we recall that the initial data is given
by $U_{m,0} = \cP_m \bar u_0$ with $\bar u_0 \in L^2(\Omega,\cH)$.
And since $J_{m} = 2J_{m-1}$,
consecutive iterations of the coupled coarse 
scheme~\eqref{eq:modeUCoarse} for $j \le N_{m-1}$ yield
\begin{equation*}\label{eq:twoIterations}
\begin{split}
  U^{(j)}_{m-1,J_{m-1}} &=  g(\lambda_j,\Delta t_{m-1})
  U^{(j)}_{m-1,J_{m-1}-1} + e^{-\lambda_j \Delta t_{m} } R_{m,J_m-2}^{(j)}+R_{m,J_m-1}^{(j)} \\
  &= \prt{g(\lambda_j,\Delta t_{m-1})}^2 U^{(j)}_{m-1,J_{m-1}-2} \\
&  \quad + g(\lambda_j,\Delta t_{m-1})\prt{e^{-\lambda_j \Delta t_{m} } R_{m,J_m-4}^{(j)}
    +R_{m,J_m-3}^{(j)}}\\ 
& \quad + e^{-\lambda_j \Delta t_{m} } R_{m,J_m-2}^{(j)}+R_{m,J_m-1}^{(j)} \\
& = \ldots \\
& = \prt{g(\lambda_j,\Delta t_{m-1})}^{J_{m-1}} U_{m-1,0}^{(j)} \\
&+  \sum_{k=0}^{J_{m-1} -1}  \prt{g(\lambda_j, \Delta t_{m-1})}^{J_{m-1}-(k+1)}
\prt{e^{-\lambda_j \Delta t_{m} } R_{m,2k}^{(j)}+R_{m,2k+1}^{(j)}}.
\end{split}
\end{equation*}
The $j$-th mode final time difference of the coupled solutions
for $ j\le N_{m-1}$ thus becomes
\begin{equation}\label{eq:jthModeDiff}
\begin{split}
&  U^{(j)}_{m,J_m} - U^{(j)}_{m-1,J_{m-1}} =
  \prt{ \prt{g(\lambda_j,\Delta t_m)}^{2J_{m-1}} -
    \prt{ g(\lambda_j,\Delta t_{m-1}) }^{J_{m-1}} } U_{\ell,0}^{(j)} \\
& + \sum_{k=0}^{J_{m-1}-1} \prt{ \prt{g(\lambda_j,\Delta t_m)}^{2k} 
  - \prt{g(\lambda_j,\Delta t_{m-1})}^{k} } R_{m,J_m-2k+1}^{(j)} \\
  &+ \sum_{k=0}^{J_{m-1}-1}  \Bigg(
   \prt{ g(\lambda_j,\Delta t_{m}) }^{2k}  g(\lambda_j,\Delta t_m)  - \prt{g(\lambda_j,\Delta t_{m-1})}^{k} e^{-\lambda_j \Delta  t_{m} }  \Bigg)R_{m,J_m-2(k+1)}^{(j)} \\
& \eqcolon I_{m,j,1} + I_{m,j,2} + I_{m,j,3}.
\end{split}
\end{equation}
For bounding these three terms, we need to estimate the 
difference between powers of $\prt{g(\lambda_j,\Delta t_m)}^2$ and
$g(\lambda_j,\Delta t_{m-1})$. Note first that
\begin{equation}\label{eq:gSquaredEq}
\begin{split}
\prt{g(\lambda_j,\Delta t_m)}^2 &= e^{-2\lambda_j \Delta t_{m} } + 2 e^{- \lambda_j \Delta t_m} \frac{1-e^{-\lambda_j \Delta
      t_m}}{\lambda_j} + \prt{\frac{1-e^{-\lambda_j \Delta
        t_m}}{\lambda_j}}^2 \\
&= \underbrace{e^{-\lambda_j \Delta t_{m-1} } + \frac{1-e^{-\lambda_j \Delta
    t_{m-1}}}{\lambda_j}}_{=g(\lambda_j,\Delta t_{m-1})}  + (1-\lambda_j)   \prt{\frac{1-e^{-\lambda_j \Delta
      t_m}}{\lambda_j}}^2.
\end{split}
\end{equation}

\begin{remark}
  Equations~\eqref{eq:jthModeDiff} and~\eqref{eq:gSquaredEq} show that
  to leading order, the additive noise from two consecutive iterations
  of the fine scheme equals the additive noise from one
  corresponding iteration of the coupled coarse scheme. The
  strong coupling of the coarse and fine
  schemes is crucial for achieving the order 1 a priori time
  discretization convergence rate.
\end{remark}

Since $\inf_{j \in \bbN }\lambda_j = \lambda_1 >1$,
it holds for all $j\in \bbN$ that
\begin{equation}\label{eq:gBound1}
  \prt{g(\lambda_j,\Delta t_m)}^2 < g(\lambda_j,\Delta t_{m-1})<1
\end{equation}
 and
\[
\abs{\prt{g(\lambda_j,\Delta t_m)}^2 -g(\lambda_j,\Delta t_{m-1})} <
\prt{1 - e^{-\lambda_j \Delta t_m}} \frac{1-e^{-\lambda_j \Delta t_m}}{\lambda_j}
\le \prt{1 - e^{-\lambda_j \Delta t_m}} \Delta t_m.
\]
By the mean value theorem, it holds for any $j,k \ge 1$ that
\begin{equation}\label{eq:meanVal1}
\begin{split}
  & \abs{ \prt{g(\lambda_j,\Delta t_m)}^{2k} - \prt{ g(\lambda_j,\Delta t_{m-1}) }^k}
 \le \prt{ g(\lambda_j,\Delta t_{m-1})}^{k-1} k \prt{1 - e^{-\lambda_j \Delta t_m}} \Delta t_m.
\end{split}
\end{equation}
Furthermore,
\begin{equation}\label{eq:gBound2}
\sup_{\lambda  \ge \lambda_1} \lambda e^{-\lambda s}  \le
\frac{e^{-1}}{s}, \quad \text{for any} \quad s>0. 
\end{equation}
By~\eqref{eq:gBound1},~\eqref{eq:meanVal1},,~\eqref{eq:gBound2},
the mean value theorem and recalling that $\Delta t_{m-1} = 2 \Delta t_m$,
it holds for any $1<k \le J_{m-1}$ and $j\ge 1$
and some $\theta_{jk} \in [0,1]$ that 
\begin{equation}\label{eq:squareBound2}
\begin{split}
  & \abs{ \prt{g(\lambda_j,\Delta t_m)}^{2k} - \prt{ g(\lambda_j,\Delta t_{m-1}) }^k}
\le \prt{ e^{-\lambda_j \Delta t_{m-1} } + \frac{1-e^{-\lambda_j    \Delta
        t_{m-1}}}{\lambda_j} }^{k-1} k \lambda_j \Delta t_m^2 \\
&\le e^{-\lambda_j (k-1) \Delta t_{m-1}} k \lambda_j \Delta t_m^2\\
  & +\prt{ e^{-\lambda_j \Delta t_{m-1} } + \theta_{jk} \frac{1-e^{-\lambda_j    \Delta t_{m-1}}}{\lambda_j} }^{k-2} (k-1) k \Delta t_{m-1}  \Delta t_m^2\\
&  \le \frac{e^{-1}k}{(k-1)\Delta t_{m-1}} \Delta t_m^2 +
\frac{T^2}{2} \Delta t_{m}\\ 
&\le  \frac{1 + T^2}{2} \Delta t_{m}.
\end{split}
\end{equation}

From~\eqref{eq:squareBound2}, we conclude that for $j\le N_{m-1}$,
\begin{equation*}\label{eq:i1Bound}
\abs{I_{m,j,1}} \le \frac{1+T^2}{2}  |U_{\ell,0}^{(j)}|\Delta t_m.
\end{equation*}
For bounding the terms $I_{m,j,2}$ and $I_{m,j,3}$, note by~\eqref{eq:jthModeDiff}
that both terms are linear combinations of i.i.d. Gaussians
from the sequence 
\[
R^{(j)}_{m,k} \sim N\prt{0, \frac{1-e^{-\lambda_j \Delta t_{m}}}{2 \lambda_j^{1+2b}}}, \quad
k =0,1,\ldots, J_{m}-1,
\] 
cf.~\eqref{eq:rMDist}, and hence, both terms mean zero-valued Gaussians.
Furthermore, $I_{m,j,2}$ and $I_{m,j,3}$ are mutually independent as
any summand of the former term is independent of any summand 
from the latter.
Consequently, $I_{m,j,2}+I_{m,j,3}$ is a mean zero-valued Gaussian with variance
\[
\E{(I_{m,j,2}+I_{m,j,3})^2} = \E{I_{m,j,2}^2} + \E{I_{m,j,3}^2}.
\]
By the mutual independence of all terms in $I_{m,j,2}$, it
holds for $j \le N_{m-1}$ that
\begin{equation}\label{eq:i2}
\begin{split}
  \E{I_{m,j,2}^2} & = 
\sum_{k=0}^{J_{m-1}-1} \prt{ \prt{g(\lambda_j,\Delta t_m)}^{2k} 
  - \prt{g(\lambda_j,\Delta t_{m-1})}^{k} }^2
\E{\prt{R_{m,J_m-2k+1}^{(j)} }^2}\\
& \le 
\frac{1-e^{-\lambda_j \Delta t_{m}}}{2 \lambda_j^{1+2b}} \sum_{k=0}^{\infty} \prt{  \prt{g(\lambda_j,\Delta t_m)}^{2k} 
  - \prt{g(\lambda_j,\Delta t_{m-1})}^{k} }^2\\
&= \frac{1-e^{-\lambda_j \Delta t_{m}}}{2 \lambda_j^{1+2b}}
\sum_{k=0}^{\infty} \Bigg[ \prt{g(\lambda_j,\Delta t_m)}^{4k}
 +  \prt{g(\lambda_j,\Delta t_{m-1})}^{2k}  \\
& \qquad \qquad \qquad \qquad \qquad 
-2 \prt{  \prt{g(\lambda_j,\Delta t_m)}^2 g(\lambda_j,\Delta t_{m-1}) }^{k} \Bigg].
\end{split}
\end{equation}
By the strict inequality~\eqref{eq:gBound1}, we are dealing with three sums of
geometric series:
\begin{equation*}\label{eq:kSum1}
\sum_{k=0}^{\infty} \prt{  \prt{g(\lambda_j,\Delta t_m)}^2
  g(\lambda_j,\Delta t_{m-1}) }^{k}  
=\frac{1}{1-g(\lambda_j,\Delta t_{m-1})\prt{g(\lambda_j,\Delta t_{m})}^2 },
\end{equation*}
\begin{equation*}\label{eq:kSum2}
\begin{split}
\sum_{k=0}^{\infty}  \prt{g(\lambda_j,\Delta t_m)}^{4k} &=
  \frac{1}{1-\prt{g(\lambda_j,\Delta t_m)}^4},
\end{split}
\end{equation*}
and
\begin{equation*}\label{eq:kSum3}
\begin{split}
\sum_{k=0}^{\infty} \prt{g(\lambda_j,\Delta t_{m-1})}^{2k} &=
  \frac{1}{1-\prt{g(\lambda_j,\Delta t_{m-1})}^2 }\,\, .
\end{split}
\end{equation*}

By applying $g(\lambda_j,\Delta t_{m}) < g(\lambda_j,\Delta t_{m-1}) <1$
and the mean value theorem,
\[
\begin{split}
&\sum_{k=0}^{\infty}  \prt{g(\lambda_j,\Delta t_m)}^{4k} +
\prt{g(\lambda_j,\Delta t_{m-1})}^{2k}  = \frac{2-\prt{g(\lambda_j,\Delta t_{m})}^4 - \prt{g(\lambda_j,\Delta t_{m-1})}^2}{\prt{1-\prt{g(\lambda_j,\Delta t_{m})}^4}\prt{1-\prt{g(\lambda_j,\Delta t_{m-1})}^2}}\\
& = \frac{2\prt{1-\prt{g(\lambda_j,\Delta t_{m})}^2g(\lambda_j,\Delta t_{m-1})}
  - \prt{ \prt{g(\lambda_j,\Delta t_{m})}^2 - g(\lambda_j,\Delta t_{m-1}) }^2 }
{\prt{1-\prt{g(\lambda_j,\Delta t_{m})}^2g(\lambda_j,\Delta t_{m-1})}^2
  - \prt{ \prt{g(\lambda_j,\Delta t_{m})}^2 - g(\lambda_j,\Delta t_{m-1}) }^2} \\
& \le  \frac{2}{1-g(\lambda_j,\Delta t_{m-1})\prt{g(\lambda_j,\Delta t_{m})}^2 }
+ \frac{2}{\prt{1-g(\lambda_j,\Delta t_{m})}^3} \frac{\prt{1-e^{-\lambda_j \Delta t_m}}^4}{\lambda_j^2} \,
\end{split}
\]
where the second summand in the last inequality follows from~\eqref{eq:gSquaredEq}.
By~\eqref{eq:i2}, we obtain that for all $j \le N_{m-1}$,
\begin{equation*}\label{eq:i2Bound}
\begin{split}
 \E{I_{m,j,2}^2} &\le \frac{1}{\prt{1-g(\lambda_j,\Delta t_{m})}^3}
 \frac{\prt{1-e^{-\lambda_j \Delta t_m}}^5}{\lambda_j^{3+2b}} \\
& \le  \frac{\lambda_j^3}{(\lambda_j-1)^3
  \prt{1-e^{-\lambda_j \Delta t_m}}^3}  \frac{\prt{1-e^{-\lambda_j
      \Delta t_m}}^3}{\lambda_j^{1+2b}} \Delta t_m^2 \\
&\le  \prt{\frac{\lambda_1}{\lambda_1-1}}^3\frac{\Delta t_m^2}{\lambda_j^{1+2b}}.
\end{split}
\end{equation*}
The last term is bounded by a similar argument:
For all $j \le N_{m-1}$,
\begin{equation*}\label{eq:i3Bound}
\begin{split}
&\E{I_{m,j,3}^2}  = 
\sum_{k=0}^{J_{m-1}-1}  
 \Bigg(  \prt{ g(\lambda_j,\Delta t_{m}) }^{2k} \prt{e^{-\lambda_j
       \Delta t_m } + \frac{1-e^{-\lambda_j\Delta t_m}}{\lambda_j}}\\
& \qquad \qquad 
 -  \prt{g(\lambda_j,\Delta t_{m-1})}^{k} e^{-\lambda_j \Delta t_m} \Bigg)^2 \E{ \prt{R_{m,J_m-2(k+1)}^{(j)} }^2 }\\
&\le 
\frac{1-e^{\lambda_j\Delta t_m}}{\lambda_j^{1+2b}}\sum_{k=0}^{\infty}  \Bigg[
 \prt{  \prt{ g(\lambda_j,\Delta t_{m}) }^{2k}   -
   \prt{g(\lambda_j,\Delta t_{m-1})}^{k} }^2\\
& \qquad \qquad \qquad  \qquad \qquad \qquad + \prt{ g(\lambda_j,\Delta t_{m}) }^{2k}
\prt{\frac{1-e^{\lambda_j\Delta t_m}}{\lambda_j}}^2 \Bigg] \\
& \le 3\prt{\frac{\lambda_1}{\lambda_1-1}}^3 \frac{\Delta t_m^2 }{\lambda_j^{1+2b}}. 
\end{split}
\end{equation*}
Here, the last inequality follows by observing that
as for $\E{I_{m,j,2}^2}$, 
\[
\frac{1-e^{\lambda_j\Delta t_m}}{\lambda_j^{1+2b}} \sum_{k=0}^{\infty} 
 \prt{  \prt{ g(\lambda_j,\Delta t_{m}) }^{2k}   -
   \prt{g(\lambda_j,\Delta t_{m-1})}^{k} }^2
\le 2\prt{\frac{\lambda_1}{\lambda_1-1}}^3 \frac{\Delta t^2_m}{\lambda_j^{1+2b}},
\]
and
\[
\begin{split}
\frac{\prt{1-e^{\lambda_j\Delta t_m}}^3}{\lambda_j^{3+2b}} \sum_{k=0}^{\infty}   \prt{ g(\lambda_j,\Delta t_{m}) }^{2k}   
&\le \frac{(1-e^{\lambda_j \Delta t_m})}{\lambda_j^{1+2b}}
\frac{1}{1-g(\lambda_j,\Delta t_m)}\Delta t^2_m\\
&= \frac{(1-e^{\lambda_j \Delta t_m})}{\lambda_j^{1+2b}}
\frac{\lambda_j}{(\lambda_j-1)(1-e^{\lambda_j \Delta t_m})}\Delta t^2_m\\
& \le \frac{\lambda_1}{\lambda_1-1} \frac{\Delta
  t^2_m}{\lambda_j^{1+2b}}\le \prt{\frac{\lambda_1}{\lambda_1-1}}^3 \frac{\Delta t^2_m}{\lambda_j^{1+2b}}.
\end{split}
\]

\end{proof}

\section{Additional proofs for completeness}
\label{app:extras}

\begin{proof}[Proof of Lemma \ref{lem:gce}]
Recalling the notation $R_n^{\rm ML} = C^{\rm ML}_nH^*$ and introducing
  the auxiliary operator $\bar R_n \coloneq \bar C_n H^*$, we have
  \[
  \begin{split}
    \mfk_n - K^{\rm ML}_n
    &= \bar R_n(H\bar R_n +\Gamma)^{-1} - R^{\rm ML}_{n}( (H R_n^{\rm ML})^+ +\Gamma)^{-1}\\
    &= \bar R_n((H\bar R_n +\Gamma)^{-1} - ((H R_n^{\rm ML})^+ +\Gamma)^{-1})\\
    & \quad + (\mfc_n - C^{\rm ML}_n)H^*( (H R^{\rm ML}_n)^+ + \Gamma)^{-1}.
  \end{split}
\]
Using the equality
\[
(H\bar R_n +\Gamma)^{-1} - ((H R_n^{\rm ML})^+ +\Gamma)^{-1}
 = (H\bar R_n +\Gamma)^{-1}( (HR_n^{\rm ML})^+ - H \bar R_n) ( (H R^{\rm ML}_n)^+ + \Gamma)^{-1},
\]
we further obtain
  \[
  \begin{split}
    \mfk_n - K^{\rm ML}_n
    &= \bar R_n(H\bar R_n +\Gamma)^{-1}( (H R_n^{\rm ML})^+ - H \bar R_n) ((H R^{\rm ML}_n)^+ + \Gamma)^{-1}\\
    & \quad + (\mfc_n - C^{\rm ML}_n)H^*( (H R^{\rm ML}_n)^+ + \Gamma)^{-1}\\
    & = \mfk_n ( (HR_n^{\rm ML})^+ -  H\bar R_n) ((HR^{\rm ML}_n)^+ +
    \Gamma)^{-1} \\
    & \quad +  (\mfc_n - C^{\rm ML}_n)H^*((HR^{\rm ML}_n)^+ + \Gamma)^{-1}.
  \end{split}
\]
Next, since $(H R^{\rm ML}_{n})^+$ and $\Gamma$
respectively are positive semi-definite and positive definite, 
\[
|( (H R^{\rm ML}_{n})^+  + \Gamma)^{-1}| \leq
| \Gamma^{-1}|  < \infty,
\]
and it follows by inequality~\eqref{eq:mlcov_bound}
and 
\[
\begin{split}
| (H R_n^{\rm ML})^+ -  H \bar R_n|
&\leq | (HR_n^{\rm ML})^+ -  HR_n^{\rm ML} |  
 + | H(R_n^{\rm ML} -   \bar R_n) |\\
& = |(HR_n^{\rm ML})^+ -  HR_n^{\rm ML}| 
 + |H(C_n^{\rm ML} -  \mfc_n ) H^*|
\end{split}
\]
that 
\begin{equation*}\label{eq:kayer1}
  \begin{split}
  \norm{\mfk_n - \kml{n}}_{L(\bbR^{m},\cH)} &\leq \prt{1+ 2\norm{\mfk_n}_{L(\bbR^{m},\cH)} \norm{H}_{L(\cH,\bbR^{m})}}\\
  & \times |\Gamma^{-1}|  \norm{H}_{L(\cH,\bbR^{m})}
  \underbrace{\norm{ \mfc_{n} - {C}^{\rm ML}_{n}}_{L(\cH^*,\cH)}}_{\le \norm{ \mfc_{n} - {C}^{\rm ML}_{n}}_{\cH \otimes \cH}}.
  \end{split}
\end{equation*}

\end{proof}

\begin{proof}[Proof of Lemma \ref{lem:ensdist}]

  We will use an induction argument to show that for arbitrary fixed
  $N \in \bbN$ and $p \ge 2$, it holds for all $n \le N$ that
  \[
  \sum_{\ell=0}^L  \|\hv_{n}^{\ell}- \hmfv_{n}^{\ell}\|_{L^{p'}(\Omega, \cH)}
  \lesssim |\log(\varepsilon)|^n\varepsilon, \quad \forall p' \le 4^{N-n}p.
  \]
  The result then follows by the arbitrariness of $N$ and $p$.
  
By~\eqref{eq:initialMFEnKF}, we have that $\hv_{0}^{\ell} = \hmfv_{0}^{\ell}$, so that for any $p' \ge 2$,
\[
\sum_{\ell=0}^L \|\hv_{0}^{\ell}- \hmfv_{0}^{\ell}\|_{p'} = 0.
\]
Fix $p\geq 2$ and $N \in \bbN$, and assume that
\begin{equation*}
  \sum_{\ell=0}^L \norm{\hv_{n-1}^{\ell}- \hmfv_{n-1}^{\ell}}_{p'} \lesssim |\log(\varepsilon)|^{n-1}\varepsilon,
  \quad \forall p' \le 4^{N+1-n} p.
\label{eq:ind1ens}
\end{equation*}
Then, by Assumption~\ref{ass:psilip}(i),
\begin{equation}
\sum_{\ell=0}^L \|v_n^{\ell} - \mfv_n^\ell\|_{p'} 
\leq  \sum_{\ell=0}^L c_{\Psi} \|\hv_{n-1}^{\ell} - \hmfv_{n-1}^\ell\|_{p'} \lesssim |\log(\varepsilon)|^{n-1}\varepsilon, \quad \forall p' \le 4^{N+1-n} p.
\label{eq:indlip1ens}
\end{equation}
Furthermore, by Lemma~\ref{lem:gce}, 
\begin{equation*}\label{eq:appearanceOfCn}
  \begin{split}
\hNorm{\hv_n^{\ell} - \hmfv_n^\ell}  &\leq   
\norm{I - \Pi_\ell \bar K_nH}_{L(\cH,\cH)} \hNorm{v_n^{\ell} - \mfv_n^\ell}\\
& \quad +  \tilde c_n  \hhNorm{C^{\rm ML}_n - \mfc_n}|\tilde y_n^{\ell}- H v_n^\ell|,
\end{split}
\end{equation*}
for all $\ell=0, \ldots, L$.
H{\"o}lder's inequality then implies
\begin{equation}\label{eq:levelLBound}
\begin{split}
\|\hv_n^{\ell} - \hmfv_n^\ell\|_{p'}  &\leq   
 \norm{I - \Pi_\ell\bar K_nH}_{L(\cH,\cH)} \|v_n^{\ell} - \mfv_n^\ell\|_{p'}\\
 & \quad +  \tilde c_n  \|{C}^{\rm ML}_n - \mfc_n\|_{L^{2p'}(\Omega, \cH\otimes \cH)}
 (\|\tilde y_n^{\ell}\|_{2p'} + \norm{H}_{L(\cH,\bbR^{m})}\|v_n^\ell\|_{2p'}).  
\end{split}
\end{equation}
Plugging~\eqref{eq:indlip1ens} 
into the right-hand side of~\eqref{eq:mlcov}
and using Lemma~\ref{lem:covspliteps}, we obtain
that for all $p' \le 4^{N-n} p$.
\[
\begin{split}
\|C^{\rm ML}_n - \bar{C}^{\rm ML}_n\|_{2p'} &\lesssim \varepsilon +
\sum_{l=0}^L \|v_{n}^{\ell} - \mfv_{n}^\ell\|_{4p'} (\|v_{n}^{\ell}\|_{4p'}  + \|\mfv_{n}^{\ell}\|_{4p'}) \\
& \lesssim |\log(\varepsilon)|^{n-1}\varepsilon.
\end{split}
\]
Summing over the levels in~\eqref{eq:levelLBound}, 
it holds for all $p' \le 4^{N-n} p$ that 
\begin{align*}
\nonumber
 \sum_{\ell=0}^L \|\hv_n^{\ell} - \hmfv_n^\ell\|_{p'} 
& \lesssim 
\sum_{\ell=0}^L  \Big\{ \|v_n^{\ell} - \mfv_n^\ell\|_{p'}  
 +  |\log(\varepsilon)|^{n-1}\varepsilon  (\|\tilde
 y_n^{\ell}\|_{2p'} + \norm{H}_{L(\cH,\bbR^{m})}\|v_n^\ell\|_{2p'}) \Big\} \\
& \lesssim |\log(\varepsilon)|^{n-1} \varepsilon \Big(1+
  \sum_{\ell=0}^L
 (\|\tilde
y_n^{\ell}\|_{2p'} + \norm{H}_{L(\cH,\bbR^{m})}\|v_n^\ell\|_{2p'})\Big)\\
& \lesssim |\log(\varepsilon)|^{n}\varepsilon.
\end{align*}

\end{proof}

\bigskip

{\bf Acknowledgements } Research reported in this publication received
support from the Alexander von Humboldt Foundation, KAUST CRG4 Award
Ref:2584.  HH acknowledges support by RWTH Aachen University and
by Norges Forskningsr{\aa}d, research project 214495 LIQCRY.  RT is a member of the KAUST SRI Center for
Uncertainty Quantification in Computational Science and Engineering.
KJHL was a staff scientist in the Computer Science and Mathematics
Division at Oak Ridge National Laboratory (ORNL) while much of this
research was done and was additionally supported by ORNL Laboratory
Directed Research and Development Strategic Hire and Seed grants.
KJHL additionally acknowledges the support of the School of
Mathematics at the University of Manchester.

\bibliography{mybib}
\bibliographystyle{siam}

\end{document}